\documentclass{amsart}
\numberwithin{equation}{section}
\usepackage[english]{babel}
\usepackage[T1]{fontenc}
\usepackage[latin1]{inputenc}
\usepackage{indentfirst}
\usepackage{enumitem}
\usepackage{amsmath,amssymb, amsbsy}
\usepackage{amsfonts}
\usepackage{hyperref}
\usepackage{latexsym}
\usepackage{amsthm}
\usepackage[dvips]{graphicx}
\DeclareGraphicsExtensions{.pdf,.png,.jpg,.eps}

\usepackage{verbatim}

\usepackage[usenames,dvipsnames]{xcolor}
\usepackage{color}
\usepackage[latin1]{inputenc}
\usepackage[active]{srcltx}
\definecolor{Arancio}{cmyk}{0,0.61,0.87,0}

\newcommand{\brd}[1]{\mathbb{#1}}
\newcommand{\R}{\brd{R}}
\newcommand{\abs}[1]{\left\lvert {#1} \right\rvert}
\newcommand{\norm}[2]{\left\Vert {#1} \right\Vert_{#2}}
\newcommand{\loc}{{{\tiny{\mbox{loc}}}}}
\newcommand\ddfrac[2]{\frac{\displaystyle #1}{\displaystyle #2}}
\newtheorem{teo}{Theorem}[section]
\newtheorem{Corollary}[teo]{Corollary}
\newtheorem{Lemma}[teo]{Lemma}
\newtheorem{Theorem}[teo]{Theorem}
\newtheorem{Proposition}[teo]{Proposition}
\theoremstyle{definition}
\newtheorem{Definition}[teo]{Definition}
\newtheorem{remarks}[teo]{Remarks}
\newtheorem{remark}[teo]{Remark}

\newtheorem{Assumption}[teo]{Assumption}
\newtheorem{Example}[teo]{Example}

\begin{document}

\title[Regularity for even solutions to degenerate or singular problems]
{Liouville type theorems and regularity of solutions to degenerate or singular problems part I: even solutions}
\date{\today}

\author{Yannick Sire, Susanna Terracini and Stefano Vita}

\address[Y. Sire]{Department of Mathematics
	\newline\indent
	Johns Hopkins University
	\newline\indent
	3400 N. Charles Street, Baltimore, MD 21218, U.S.A.}
\email{sire@math.jhu.edu}

\address[S. Terracini]{Dipartimento di Matematica G. Peano
	\newline\indent
	Universit\`a degli Studi di Torino
	\newline\indent
	 Via Carlo Alberto 10, 10123 Torino, Italy}
\email{susanna.terracini@unito.it}

\address[S. Vita]{Dipartimento di Matematica
	\newline\indent
	Universit\`a degli Studi di Milano Bicocca
	\newline\indent
	 Piazza dell'Ateneo Nuovo 1, 20126, Milano, Italy}
\email{stefano.vita@unimib.it}

\date{\today} 
\thanks{Work partially supported by the ERC Advanced Grant 2013 n.~339958 Complex Patterns for Strongly Interacting Dynamical Systems - COMPAT. Y.S. is partially supported by the Simons foundation. This work was initiated while S. V. was visiting Johns Hopkins University. He would like to thank the mathematics department for its hospitality. We would like to thank Gabriele Cora for some fruitful discussions and suggestions}


\maketitle

\begin{abstract}
We consider a class of equations in divergence form with a singular/degenerate weight \[
-\mathrm{div}(|y|^a A(x,y)\nabla u)=|y|^a f(x,y)\; \quad\textrm{or} \ \textrm{div}(|y|^aF(x,y))\;.
\]
Under suitable regularity assumptions for the matrix $A$ and $f$  (resp. $F$) we prove  H\"older continuity of solutions which are even in $y\in\mathbb{R}$, and possibly of their derivatives up to order two or more (Schauder estimates).  In addition, we show stability of the $C^{0,\alpha}$ and $C^{1,\alpha}$ a priori bounds for approximating  problems in the form
\[
-\mathrm{div}((\varepsilon^2+y^2)^{a/2} A(x,y)\nabla u)=(\varepsilon^2+y^2)^{a/2} f(x,y)\;  \quad\textrm{or} \ \textrm{div}((\varepsilon^2+y^2)^{a/2}F(x,y))
\]
as $\varepsilon\to 0$. Finally, we derive  $C^{0,\alpha}$ and $C^{1,\alpha}$ bounds for inhomogenous Neumann boundary problems as well. Our method is based upon blow-up and appropriate Liouville type theorems.
\end{abstract}

\section{Introduction and main results}
Let $z=(x,y)\in\Omega\subset\mathbb{R}^{n+1}$, with $x\in\mathbb{R}^{n}$ and $y\in\mathbb{R}$, $n\geq1$, $a\in\mathbb{R}$. We are concerned with qualitative properties of solutions to a class of problems involving the operator in divergence form given by
$$\mathcal L_au:=\mathrm{div}(|y|^a A(x,y)\nabla u)\;,$$
where the matrix $A$ is symmetric, continuous and satisfies the uniform ellipticity condition $\lambda_1|\xi |^2 \leq A(x,y)\xi\cdot\xi \leq \lambda_2|\xi |^2$, for all $\xi\in\R^{n+1}$, for every $(x,y)\in\Omega$ and some ellipticity constants $0<\lambda_1\leq\lambda_2$. Such class of elliptic operators  arises in the study of fractional powers of elliptic operators as well as in applications to physics, ecology and biological sciences.

We denote by $\Sigma:=\{y=0\}\subset\R^{n+1}$ the characteristic manifold that we assume to be invariant with respect to $A$. Operators on our class may be degenerate or singular, in the sense that the coefficients of the differential operator may vanish or be infinite over $\Sigma$, and this happens respectively when $a>0$ and $a<0$. Such  behaviour affects the regularity of solutions: indeed $u(x,y)=|y|^{-a}y$ is $\mathcal L_a$ harmonic, when $A\equiv \mathbb{I}$ and lacks of smoothness whenever $a$ is not an integer. 

We recall that a function $w \in L^1_\loc (\mathbb R^{n+1})$ is said an $A_p$ weight  if the following holds 
$$
\sup_{B \subset \mathbb R^{n+1}}  \Big ( \frac{1}{|B|} \int_B w \Big ) \Big ( \frac{1}{|B|} \int_B w^{-1/(p-1)} \Big )^{p-1} <\infty 
$$
whenever $B$ is a ball. The weights $\rho(y)=|y|^a$ belong to the Muckenhoupt class $A_p$ when $a\in(-1,p-1)$.  The theory of weights is a crucial tool and central area in harmonic analysis. The seminal series of papers  \cite{FabKenSer,FabJerKen1,FabJerKen2,JerKen} develops in the framework of elliptic equations in divergence form, a basic regularity theory for such equations involving $A_2$ weights. Nevertheless, as already noticed in \cite{FabKenSer}, this class of weights may fail to be optimal to obtain an equivalent of the classical De Giorgi-Nash-Moser theory and the present paper deals also with a larger class of operators. 

This paper is devoted to the study of fine qualitative properties of suitably defined solutions of equations of the type $\mathcal L_a u=r.h.s$ (for a wide class of right hand sides) and is a predecessor to the paper of two of the authors and Tortone \cite{SirTerTor} on the study of the nodal set of such solutions, with an extension to the parabolic case in \cite{AudTer}. The interest in this specific type of operators arose in the last decade as they appear  in connection with the fractional powers of the laplacian (see the influential work \cite{CafSil1}). In this framework, the auxiliary variable $y$ is defined on $\mathbb R_+$ and is actually the distance to the characteristic manifold $\Sigma$. Since our results are local, possibly after a change of coordinates in the (sufficiently smooth) domain where the equation is defined or for instance after taking blow-ups of solutions,  they apply to equations in various domains of $\mathbb R^{n+1}$ or smooth manifolds to name.  Contrary to the extensive literature on uniformly elliptic equations, the fine regularity theory of degenerate/singular equations is at the very beginning. The present paper is a first contribution towards a systematic study of these operators. 

This paper is also the first part of a two part series of papers dealing with specific solutions, namely even and odd solutions. The reason why we consider those two separate cases will become clear later but let us say right away that in the present case of even solutions one can provide a rather complete Schauder theory for suitably defined solutions, while we can not expect regularity for general solutions (see Example \ref{ex1}). Along the way, we prove several useful Liouville-type theorems which are somehow of independent interest. The operator being linear, it is very natural to consider even solutions and odd solutions. We will see that, depending on the symmetry of the solutions, several nice regularity properties can be proved. In order to give a flavor of our results, it is worth first stating our Schauder estimates:

\begin{Theorem}\label{cinf}
Let $a\in(-1,+\infty)$, $k\in\mathbb{N}\cup\{0\}$ and $f\in C^{k,\alpha}(B_1)$ for $\alpha\in(0,1)$ and even in $y$. Let also $u\in H^{1,a}(B_1)$ be an even energy solution to (see the next section for the precise definition)
\begin{equation*}
-\mathrm{div}(|y|^a\nabla u)=|y|^af\qquad\mathrm{in \ }B_1.
\end{equation*}
Then, $u\in C^{k+2,\alpha}_{\mathrm{loc}}(B_1)$. If moreover $f\in C^{\infty}(B_1)$, then, $u\in C^{\infty}(B_1)$.
\end{Theorem}

This  result is somewhat surprising, in view of the lack of regularity of the coefficients of the operators and can be attributed to the joint regularising effect of the equation and the Neumann boundary condition in the half ball, associated with evenness. We stress that odd solutions may indeed lack regularity, as shown by the example $u(x,y)=|y|^{-a}y$ which solves $-\mathrm{div}(|y|^a\nabla u)=0$ whenever $a<1$.   

We are going to follow a perturbative method,  actually allowing us to deal with more general equations with right hands in possibly divergence form, and to deal with an entire class of regularised problems in the form:

\begin{equation*}
-\mathrm{div}((\varepsilon^2+y^2)^{a/2} A(x,y)\nabla u)=
\begin{cases}\textrm{either} \; \;&\phantom{divw}(\varepsilon^2+y^2)^{a/2}f(x,y)\;,\\
\textrm{or} \qquad&\mathrm{div}\left((\varepsilon^2+y^2)^{a/2}F(x,y)\right)\;,
\end{cases}
\end{equation*}

and derive both $C^{0,\alpha}$ and $C^{1,\alpha}$ estimates which are uniform with respect to the parameter $\varepsilon\geq 0$ (we shall refer to this fact as a {\sl $\varepsilon$-stable property}). We state below the main results related to this fact.

\begin{Theorem}
Let $a\in(-1,+\infty)$ and let $u\in H^{1,a}(B_1)$ be an even in $y$ energy solution to
\begin{equation*}
-\mathcal L_au=|y|^af\qquad\mathrm{in \ }B_1,
\end{equation*}
with $f\in L^p(B_1,|y|^a\mathrm{d}z)$. Then
\begin{itemize}
\item[i)]  If $A$ is continuous, $\alpha\in(0,1)\cap(0,2-\frac{n+1+a^+}{p}]$, $p>\frac{n+1+a^+}{2}$, $\beta>1$ and $r\in(0,1)$ one has: there exists a constant $c>0$ such that
$$\|u\|_{C^{0,\alpha}(B_r)}\leq c\left(\|u\|_{L^\beta(B_1,|y|^a\mathrm{d}z)}+\|f\|_{L^p(B_1,|y|^a\mathrm{d}z)}\right).$$
\item[ii)] If  $A$ is $\alpha$-H\"older continuous, $\alpha\in(0,1-\frac{n+1+a^+}{p}]$, $p>n+1+a^+$, $\beta>1$ and $r\in(0,1)$ one has: there exists a constant $c>0$ such that
$$\|u\|_{C^{1,\alpha}(B_r)}\leq c\left(\|u\|_{L^\beta(B_1,|y|^a\mathrm{d}z)}+\|f\|_{L^p(B_1,|y|^a\mathrm{d}z)}\right).$$
\end{itemize}
Moreover, both these properties are $\varepsilon$-stable.
\end{Theorem}

As far as right hand sides in divergence form are concerned we have

\begin{Theorem}
Let $a\in(-1,+\infty)$ and let $u\in H^{1,a}(B_1)$ be an even in $y$ energy solution to
\begin{equation*}
-\mathcal L_au=\mathrm{div}\left(|y|^aF\right)\qquad\mathrm{in \ }B_1,
\end{equation*}
with $F=(f_1,...,f_{n+1})$. Then
\begin{itemize}
\item[i)] If $A$ is continuous, $\alpha\in(0,1-\frac{n+1+a^+}{p}]$, $F\in L^p(B_1,|y|^a\mathrm{d}z)$ with $p>n+1+a^+$, $\beta>1$ and $r\in(0,1)$ one has: there exists a constant $c>0$ such that
$$\|u\|_{C^{0,\alpha}(B_r)}\leq c\left(\|u\|_{L^\beta(B_1,|y|^a\mathrm{d}z)}+\|F\|_{L^p(B_1,|y|^a\mathrm{d}z)}\right).$$
\item[ii)] If  $A$ is $\alpha$-H\"older continuous and $F\in C^{0,\alpha}(B_1)$ with $\alpha\in(0,1)$, $f_{n+1}(x,0)=f^y(x,0)=0$, $\beta>1$ and $r\in(0,1)$ one has that there exists a constant $c>0$ such that
$$\|u\|_{C^{1,\alpha}(B_r)}\leq c\left(\|u\|_{L^\beta(B_1,|y|^a\mathrm{d}z)}+\|F\|_{C^{0,\alpha}(B_1)}\right).$$
\end{itemize}
Moreover, both these estimates are $\varepsilon$-stable.
\end{Theorem}

These Theorems tell us that, even though the operator is degenerate or singular, one recovers the standard regularity result holding for uniformly operators with smooth coefficients, when even-in-$y$ solutions are considered. Our strategy goes as follows:

\begin{enumerate}
\item as already said, we first regularize the problem by introducing a parameter $\varepsilon$ such that the operator becomes uniformly elliptic when $\varepsilon >0$;
\item by means of appropriate Liouville-type theorems, which may be of independent interest,  we then obtain uniform estimates in $\varepsilon \geq 0$ in H\"older spaces $C^{0,\alpha}$ and $C^{1,\alpha}$ for even solutions in $y$. This is the main part of the paper and relies heavily on some spectral properties;
\item we prove that all solution to the singular/degenerate equation can be obtain as limits of solutions to a sequence of regularized problems;
\item to provide higher regularity in the case $A=\mathbb I$, we use the structure of the operator $\mathcal L_a$, the evenness of the solutions and algebraic manipulations. 
\end{enumerate}

It has to be noticed that the H\"older continuity of the solutions (for much more general weights) was already proved in \cite{FabKenSer} in the $A_2$ Muckenhoupt case ($a\in(-1,1)$ here) and in the case of quasiconformal weights (that is, weights appearing after performing a quasiconformal map on uniformly elliptic equations), without the optimal exponent and not in an $\varepsilon$-stable form. Of course, the optimal regularity is strongly related with the homogeneity of the weight $|y|^a$. In this aspect, the present work is related to a paper by the first author with Lamboley and Teixeira where they investigate a free boundary problem with a singular weight \cite{LST} (see also \cite{AllLinPet,BraLinSch}).

Finally, we point out that we are able to solve completely the problem in the energy space for all values $a \in (-1,\infty)$ and this range is wider than that for which the weight is $A_2$, i.e. $(-1,1)$.
Moreover, in the super degenerate range $a\geq1$, the evenness assumption can not be removed for H\"older regularity (also for continuity). In fact, we have the following counterexample:
\begin{Example}\label{ex1}
When $a\geq1$, the jump function
$$\overline u(z)=\begin{cases}
\phantom{-}1 & \mathrm{in} \ B_1^+\\
-1 & \mathrm{in} \ B_1^-,
\end{cases}$$
is an energy (not even) $\mathcal L_a$-harmonic function. Even more, replacing the constant $1$ (say) in $B_1^+$ by $0$, one produces also an energy $\mathcal L_a$-harmonic function for which the unique continuation principle does not hold.
\end{Example}

$\mathcal L_a$-harmonic functions have been widely studied in connection with fractional Laplacians, in view of the well known realization of $(-\Delta)^s$ as a Dirichlet-to-Neumann operator, when $s\in(0,1)$, \cite{CafSil1}. Following  this approach, $H^s$-functions over $\R^n$ are uniquely extended in $\R^{n+1}_+$ by convolution
with the Poisson kernel of $\mathcal L_a$, where $a=1-2s\in(-1,1)$. Thus, in order to study the equation $(-\Delta)^s\cdot=f$, one is lead to  deal with global finite energy solutions to
\begin{equation*}
\begin{cases}
-\mathcal L_au=0\; \quad &y>0\\
-\lim_{y\to 0}y^a\partial_y u(x,y)=f(x)\; &y=0\;.
\end{cases}
\end{equation*}

In this light, we are naturally  lead to extend our analysis to inhomogeneous Neumann boundary
value problems associated with $\mathcal L_a$. Again, we shall mainly (though not exclusively) seek $\varepsilon$-stable estimates. In this perspective, in Section \ref{sec:neumann}, we shall prove the
following estimates:

\begin{Theorem}
Let $a\in(-1,1)$ and let $u\in H^{1,a}(B_1^+)$ be an energy solution to 
\begin{equation*}
\begin{cases}\label{pb:neumann_inhom}
-\mathcal L_au=0\; \quad &\mathrm{in}\; B^+_1\\
-\lim_{y\to 0}y^a\partial_y u(x,y)=f(x)\; &\mathrm{on}\; \partial ^0B^+_1\;,
\end{cases}
\end{equation*}
with $f\in L^p(\partial ^0B_1^+)$. Then, if $A$ is continuous, $p>\frac{n}{1-a}$, $\alpha\in(0,1-a-\frac{n}{p}]\cap(0,1)$, $\beta>1$ and $r\in(0,1)$ one has: there exists a constant $c>0$ such that
$$\|u\|_{C^{0,\alpha}(B_r)}\leq c\left(\|u\|_{L^\beta(B_1,|y|^a\mathrm{d}z)}+\|f\|_{L^p(\partial ^0B_1^+)}\right).$$
Moreover, if $p>\frac{n}{1-a^+}$ and $\alpha\in(0,1-a^+-\frac{n}{p}]\cap(0,1)$ this estimate is $\varepsilon$-stable.
\end{Theorem}
In order to prove the $C^{1,\alpha}$ estimates, we have to restrict ourselves to the cases $a<0$.

\begin{Theorem}
Let $a\in(-1,0)$ and let $u\in H^{1,a}(B_1^+)$ be an energy solution to 
\begin{equation*}
\begin{cases}\label{pb:neumann_inhom}
-\mathcal L_au=0\; \quad &\textrm{in}\; B^+_1\\
-\lim_{y\to 0}y^a\partial_y u(x,y)=f(x)\;, &\textrm{on}\; \partial ^0B^+_1\;.
\end{cases}
\end{equation*}
Assume $A$ is $\alpha$-H\"older continuous. Then, 
\begin{itemize}
\item[i)]   if  $f\in L^p(\partial ^0B_1^+)$ with  $p>\frac{n}{-a}$, $\alpha\in(0,-a-\frac{n}{p}]$, $\beta>1$ and $r\in(0,1)$ one has: there exists a constant $c>0$ such that
$$\|u\|_{C^{1,\alpha}(B_r^+)}\leq c\left(\|u\|_{L^\beta(B_1^+,y^a\mathrm{d}z)}+\|f\|_{L^p(\partial ^0B_1^+)}\right).$$
\item[ii)] If $f\in C^{0,\alpha}(\partial ^0B_1^+)$ with $\alpha\in(0,-a]$, $r\in(0,1)$ and $\beta>1$ one has: there exists a constant $c>0$ such that
$$\|u\|_{C^{1,\alpha}(B_r^+)}\leq c\left(\|u\|_{L^\beta(B_1^+,\rho_\varepsilon^a(y)\mathrm{d}z)}+ \|f\|_{C^{0,\alpha}(\partial^0 B_1^+)}\right).$$
Moreover, this property is $\varepsilon$-stable.
\end{itemize}
\end{Theorem}

These results should be compared with the known
 $C^{0,\alpha}$ and $C^{1,\alpha}$  local estimates for solutions to inhomogeneous fractional Laplace equations by Silvestre, Caffarelli-Stinga, and other authors (\cite{CafSti, Sil, BucKar,
DonKim}). Their method is essentially based on singular integrals involving Riesz potentials. It is worthwhile noticing, however that we take a different perspective; first of all we seek for regularity in all the $n+1$ variables, while the quoted papers deal with the regularity in the $x$-variable only. The presence of the special solution $y^{1-a}$ gives a necessary bound $\alpha\leq -a$ in the $C^{1,\alpha}$ estimate. Moreover, our results apply to the whole family of regularizing weights, with constants which are uniform in $\varepsilon$.
Eventually we remark that our Theorem \ref{cinf} is providing local $C^\infty$-regularity for extensions of $s$-harmonic functions (homogeneous Neumann boundary condition) in any variable (also the extension variable $y$) up to the characteristic manifold $\Sigma$.\\

\noindent{\bf Notations.}\\

\begin{tabular}{ll}
$\R^{n+1}_+=\R^n\times(0,+\infty)$ & $z=(x,y)$ with $x\in\R^n$, $y>0$ \\
$\Sigma=\{y=0\}$ & characteristic manifold \\
$B_r^+=B_r\cap\{y>0\}$ & half ball \\
$\partial^+B_r^+=S^n_+(r)=\partial B_r\cap\{y>0\}$ & upper boundary of the half ball \\
$\partial^0B_r^+=B_r\cap\{y=0\}$ & flat boundary of the half ball \\
$\rho_\varepsilon^a(y)=\left(\varepsilon^2+y^2\right)^{a/2}$ & regularized weight \\
$\mathcal L_{\rho_\varepsilon^a}u=\mathrm{div}\left(\rho_\varepsilon^a(y)A(x,y)\nabla u\right)$ & regularized operator \\
$H^{1}(\Omega,\rho_\varepsilon^a(y)\mathrm{d}z)$ & weighted Sobolev space given by the completion of $C^\infty(\overline\Omega)$ \\
$H^{1}_0(\Omega,\rho_\varepsilon^a(y)\mathrm{d}z)$ & weighted Sobolev space given by the completion of $C^\infty_c(\Omega)$ \\
$\tilde H^{1}(\Omega,\rho_\varepsilon^a(y)\mathrm{d}z)$ & weighted Sobolev space given by the completion of $C^\infty_c(\overline\Omega\setminus\Sigma)$ \\
$H^{1,a}(\Omega)=H^{1}(\Omega,|y|^a\mathrm{d}z)$ & weighted Sobolev space for $\varepsilon=0$ \\
$\partial_y^au=|y|^a\partial_yu$ & "weighted" derivative \\
$\mathcal F_au=\partial_y^{-a}\partial_y^au$ & \\
$\mathcal Gu=y^{-1}\partial_yu$ & \\
$a^+=\max\{a,0\}$ &\\
\end{tabular}

\tableofcontents


\section{Functional setting and preliminary results}\label{secs2}

\subsection{Regularized operators for approximation}
Let $\Omega\subset\mathbb{R}^{n+1}$ be non empty, open and bounded. In order to better understand the regularity of solutions to degenerate and singular problems involving the operator $\mathcal L_a$, we introduce a family of regularized operators. For $a\in\mathbb{R}$ fixed, let us consider the family in $\varepsilon\geq0$ of weights $\rho_\varepsilon^a(y):\Omega\to\mathbb{R}_+$ defined as
\begin{equation*}\label{rho}
\rho_\varepsilon^a(y):=\begin{cases}
(\varepsilon^2+y^2)^{a/2}\min\{\varepsilon^{-a},1\} &\mathrm{if \ }a\geq0,\\
(\varepsilon^2+y^2)^{a/2}\max\{\varepsilon^{-a},1\} &\mathrm{if \ }a\leq0,
\end{cases}
\end{equation*}
and that of the associated operators
\begin{equation*}
\mathcal L_{\rho_\varepsilon^a}u=\mathrm{div}\left(\rho_\varepsilon^a(y)A(x,y)\nabla u\right).
\end{equation*}
Obviously, the family $\{\rho^a_\varepsilon\}_{\varepsilon}$ satisfies the following properties:
\begin{itemize}
\item[1)] $\rho_\varepsilon^a(y)\to|y|^a$ as $\varepsilon\to0^+$ almost everywhere in $\Omega$,
\item[2)] $\rho_\varepsilon^a(y)=\rho_\varepsilon^a(-y)$,
\item[3)] for any $\varepsilon>0$, the operator $-\mathcal L_{\rho_\varepsilon^a}$ is uniformly elliptic.
\end{itemize}
Now we set the minimal assumptions on the matrix $A$ that we need through the paper
\begin{Assumption}[HA]\label{(HA)}
The matrix $A=(a_{ij})$ is $(n+1,n+1)$-dimensional and symmetric $A=A^T$, has the following symmetry with respect to $\Sigma$: 
we have
\begin{equation*}
A(x,y)=JA(x,-y)J,\qquad\qquad\mathrm{with}\qquad\qquad
J
=\left(
\begin{array}{c|c}
\mathbb I_n & 0 \\
\hline
0 & -1
\end{array}
\right).
\end{equation*}
Therefore, $A$ is continuous and satisfies the uniform ellipticity condition $\lambda_1|\xi |^2 \leq A(x,y)\xi\cdot\xi \leq \lambda_2|\xi |^2$, for all $\xi\in\R^{n+1}$, for every $(x,y)$ and some ellipticity constants $0<\lambda_1\leq\lambda_2$. Moreover, the characteristic manifold $\Sigma$ is assumed to be invariant with respect to $A$; that is, there exists a suitable scalar function $\mu$ such that there exists a positive constant such that
\begin{equation}\label{boundmu}
\frac{1}{C}\leq\mu(x,y)\leq C
\end{equation}
and with
$$A(x,0)\cdot e_{y}=\mu(x,0) e_y.$$
\end{Assumption}
Whenever the hypothesis on $A$ are not specified, we always imply Assumption (HA). From now on, through this section, whenever not otherwise specified, in order to ease the notations, we will work with $A=\mathbb I$ every time this condition is not playing a role in the proofs. We actually wish to stress the fact that our results in this section, with the sole exception of Lemma \ref{rho-1},  still hold true for uniformly elliptic matrixes of the form defined in the introduction.

\subsection{Weighted Sobolev spaces}
The natural functional settings for our problems involves some weighted Sobolev spaces.  Following the definition in \cite{Nec}, we denote by $C^\infty(\overline\Omega)$ the set of real functions $u$ defined on $\overline\Omega$ such that the derivatives $D^\alpha u$ can be continuously extended to $\overline\Omega$ for all multiindices $\alpha$. Hence, for any $a\in\mathbb{R}$, $\varepsilon\geq0$ we define the weighted Sobolev space $H^1(\Omega,\rho_\varepsilon^a(y)\mathrm{d}z)$ as the closure of $C^\infty(\overline\Omega)$ with respect to the norm
$$\|u\|_{H^{1}(\Omega,\rho_\varepsilon^a(y)\mathrm{d}z)}=\left(\int_{\Omega}\rho_\varepsilon^au^2+\int_{\Omega}\rho_\varepsilon^a|\nabla u|^2\right)^{1/2}.$$
To ease the notation we will indicate briefly with
$$H^{1,a}(\Omega)=H^{1}(\Omega,|y|^a\mathrm{d}z)=H^{1}(\Omega,\rho_0^a(y)\mathrm{d}z).$$
In the same way, we define $H^{1}_0(\Omega,\rho_\varepsilon^a(y)\mathrm{d}z)$ as the closure of $C^\infty_c(\Omega)$ with respect to the norm
$$\|u\|_{H^{1}_0(\Omega,\rho_\varepsilon^a(y)\mathrm{d}z)}=\left(\int_{\Omega}\rho_\varepsilon^a|\nabla u|^2\right)^{1/2}.$$
As it is remarked in \cite{Nec}, in the case $\varepsilon=0$, when $a\leq -1$, the functions in these spaces have zero trace on $\Sigma$ (in fact the weight $|y|^a$ is not locally integrable), while as $a>1$, the traces on $\Sigma$ have no sense in general.


Following this intuition, we will denote by $\tilde H^1(\Omega,\rho_\varepsilon^a(y)\mathrm{d}z)$ the closure of $C^\infty_c(\overline\Omega\setminus\Sigma)$ with respect to the norm $\|\cdot\|_{H^1(\Omega,\rho_\varepsilon^a(y)\mathrm{d}z)}$. In particular, when $a<1$, there is a natural isometry (on balls $B$ centered in a point on $\Sigma$ of any radius)
$$T_\varepsilon^a: \tilde H^{1}(B,\rho_\varepsilon^a(y)\mathrm{d}z)\to\tilde H^{1}(B): u\mapsto v=\sqrt{\rho_\varepsilon^a} u,$$
where  $\tilde H^{1}(B)$ is endowed with with the equivalent norm with squared expression
\[
Q_\varepsilon(v)=\int_B|\nabla v|^2 +\left[ \left(\dfrac{\partial_y \rho_\varepsilon^a}{2\rho_\varepsilon^a}\right)^2+\partial_y\left(\dfrac{\partial_y \rho_\varepsilon^a}{2\rho_\varepsilon^a}\right)\right]v^2-\int_{	\partial B}\dfrac{\partial_y \rho_\varepsilon^a}{2\rho_\varepsilon^a}yv^2\;,
\]
(this is done in details in \cite{SirTerVit2}). We remark that both in the super singular and super degenerate cases, that is $a\in(-\infty,-1]\cup[1,+\infty)$ and $\varepsilon=0$, when the weight is taken outside the $A_2$ Muckenhoup class,  one has identity of the spaces
\begin{equation}\label{tildeH1=H1}
H^{1,a}(\Omega)=\tilde H^{1,a}(\Omega)\,.
\end{equation}
This happens for very opposite reasons: roughly speaking, when $a\leq-1$ then the singularity is so strong to force the function to annihiliate on $\Sigma$ (we will call this case the super singular case). Instead, when $a\geq1$, then the strong degeneracy leaves enough freedom to the function to allow it to be very irregular through $\Sigma$ (we will call this case the super degenerate case). In the latter case, $\Sigma$ has vanishing capacity with respect to the energy $\int |y|^a |\nabla u|^2$.  In other words, we are claiming the following
\begin{Proposition}\label{prop:super}
Let $a\in(-\infty,-1]\cup[1,+\infty)$. Then the space $C^\infty_c(\overline\Omega\setminus\Sigma)$ is dense in $H^{1,a}(\Omega)$. 
\end{Proposition}
\proof 
Let $a\leq-1$. Then let us fix $u\in C^\infty(\overline\Omega)$ such that $\|u\|_{H^{1,a}(\Omega)}<+\infty$. Obviously $u=0$ in $\Sigma$. Now, let us consider a monotone nondecreasing function $\eta\in C^\infty(\R)$ such that $\eta(t)=0$ for $|t|\leq1$ and $\eta(t)=t$ for $|t|\geq2$. Hence, for any $\varepsilon>0$, we can define
$$u_\varepsilon=\varepsilon\eta(u/\varepsilon).$$
It holds that $u_\varepsilon=0$ in $\{|u|\leq\varepsilon\}$ and $u_\varepsilon=u$ in $\{|u|\geq2\varepsilon\}$. Nevertheless, $\nabla u_\varepsilon=\eta'(u/\varepsilon)\nabla u$, with $\nabla u_\varepsilon=0$ in $\{|u|\leq\varepsilon\}$ and $\nabla u_\varepsilon=\nabla u$ in $\{|u|\geq2\varepsilon\}$. Hence,
$$\int_{\Omega}|y|^a|\nabla u_\varepsilon-\nabla u|^2=\int_{\Omega}|y|^a(\eta'(u/\varepsilon)-1)^2|\nabla u|^2\leq c\int_{\Omega\cap\{|u|\leq2\varepsilon\}}|y|^a|\nabla u|^2\to0.$$
Moreover,
$$\int_{\Omega}|y|^a(u_\varepsilon-u)^2=\int_{\Omega}|y|^a(\varepsilon\eta(u/\varepsilon)-1)^2u^2\leq c\int_{\Omega\cap\{|u|\leq2\varepsilon\}}|y|^au^2\to0.$$
Let now $a\geq1$. Then let us fix $u\in C^\infty(\overline\Omega)$ such that $\|u\|_{H^{1,a}(\Omega)}<+\infty$. Let us consider, for $0<\delta<1$ the function
$$f_{\delta}(y)=\begin{cases}
0 & \mbox{ on } \{(x,y) \in \Omega \colon \abs{y}\leq \delta^2\}, \\
\log\frac{y}{\delta^2}/\log\frac{1}{\delta} & \mbox{ on } \{(x,y) \in\Omega \colon \delta^2\leq\abs{y}\leq \delta\}, \\
1 & \mbox{ on } \{(x,y) \in \Omega\colon \abs{y}\geq\delta\}.
\end{cases}$$
Hence it is easy to see that $u_\delta=f_\delta u\to u$ strongly in $H^{1,a}(\Omega)$ as $\delta\to0$. Eventually we remark that one can replace $f_\delta$ with a function with the same properties which is $C^\infty(\overline\Omega)$.
\endproof
\begin{remark}
In literature there is another well known equivalent way to define the Sobolev spaces, which works whenever $\varepsilon >0$ and $a\in\R$, given by
$$W^{1,2}(\Omega,\rho_\varepsilon^a(y)\mathrm{d}z)=\left\{u\in L^2(\Omega,\rho_\varepsilon^a(y)\mathrm{d}z) \ \mathrm{having \ weak \ gradient}: \ \|u\|_{H^1(\Omega,\rho_\varepsilon^a(y)\mathrm{d}z)}<+\infty\right\}.$$
We shall be concerned, at some point, with the so called (H=W) property: 
\begin{equation}\label{H=W}
(H=W) \qquad H^{1}(\Omega,\rho_\varepsilon^a(y)\mathrm{d}z)=W^{1,2}(\Omega,\rho_\varepsilon^a(y)\mathrm{d}z).
\end{equation}

It is now established that also when $\varepsilon=0$ and $a\in(-1,1)$, then property \eqref{H=W} still holds (see e.g. \cite{AmbPinSpe, JeiKosShaTys}).
When $\varepsilon=0$, in the super singular case $a\leq-1$, by \eqref{tildeH1=H1} and the isometry $T_\varepsilon^a$ we will have a useful tool to work without property \eqref{H=W}. Moreover, when $\varepsilon=0$, in the super degenerate case $a\geq1$, we will see that in fact one can work without condition \eqref{H=W} by using an easy inclusion argument of spaces. 
\end{remark}
To end this section, we finally remark that we have the obvious embeddings, where we have fixed $a\in\mathbb{R}$, and $\varepsilon>0$,
\begin{equation*}
H^1(\Omega,\rho_\varepsilon^a(y)\mathrm{d}z)\subseteq H^{1,a^+}(\Omega),
\end{equation*}
with a constant of immersion $c>0$ not depending on $\varepsilon$ (where $a^+:=\max\{a,0\}$). This is due to the fact that on $\Omega$, if $a\geq 0$, for $0<\varepsilon_1<\varepsilon_2<+\infty$ there holds $|y|^a\leq\rho_{\varepsilon_1}^a(y)\leq\rho_{\varepsilon_2}^a(y)<(1+\mathrm{diam}(\Omega))^{a/2},$
whereas if $a\leq 0$, for $0<\varepsilon_1<\varepsilon_2<+\infty$, we have $(1+\mathrm{diam}(\Omega))^{a/2}<\rho_{\varepsilon_2}^a(y)\leq\rho_{\varepsilon_1}^a(y)\leq |y|^a$.

\subsection{Sobolev embeddings}\label{sobo}
Sobolev inequalities for weighted Sobolev spaces have been deeply studied in different contexts and by many authors (see for example \cite{CabRos,FabKenSer,Haj}) as they play a key role in the regularity theory for elliptic PDEs.  
We are concerned with a class of weighted Sobolev inequalities - not necessarily with the best constant nor with the best exponent -  for the class of approximating weights $\rho_\varepsilon^a$ but with constants which are uniform as $\varepsilon\to0$. For this aim we can use the known results of \cite{Haj} about Sobolev spaces involving general measures, where, in our context the measure is naturally defined as  $\mathrm{d}\mu=\rho_\varepsilon^a(y)\mathrm{d}z$. The basic requirement is a local growth condition on the measure of balls, which reflects in a local uniform in $\varepsilon$ integrability condition of the weights. When $a\in(-1,+\infty)$, then a bounded domain $\Omega\subset\R^{n+1}$ has $\mu(\Omega)<+\infty$ for any $\varepsilon\geq0$. According with \cite{Haj},  a domain $\Omega$ is said to be $d$-regular with respect to $\mu$ if there exists $b>0$ such that for any $z\in\Omega$, for any $r<\mathrm{diam}(\Omega)$,
$$\mu(B_r(z))\geq br^d.$$
In our context, we may assume up to rescalings, that $\mathrm{diam}(\Omega)\leq 1$. If $a\in(-1,+\infty)$, then any bounded $\Omega$ is $d$-regular with respect to $\mu$ and the effective dimension
$$d=n+1+a^+=n^*(a).$$
We remark that the constant $b>0$ can be taken independent of $\varepsilon\geq0$. Moreover, since we are interested in Sobolev inqualities with $p=2$, we have the following Sobolev embedding.
\begin{Theorem}\label{sobemb1}
Let $a\in(-1,+\infty)$, $n\geq2$, $\varepsilon\geq0$ and $u\in C^1_c(\Omega)$. Then there exists a constant which does not depend on $\varepsilon\geq0$ such that
\begin{equation*}\label{sobo>-1}
\left(\int_{\Omega}\rho_\varepsilon^a|u|^{2^*(a)}\right)^{2/2^*(a)}\leq c(d,b,p,\Omega)\int_{\Omega}\rho_\varepsilon^a|\nabla u|^2,
\end{equation*}
where the optimal embedding exponent is
\begin{equation*}\label{2*a}
2^*(a)=\frac{2(n+1+a^+)}{n+a^+-1}=\frac{2n^*(a)}{n^*(a)-2}.
\end{equation*}
When $n=1$ and $a^+>0$ the same inequality holds. When $n=1$ and $a^+=0$ then the embedding holds in any weighted $L^p(\Omega,\rho_\varepsilon^a(y)\mathrm{d}z)$ for $p>1$.
\end{Theorem}
The inequality extends by density for any function in $H^{1}_0(\Omega,\rho_\varepsilon^a(y)\mathrm{d}z)$.\\\\
Moreover, for $a\in(-\infty,1)$ using the already mentioned isometry $T_\varepsilon^a:\tilde H^1(B,\rho_\varepsilon^a(y)\mathrm{d}z)\to\tilde H^1(B)$ (see details in \cite{SirTerVit2}), where $B$ is a ball containing $\Omega$, we easily obtain
\begin{Theorem}\label{sobemb2}
Let $a\in(-\infty,1)$, $n\geq 2$, $\varepsilon\geq0$ and $u\in C^\infty_c(\Omega\setminus\Sigma)$. Then there exists a constant which does not depend on $\varepsilon\geq0$ such that
\begin{equation}\label{sobo<-10}
\left(\int_{\Omega}(\rho_\varepsilon^a)^{2^*/2}|u|^{2^*}\right)^{2/2^*}\leq c(n,a,\Omega)\int_{\Omega}\rho_\varepsilon^a|\nabla u|^2,
\end{equation}
where
$$2^*=\frac{2(n+1)}{n-1}.$$
When $n=1$ the above inequality holds with $2^*$ replaced with any $p>1$ and $c=c(n,a,p,\Omega)$.
\end{Theorem}
The inequality extends by density for any function in $\tilde H^{1}_0(\Omega,\rho_\varepsilon^a(y)\mathrm{d}z)$. Moreover, using the validity of \eqref{tildeH1=H1} when $a\leq-1$ and $\varepsilon=0$, this provides a Sobolev embedding in the non locally integrable case. We remark that inequality in \eqref{sobo<-10} implies, taking $\mathrm{diam}(\Omega)\leq 1$, the weaker inequality
\begin{equation*}\label{sobo<-1}
\left(\int_{\Omega}\rho_\varepsilon^a|u|^{2^*}\right)^{2/2^*}\leq c(n,a,\Omega)\int_{\Omega}\rho_\varepsilon^a|\nabla u|^2.
\end{equation*}
\begin{remark}
We remark that when $a\geq0$, then the constant in Theorem \ref{sobemb1} can be chosen independent from $\Omega$ (the inequality is scale invariant), while when $a<0$ this is not possible.
\end{remark}

\subsection{Energy solutions}
In the light of the Sobolev embeddings in the previous section, we now are in a position to give a notion of energy solution to the elliptic equation also in the case when $\varepsilon=0$, 
\begin{equation}\label{La}
-\mathcal L_au=|y|^af\qquad\mathrm{in \ }B_1.
\end{equation}
We say that $u\in H^{1,a}(B_1)$ is an energy solution to \eqref{La} if
\begin{equation}\label{variationLa}
\int_{B_1}|y|^aA(x,y)\nabla u\cdot\nabla\phi=\int_{B_1}|y|^af\phi,\qquad\forall\phi\in H^{1,a}_0(B_1).
\end{equation}
By $\mathcal L_a$-harmonic functions we will mean energy solutions $u\in H^{1,a}(B_1)$ to
\begin{equation}\label{Larm}
-\mathcal L_au=0\qquad\mathrm{in} \ B_1.
\end{equation}
Similarly, we can give a natural notion of energy solutions to
\begin{equation}\label{Lafks}
-\mathcal L_au=\mathrm{div}\left(|y|^aF\right)\qquad\mathrm{in \ }B_1.
\end{equation}
We say that $u\in H^{1,a}(B_1)$ is an energy solution to \eqref{Lafks} if
\begin{equation}\label{variationLafks}
-\int_{B_1}|y|^aA(x,y)\nabla u\cdot\nabla\phi=\int_{B_1} |y|^aF\cdot\nabla\phi,\qquad\forall\phi\in H^{1,a}_0(B_1).
\end{equation}
We remark that the condition in \eqref{variationLa} and \eqref{variationLafks} can be equivalently expressed testing with any $\phi\in C^\infty_c(B_1)$ if $a\in(-1,+\infty)$ and with any $\phi\in C^\infty_c(B_1\setminus\Sigma)$ if $a\in(-\infty,-1]\cup[1,+\infty)$.
In order to give a sense to energy solutions to \eqref{La} and \eqref{Lafks} we need the following minimal hypothesis on the right hand sides.
\begin{Assumption}[Hf]
Let $a\in(-1,+\infty)$. Then if $n\geq2$ or $n=1$ and $a^+>0$, the forcing term $f$ in \eqref{La} belongs to $L^p(B_1,|y|^a\mathrm{d}z)$ with $p\geq(2^*(a))'$ the conjugate exponent of $2^*(a)$; that is,
$$(2^*(a))'=\frac{2(n+1+a^+)}{n+a^++3}.$$
If $n=1$ and $a^+=0$ then $f\in L^p(B_1,|y|^a\mathrm{d}z)$ with $p>1$.\\
Let $a\in(-\infty,-1]$. Then if $n\geq2$, the condition on the forcing term is $|y|^{a/2}f\in L^p(B_1)$ with $p\geq(2^*(a))'=(2^*)'$. If $n=1$, then any $p>1$ is allowed.
\end{Assumption}
\begin{Assumption}[HF]
Let $a\in(-1,+\infty)$. The condition on the field $F=(f_1,...,f_{n+1})$ in \eqref{Lafks} is $F\in L^p(B_1,|y|^a\mathrm{d}z)$ with $p\geq2$. Let $a\in(-\infty,-1]$. Then the condition is $|y|^{a/2}F\in L^p(B_1)$ with $p\geq2$.
\end{Assumption}

\begin{remark}
These minimal regularity assumptions,  already partially mentioned in \cite{FabKenSer}, ensure  the right hand sides of  \eqref{La} and \eqref{Lafks} to be in the duals of the appropriate energy spaces. As a consequence,  let us fix $\overline u\in H^{1,a}(B_1)$. Then, there exists a unique energy solution $u\in H^{1,a}(B_1)$ to \eqref{La} or to \eqref{Lafks} such that $u-\overline u\in H^{1,a}_0(B_1)$. 
\end{remark}

\subsection{Boundary conditions, even and odd solutions}
Throughout this paper we shall often add to the differential equation either symmetries or, what's the same in our mind, some boundary conditions to equations \eqref{La} and \eqref {Lafks}. 

\begin{Definition}\label{evenoddsolutions}
Let $a\in\R$. We say that a function $u\in H^{1,a}(B_1)$ which is an energy solution to either \eqref{La} or to \eqref{Lafks} in $B_1$ is even in $y$ if $u(x,y)=u(x,-y)$ for almost every $z\in B_1$. We say that a function $u\in H^{1,a}(B_1)$ which is an energy solution to \eqref{La} or to \eqref{Lafks} in $B_1$ is odd in $y$ if $u(x,y)=-u(x,-y)$ for almost every $z\in B_1$.
\end{Definition}
\begin{remarks}
(i) We remark that if $u\in H^{1,a}(B_1)$ is an even/odd energy solution to \eqref{La} in $B_1$, then the forcing term $f$ must be even/odd in $y$. If $u\in H^{1,a}(B_1)$ is an even/odd energy solution to \eqref{Lafks} in $B_1$, then the divergence of the field $F=(f_1,\dots,f_n,f_{n+1})$ must be even/odd in $y$.\\
(ii) With a little abuse, we shall systematically associate the Neumann boundary condition 
$$\lim_{y\to 0} y^a\partial_y u=0, \; \textrm{on}\; \partial ^0 B^+_1$$ 
with even solutions (resp. Dirichlet  b.c. $\lim_{y\to 0} u=0$ on $\partial ^0 B^+_1$), meaning that the even (resp. odd)-in-$y$ extension of $u$ from the upper to the lower half ball is still an energy solution to the differential equation in the whole ball. This is consistent with the uniformly elliptic case of the regularized weights $\varepsilon>0$.\\
(iii) As already remarked in Proposition \ref{prop:super}, when $a\in(-\infty,-1]\cup[1,+\infty)$, the space $C_c^\infty(\overline B_1\setminus \Sigma)$ is dense in the energy space $H^{1,a}(B_1)$. As a consequence, the two involutions
\[
I_{\pm}: C_c^\infty(\overline B_1\setminus \Sigma)\to C_c^\infty(\overline B_1\setminus \Sigma): u\mapsto \begin{cases}u(z)\; &\textrm{if}\; z\in B^\pm_1\\ 0 &\textrm{if}\; z\in B^\mp_1\;\end{cases}
\]
extend continuously to the energy space, giving rise to an orthogonal splitting $$H^{1,a}(B_1)=H^{1,a}_+(B_1)\oplus H^{1,a}_{-}(B_1),$$ where $H^{1,a}_\pm(B_1)$ are the closures of $I_{\pm}(C_c^\infty(\overline B_1\setminus \Sigma))$ with respect to the $H^{1,a}$-norm. Thus it is immediate to check that, given an energy solution to either  \eqref{La} or to \eqref{Lafks} in $B_1$, also $I_\pm(u)$ are solutions to the similar equations wih $I_\pm(f)$  in replacement of $f$ (resp. $I_\pm(F)$ replacing $F$, with a some caution when taking the divergence which has to be intended in the distributional sense). 
\end{remarks}

\subsection{The limits of regularized problems}
This subsection is devoted to establish a first set of links between  solutions to the regularized problems (when $\varepsilon>0$) and solutions to the limit problem ($\varepsilon=0$). As a first step, we wish to show that, under the most natural assumptions on the right hand sides and the solutions, the limits are indeed energy solutions to the singular/degenerate problem. 
\begin{Lemma}\label{Good->Energy}
Let $a\in\R$. Let $\{u_\varepsilon\}$ for $\varepsilon\to0$ be family of solutions to either
\begin{equation}\label{eq:f}
-\mathrm{div}\left(\rho_\varepsilon^aA\nabla u_\varepsilon\right)=\rho_\varepsilon^af_\varepsilon\qquad\mathrm{in \ } B_1
\end{equation}
or to
\begin{equation}\label{eq:divF}
-\mathrm{div}\left(\rho_\varepsilon^aA\nabla u_\varepsilon\right)=\mathrm{div}\left(\rho_\varepsilon^aF_\varepsilon\right)\qquad\mathrm{in \ } B_1
\end{equation}
such that, for all $\varepsilon$
\begin{equation}\label{eq:H1bound}
\|u_\varepsilon\|_{H^1(B_1,\rho_\varepsilon^a\mathrm{d}z)}\leq c.
\end{equation}
Let moreover, respectively in the two cases, $f_\varepsilon\to f$ in $L^p_{\mathrm{loc}}(B_1\setminus\Sigma)$ with $p$ satisfying the Assumption (Hf), and with a (uniform in $\varepsilon\to0$) constant $c>0$ such that
$$\|f_\varepsilon\|_{L^p(B_1,\rho_\varepsilon^a\mathrm{d}z)}\leq c,$$
and $F_\varepsilon\to F$ in $L^p_{\mathrm{loc}}(B_1\setminus\Sigma)$ with $p$ satisfying the Assumption (HF), and with a (uniform in $\varepsilon\to0)$ constant $c>0$ such that
$$\|F_\varepsilon\|_{L^p(B_1,\rho_\varepsilon^a\mathrm{d}z)}\leq c.$$
Then, there exists a sequence such that $u_{\varepsilon_k}\to u$ in $H^1_{\mathrm{loc}}(B_1\setminus\Sigma)$, where the limit $u\in H^{1,a}(B_1)$ is an energy solution on $B_1$ respectively to \eqref{La} with $f$ satisfying Assumption (Hf) and to \eqref{Lafks} with $F$ satisfying Assumption (HF).\\
\end{Lemma}
\proof
Without loss of generality we will take $A=\mathbb I$. Consider an exhausting sequence of compact subsets of  $B_1\setminus\Sigma$. By a diagonal process,  we can extract a sequence $u_{\varepsilon_k}$ weakly converging to some limit $u$ in $H^1_{\mathrm{loc}}(B_1\setminus\Sigma)$. Then, by standard elliptic estimates applied on each compact set, we can extract a subsequence strongly converging in $H^1_{\mathrm{loc}}(B_1\setminus\Sigma)$ with the further property that $\nabla u_{\varepsilon_k}$ and $u_{\varepsilon_k}$ do converge almost everywhere. Similarily, using the $L^p_{\mathrm{loc}}(B_1\setminus\Sigma)$ convergence $f_\varepsilon\to f$, we have, for a sequence,
$$\rho^a_{\varepsilon_k} u^2_{\varepsilon_k}\longrightarrow |y|^au^2,\quad\quad\rho^a_{\varepsilon_k}|\nabla u_{\varepsilon_k}|^2\longrightarrow |y|^a|\nabla u|^2\quad\mathrm{and}\quad\rho^a_{\varepsilon_k}|f_{\varepsilon_k}|^p\longrightarrow |y|^a|f|^p,\quad\mathrm{a.e. \ in \ } B_1.$$
By the Fatou Lemma and the uniform bounds for the sequence $\{f_{\varepsilon_k}\}$, we can say that $f\in L^p(B_1,|y|^a\mathrm{d}z)$. Next we wish to show that $u\in H^{1,a}(B_1)$: to this end, we distinguish three cases. When $a\geq 0$, then we have
$$\|u_\varepsilon\|_{H^{1,a}(B_1)}\leq\|u_\varepsilon\|_{H^1(B_1,\rho_\varepsilon^a\mathrm{d}z)}\leq c\;,$$
which implies that any element $u_\varepsilon\in H^{1,a}(B_1)$ (using the sequence of $C^\infty(\overline B_1)$ functions which approximates $u_\varepsilon$ in $H^1(B_1,\rho_\varepsilon^a\mathrm{d}z)$) and by uniform boundedness, we infer weak convergence in $H^{1,a}(B_1)$. When $a<0$, in contrast, we have $$\|u_\varepsilon\|_{H^{1}(B_1)}\leq\|u_\varepsilon\|_{H^1(B_1,\rho_\varepsilon^a\mathrm{d}z)}\leq c\;,$$ so that we deduce $u\in H^1(B_1)$ and weak convergence in such a space.  Invoking again Fatou's Lemma  we can say that  both $\nabla u$ and $u$ belong to $L^2(B_1,|y|^a\mathrm{d}z)$. This is enough to conclude when $a\in(-1,0)$, thanks to the $W=H$ theorem (see \cite{AmbPinSpe} and references therein). Finally, when $a\leq-1$ we perform the change of variable $v_\varepsilon=\sqrt{\rho_\varepsilon^a}u_\varepsilon$ (see \cite{SirTerVit2}) and easily obtain weak convergence of the corresponding sequence in $H^1(B_1)$ to $v=|y|^{a/2}u$. Hence $u\in H^{1,a}(B_1)$ which is isometric to $H^1(B_1)$.

Next we prove that the limit solves the differential equation in the energy sense. At first, let $a>-1$. Then, for any $\phi\in C^\infty_c(B_1)$ we have
$$\rho^a_{\varepsilon_k} \nabla u_{\varepsilon_k}\cdot\nabla\phi\longrightarrow |y|^a\nabla u\cdot\nabla\phi,\quad\mathrm{and}\quad\rho^a_{\varepsilon_k} f_{\varepsilon_k}\phi\longrightarrow |y|^a f\phi\quad\mathrm{a.e. \ in \ } B_1,$$
and
$$\int_{B_1}\rho^a_{\varepsilon_k} \nabla u_{\varepsilon_k}\cdot\nabla\phi=\int_{B_1}\rho^a_{\varepsilon_k} f_{\varepsilon_k}\phi.$$
Moreover, since $|y|^a\in L^1(B_1)$, the families of functions $h^1_{\varepsilon_k}:=\rho^a_{\varepsilon_k} \nabla u_{\varepsilon_k}\cdot\nabla\phi$ and $h^2_{\varepsilon_k}:=\rho^a_{\varepsilon_k} f_{\varepsilon_k}\phi$ are uniformly integrable, in the sense that for $i=1,2$, and any $\eta>0$ there exists $\delta,\overline{\varepsilon}>0$ such that
$$\int_E|h^i_{\varepsilon_k}|<\eta\qquad\forall0<{\varepsilon_k}\leq\overline{\varepsilon} \ \mathrm{and} \ \forall E\subset B_1 \ \mathrm{with} \ |E|<\delta.$$
Let now $a\leq-1$. Then, we apply the same reasoning with $\phi\in C^\infty_c(B_1\setminus\Sigma)$, and we use the fact that $|y|^a\in L^1(B_1\cap\mathrm{supp}\phi)$.
Hence we can apply in both cases the Vitali's convergence Theorem over the families $\{h^i_{\varepsilon_k}\}$ (for $i=1,2$) obtaining
$$\int_{B_1}|y|^a\nabla u\cdot\nabla\phi=\lim_{{\varepsilon_k}\to0}\int_{B_1}\rho_{\varepsilon_k}^a\nabla u_{\varepsilon_k}\cdot\nabla\phi=\lim_{{\varepsilon_k}\to0}\int_{B_1}\rho_{\varepsilon_k}^af_{\varepsilon_k}\phi=\int_{B_1}|y|^af\phi.$$
For the case of righr hand sides in divergence form, we note that the considerations done on the family $\{u_{\varepsilon_k}\}$ hold also in this case. Moreover, thanks to the $L^p_{\mathrm{loc}}(B_1\setminus\Sigma)$ convergence $F_{\varepsilon_k}\to F$, we have
$$\rho^a_{\varepsilon_k} \left|F_{\varepsilon_k}\right|^p\longrightarrow |y|^a \left|F\right|^p,\quad\mathrm{a.e. \ in \ } B_1.$$
By the Fatou Lemma and the uniform bound on the sequence $\{F_{\varepsilon_k}\}$, we can say that $F\in L^p(B_1,|y|^a\mathrm{d}z)$.\\\\
Also in this case one consider separately the case $a>-1$ and $a\leq-1$, obtaining with the very same reasonings that the family of functions $\rho^a_{\varepsilon_k} F_{\varepsilon_k} \cdot\nabla\phi$ is uniformly integrable, where $\phi\in C^\infty_c(B_1)$ in the first case and $\phi\in C^\infty_c(B_1\setminus\Sigma)$ in the second case.

Hence we can apply the Vitali's convergence Theorem, getting
$$\int_{B_1}|y|^a\nabla u\cdot\nabla\phi=\lim_{{\varepsilon_k}\to0}\int_{B_1}\rho_{\varepsilon_k}^a\nabla u_{\varepsilon_k}\cdot\nabla\phi=-\lim_{{\varepsilon_k}\to0}\int_{B_1}\rho^a_{\varepsilon_k} F_{\varepsilon_k}\cdot\nabla\phi=-\int_{B_1}|y|^aF\cdot\nabla\phi.$$
\endproof

A simple variant of this Lemma, which will turn out to be useful later, concerns the limiting profile of uniformly converging sequences of solutions to the approximating equations.

\begin{Lemma}\label{Good->Energy_bounded}
Let $a\in\R$, and let $\{u_k\}$ be a sequence of solutions to \eqref{eq:f}, for $\varepsilon_k\to0$, uniformly converging to $u$ in $B_1$. Let moreover $\rho_\varepsilon^a f_\varepsilon\to |y|^a f$ in $L^1(B_1)$. Then, the  limit $u\in H^{1,a}(B_1)\cap L^\infty(B_1)$ is an energy solution to \eqref{La} on $B_1$ with $f\in L^1(B_1,|y|^a\mathrm{d}z)$.
\end{Lemma}

\begin{proof}
First we observe that, there is a constant $c$ such that, for all $k$,
$$\|u_k\|_{L^\infty(B_1)}\leq c.$$
Let $K$ be any compact subset of $B_1$ and $\eta$ be a cut-off function. Then, by multiplying the equations by $\eta^2 u_k$ one easily obtains bounds in the space $H^1_{\mathrm{loc}}(B_1,\rho^a_{\varepsilon_k}\mathrm{d}z)$. Similarly, by testing the equations with $\eta^2 (u_k-u)$, strong convergence in $H^1_{\mathrm{loc}}(B_1)$ is readily obtained. The proof then proceeds exactly as that of Lemma \ref{Good->Energy}. 
\end{proof}

To end this subsection, let us remind that we can locally weaken the energy bound \eqref{eq:H1bound} into an $L^2$ one by the following well known Caccioppoli type estimate:

\begin{Lemma}\label{L2>H1}
Let $a\in\mathbb{R}$ and $\varepsilon\geq0$. 
 Let $u\in H^{1}(B_1,\rho_\varepsilon^a(y)\mathrm{d}z)$ be an energy solution to either
\begin{equation}\label{Larhof}
-\mathrm{div}(\rho_\varepsilon^aA\nabla u)=\rho_\varepsilon^af\qquad\mathrm{in \ }B_1,
\end{equation}
or 
\begin{equation}\label{Larhofksf}
-\mathrm{div}(\rho_\varepsilon^aA\nabla u)=\mathrm{div}\left(\rho_\varepsilon^aF\right)\qquad\mathrm{in \ }B_1,
\end{equation}
with $f\in L^p(B_1,\rho_\varepsilon^a(y)\mathrm{d}z)$  and $F\in L^p(B_1,\rho_\varepsilon^a(y)\mathrm{d}z)$ with $p$ satisfying respectively Assumptions (Hf) and (HF). Then, for any $0<r<1$ there exists a positive constant independent of $\varepsilon$ such that either
$$\|\nabla u\|_{L^2(B_{r},\rho_\varepsilon^a(y)\mathrm{d}z)}\leq c\left(\|u\|_{L^2(B_1,\rho_\varepsilon^a(y)\mathrm{d}z)}+\|f\|_{L^p(B_1,\rho_\varepsilon^a(y)\mathrm{d}z)}\right),$$
or
$$\|\nabla u\|_{L^2(B_{r},\rho_\varepsilon^a(y)\mathrm{d}z)}\leq c\left(\|u\|_{L^2(B_1,\rho_\varepsilon^a(y)\mathrm{d}z)}+\|F\|_{L^p(B_1,\rho_\varepsilon^a(y)\mathrm{d}z)}\right),$$
respecively.
\end{Lemma}
\proof
Without loss of generality we can take $A=\mathbb I$. Let $u$ be a solution to \eqref{Larhof} and let $\eta\in C^\infty_c(B_1)$ be a non increasing radial cut-off function such that $0\leq\eta\leq1$ such that $\eta\equiv 1$ in $B_r$.  Let us test the equation \eqref{Larhof} with $\eta^2u^{\beta-1}$, with $\beta>1$. After some standard computations (see e.g, \cite[Chapter 8]{gt}), one obtains the following Caccioppoli type inequality.
\begin{equation}\label{eq:caccioppoli}
2C_\beta\int_{B_R}\rho|\nabla(\eta u^{\beta/2})|^2\leq 2\int_{B_R}\rho u^\beta|\nabla\eta|^2+\int_{B_R}\rho f\eta^2u^{\beta-1}.
\end{equation}
which yields the desired estimate, via Young inequality,  when $\beta=2$. A similar argument works in case of solutions to \eqref{Larhofksf}.
%
%
\endproof

\subsection{Approximating energy solutions}
Next we show that all solutions to the homogenous limit problem can be obtained as pointwise limits of families of solutions to the regularized ones. This will be achieved by a simple $\Gamma$-convergence argument.

\begin{Lemma}\label{Energy->Good}
Let $a\in\R$ and let $u$ be an energy solution to \eqref{Larm} on $B_1$. Then for any $0<r<1$, there exists a family $\{u_\varepsilon\}$ (for $\varepsilon\to0$) of solutions to
$$-\mathrm{div}\left(\rho_\varepsilon^aA\nabla u_\varepsilon\right)=0\qquad\mathrm{in \ } B_r$$
such that $u_\varepsilon\to u$ strongly in $H^1_{\mathrm{loc}}(B_r\setminus\Sigma)$ with a uniform in $\varepsilon\to0$ constant $c>0$ such that
$$\|u_\varepsilon\|_{H^1(B_r,\rho_\varepsilon^a\mathrm{d}z)}\leq c.$$
\end{Lemma}
\proof
Without loss of generality we will take $A=\mathbb I$.\\
$\bf{Case}$ $a>-1$.\\
First, we may replace $u$ with  $\overline u\in H^{1,a}(B_1)\cap C^\infty(\overline{B_{\frac{3+r}{2}}})$ such that $u-\overline u\in H^{1,a}_0(B_1)$. Then,
$$Y^a:=\{w\in H^{1,a}(B_1): w-u\in H^{1,a}_0(B_1)\}=\{w\in H^{1,a}(B_1): w-\overline u\in H^{1,a}_0(B_1)\}.$$
Moreover, defining
\begin{equation}\label{P0}
c_0=\inf\left\{\int_{B_1}|y|^a|\nabla w|^2:w\in Y^a\right\},
\end{equation}
then $c_0=\int_{B_1}|y|^a|\nabla u|^2.$ Let $0<r<1$ and $\eta_r\in C^\infty_c(B_{\frac{1+r}{2}})$ be a radial cut-off function such that $0\leq\eta_r\leq1$ in $B_{\frac{1+r}{2}}$, $\eta_r\equiv 1$ in $B_r$. Let us define for $0<\varepsilon\leq1$
\begin{equation*}\label{rhor}
\rho^a_{\varepsilon,r}:=\left((\varepsilon\eta_r)^2+y^2\right)^{a/2}\quad\mathrm{in \ } B_1.
\end{equation*}
Then we set the approximating problems
\begin{equation*}\label{Pepsr}
c_{\varepsilon,r}=\inf\left\{\int_{B_1}\rho^a_{\varepsilon,r}|\nabla w|^2:w\in Y^a_{\varepsilon,r}\right\},
\end{equation*}
where
$$Y^a_{\varepsilon,r}:=\{w\in H^{1}(B_1,\rho^a_{\varepsilon,r}\mathrm{d}x): w-\overline u\in H^{1}_0(B_1,\rho^a_{\varepsilon,r}\mathrm{d}x)\}.$$
Moreover let $w_{\varepsilon,r}$ be the minimizer; that is, such that $c_{\varepsilon,r}=\int_{B_1}\rho^a_{\varepsilon,r}|\nabla w_{\varepsilon,r}|^2$.\\\\
$\bf{Subcase}$ $a\geq0$.\\
Fixed $0<r<1$, taking $0<\varepsilon_1<\varepsilon_2<1$, by the inclusion $Y^a_{1,r}\subseteq Y^a_{\varepsilon_2,r}\subseteq Y^a_{\varepsilon_1,r}\subseteq Y^a$, one easily obtains
$$c_0\leq c_{\varepsilon_1,r}\leq c_{\varepsilon_2,r}\leq c_{1,r}.$$
Thus, the sequence $\{c_{\varepsilon,r}\}_{0<\varepsilon\leq1}$ is monotone non decreasing and has a limit
$$c_{\varepsilon,r}\searrow c_r\in[c_0,c_{1,r}]\quad\mathrm{as} \ \varepsilon\to0.$$
One has that
\begin{equation*}
\int_{B_1}|y|^a|\nabla w_{\varepsilon,r}|^2\leq\int_{B_1}\rho^a_{\varepsilon,r}|\nabla w_{\varepsilon,r}|^2=c_{\varepsilon,r}\leq c_{1,r}.
\end{equation*}
Hence, the set $\{w_{\varepsilon,r}\}$ is uniformly bounded in $H^{1,a}(B_1)$, and hence in the same space $w_{\varepsilon,r}\rightharpoonup\overline w$. Moreover, the sequence is contained in $Y^a$, which, as  a closed subspace, is weakly closed, so that the weak limit $\overline w\in Y^a$. Let us consider any $\phi\in C^\infty_c(B_1)$. Then, by the weak convergence in $H^{1,a}(B_1)$,
\begin{eqnarray*}
\int_{B_1}|y|^a\nabla\overline w\cdot\nabla\phi&=&\lim_{\varepsilon\to0}\int_{B_1}(|y|^a-\rho_{\varepsilon,r}^a)\nabla w_{\varepsilon,r}\cdot\nabla\phi+\lim_{\varepsilon\to0}\int_{B_1}\rho_{\varepsilon,r}^a\nabla w_{\varepsilon,r}\cdot\nabla\phi\\
&=&\lim_{\varepsilon\to0}\int_{B_1}(|y|^a-\rho_{\varepsilon,r}^a)\nabla w_{\varepsilon,r}\cdot\nabla\phi=0.
\end{eqnarray*}
In fact, since $|y|^a\leq\rho_{\varepsilon,r}^a$,
$$\left|\int_{B_1}(|y|^a-\rho_{\varepsilon,r}^a)\nabla w_{\varepsilon,r}\cdot\nabla\phi\right|\leq2(c_{\varepsilon,r})^{1/2}\left(\int_{B_1}||y|^a-\rho_{\varepsilon,r}^a| \ |\nabla\phi|^2\right)^{1/2}\to0.$$
Hence, $\overline w$ is an energy solution to $\mathcal L_a\overline w=0$ in $B_1$ with condition $\overline w-\overline u\in H^{1,a}_0(B_1)$, and by uniqueness of solutions to the Dirichlet problem, we obtain $\overline w=u$. Obviously the sequence $\{w_{\varepsilon,r}\}$ satisfies the desired conditions on $B_r$: in fact outside $\Sigma$, functions $w_{\varepsilon,r}$ are solutions of uniformly elliptic problems with ellipticity constants bounded from above and below uniformly in $0<\varepsilon\leq1$. So, in subsets $\omega$ compactly contained in $B_1\setminus\Sigma$ one must have convergence $w_{\varepsilon,r}\to u$ in $W^{2,p}(\omega)$.\\\\
$\bf{Subcase}$ $-1<a<0$.\\
Fixed $0<r<1$, taking $0<\varepsilon_1<\varepsilon_2<1$, by the inclusion $Y^a\subseteq Y^a_{\varepsilon_1,r}\subseteq Y^a_{\varepsilon_2,r}\subseteq Y^a_{1,r}$, one easily obtains
$$c_{1,r}\leq c_{\varepsilon_2,r}\leq c_{\varepsilon_1,r}\leq c_{0}.$$
Hence the sequence $\{c_{\varepsilon,r}\}_{0<\varepsilon\leq1}$ is monotone non increasing and there exists the limit
$$c_{\varepsilon,r}\nearrow c_r\in[c_{1,r},c_0]\quad\mathrm{as} \ \varepsilon\to0.$$
First of all, we remark that
\begin{equation*}
\int_{B_1}\rho^a_{1,r}|\nabla w_{\varepsilon,r}|^2\leq \int_{B_1}\rho^a_{\varepsilon,r}|\nabla w_{\varepsilon,r}|^2=c_{\varepsilon,r}\leq c_0.
\end{equation*}
Hence, the set $\{w_{\varepsilon,r}\}$ is uniformly bounded in $H^{1}(B_1,\rho^a_{1,r}(y)\mathrm{d}z)$, and hence in the same space $w_{\varepsilon,r}\rightharpoonup\overline w$. Therefore, outside $\Sigma$, functions $w_{\varepsilon,r}$ are solutions of uniformly elliptic problems with ellipticity constants bounded from above and below uniformly in $0<\varepsilon\leq1$. So, in subsets $\omega$ compactly contained in $B_1\setminus\Sigma$ one must have convergence $w_{\varepsilon,r}\to\overline w$ in $W^{2,p}(\omega)$. This is enough to have the pointwise convergence $|\nabla w_{\varepsilon,r}|^2\to|\nabla \overline w|^2$ almost everywhere in $B_1$. Hence, by Fatou's Lemma
\begin{equation*}
\int_{B_1}|y|^a|\nabla \overline w|^2\leq \liminf_{\varepsilon\to0}\int_{B_1}\rho^a_{\varepsilon,r}|\nabla w_{\varepsilon,r}|^2=\liminf_{\varepsilon\to0}c_{\varepsilon,r}\leq c_0.
\end{equation*}
Hence $\overline w\in Y^a_{\varepsilon,r}$ for all $\varepsilon>0$, since definitely the sequence $\{w_{\varepsilon,r}\}$ is contained in any of them which are convex and closed and so weakly closed. Hence, by the (H=W) condition in \eqref{H=W} for the $A_2$ case, we have that $\overline w\in H^{1,a}(B_1)$.\\\\
We remark that for any $\varepsilon\in[0,1]$, the weight $\rho_{\varepsilon,r}^a(y)=|y|^a$ in $B_1\setminus B_{\frac{1+r}{2}}$. Hence, by weak convergence in $Y^a_{1,r}$, we obtain, for any $\phi\in C^\infty_c(B_1\setminus B_{\frac{1+r}{2}})\cap H^{1,a}(B_1\setminus B_{\frac{1+r}{2}})$
$$\int_{B_1\setminus B_{\frac{1+r}{2}}}|y|^a\nabla\overline w\cdot\nabla\phi=\lim_{\varepsilon\to0}\int_{B_1\setminus B_{\frac{1+r}{2}}}|y|^a\nabla w_{\varepsilon,r}\cdot\nabla\phi=0$$
since any $w_{\varepsilon,r}$ is solution to
$$-\mathrm{div}(|y|^a\nabla w_{\varepsilon,r})=0\qquad\mathrm{in \ }B_1\setminus B_{\frac{1+r}{2}}.$$
Hence, also $\overline w$ is solution on the annulus. Let us consider $\eta\in C^\infty_c(B_t)$ cut off radial decreasing with $\eta\equiv 1$ in $B_{\frac{1+r}{2}}$, $0\leq\eta\leq1$ and $t\in(\frac{1+r}{2},1)$. So, testing the equation of the difference with $(1-\eta)^2(w_{\varepsilon,r}-\overline w)$ we obtain
\begin{equation*}
\int_{B_1\setminus B_{\frac{1+r}{2}}}|y|^a|\nabla((1-\eta)(w_{\varepsilon,r}-\overline w))|^2=\int_{B_1\setminus B_{\frac{1+r}{2}}}|y|^a|\nabla(1-\eta)|^2(w_{\varepsilon,r}-\overline w)^2\to0,
\end{equation*}
by the compact embedding in $L^{2,a}(B_1\setminus B_{\frac{1+r}{2}})$. So we obtain
$$\int_{B_1\setminus B_{t}}|y|^a|\nabla((w_{\varepsilon,r}-\overline w))|^2\to0.$$
Hence, in order to prove that $\overline w\in Y^a$, it remains to prove the existence for any $\delta>0$ of a function $v_\delta\in C^\infty_c(B_1)$ such that
$$\|\overline w-\overline u-v_\delta\|_{H^{1,a}_0(B_1)}<\delta.$$
This can be done considering $\varepsilon$ small enough such that 
$$\int_{B_1\setminus B_{t}}|y|^a|\nabla((w_{\varepsilon,r}-\overline w))|^2<\delta.$$
Hence, since $w_{\varepsilon,r}\in Y^a_{\varepsilon,r}$, we consider $\phi_\delta\in C^\infty_c(B_1)$ such that
$$\int_{B_1\setminus B_{\frac{1+r}{2}}}|y|^a|\nabla((w_{\varepsilon,r}-\overline u-\phi_\delta))|^2<\delta.$$
Moreover, using the fact that $\overline w-\overline u\in H^{1,a}(B_1)$, we can choose $\psi_\delta\in C^\infty(B_1)$ such that
$$\int_{B_1}|y|^a|\nabla(\overline w-\overline u-\psi_\delta))|^2<\delta.$$
Hence, considering the radial cut-off function $f_r\in C^\infty_c(B_{\frac{3+r}{4}})$ such that $0\leq f_r\leq1$ in $B_{\frac{3+r}{4}}$ and $ f_r\equiv 1$ in $B_{t}$. Hence the function $v_\delta:=(1-f_r)\phi_\delta+f_r\psi_r$  is the desired function.\\\\
Hence, $\overline w\in Y^a$ and so it is a competitor for the problem in \eqref{P0}. By the minimality of $u$, we obtain $\overline w=u$. Moreover the family $\{w_{\varepsilon,r}\}$ satisfies the desired conditions in $B_r$.\\\\
$\bf{Case}$ $a\leq-1$.\\
In this case we remark that our energy solution $u\in H^{1,a}(B_1)=\tilde H^{1,a}(B_1)$. Let us consider the function $v=|y|^{a/2}u\in \tilde H^1(B_1)$ which is minimizer for the quadratic form
$$c_0= Q_a(v)= \int_{B_1}|\nabla v|^2+\left(\frac{a^2}{4}-\frac{a}{2}\right)\frac{v^2}{y^2}-\frac{a}{2}\int_{\partial B_1}v^2$$
in $Y=\{w\in\tilde H^1(B_1) \ : \ w-v\in\tilde H^1_0(B_1)\}$.
Hence let us consider for any $0<\varepsilon<1$, the minimizer $v_\varepsilon$ in $Y$ for the form
\[
c_\varepsilon=Q_\varepsilon(v_\varepsilon)=\int_{B_1}|\nabla v_\varepsilon|^2 +\left[ \left(\dfrac{\partial_y \rho_\varepsilon^a}{2\rho_\varepsilon^a}\right)^2+\partial_y\left(\dfrac{\partial_y \rho_\varepsilon^a}{2\rho_\varepsilon^a}\right)\right]v_\varepsilon^2-\int_{	\partial B_1}\dfrac{\partial_y \rho_\varepsilon^a}{2\rho_\varepsilon^a}yv_\varepsilon^2\;.
\]
One can easily prove that in fact $Q_\varepsilon$ are equivalent norms in $\tilde H^1(B_1)$ (this is done in details in \cite{SirTerVit2}), with a positive constant which does not depend on $\varepsilon\geq0$ such that
$$\frac{1}{c}\|w\|_{\tilde H^1(B_1)}\leq Q_\varepsilon(w)\leq c\|w\|_{\tilde H^1(B_1)}.$$
This provides the uniform bound
$$Q_\varepsilon(v_\varepsilon)\leq c,$$
and hence weak convergence $v_\varepsilon\to v$ in $\tilde H^1(B_1)$. Nevertheless, by testing the equations related to the quadratic forms $Q_\varepsilon$ with $v_\varepsilon-v$, one can show that the convergence is strong. By the isometry $T_\varepsilon^a: \tilde H^{1}(B_1,\rho_\varepsilon^a(y)\mathrm{d}z)\to\tilde H^1(B_1)$, one obtain easily that the sequence
$$u_\varepsilon=\frac{v_\varepsilon}{(\rho_\varepsilon^a)^{1/2}}$$
satisfies the thesis, since they are solutions to $$-\mathrm{div}\left(\rho_\varepsilon^a\nabla u_\varepsilon\right)=0\qquad\mathrm{in \ } B_1,$$
with uniform bound $$\|u_\varepsilon\|_{H^1(B_1,\rho_\varepsilon^a\mathrm{d}z)}=Q_\varepsilon(v_\varepsilon)\leq c.$$
Nevertheless, the convergence $u_\varepsilon\to u$ is strong in $H^1_{\mathrm{loc}}(B_1\setminus\Sigma)$, using the strong convergence $v_\varepsilon\to v$ in $\tilde H^1(B_1)$ and the fact that $\rho_\varepsilon^a\to|y|^a$ in $C^\infty_{\mathrm{loc}}(B_1\setminus\Sigma)$.
\endproof

An immediate, yet relevant, consequence of this result is the validity of the following duality relation between the weights $\rho_\varepsilon^{a}$ and $\rho_\varepsilon^{-a}$
\begin{Lemma}\label{rho-1}
Let $a\in\mathbb{R}$, $\varepsilon\geq0$ and let $w$ be an energy solution to
$$-\mathrm{div}\left(\rho_\varepsilon^a\nabla w\right)=0\qquad\mathrm{in \ }B_1.$$
Then $v=\rho_\varepsilon^a\partial_yw$ is a (local) energy solution to
$$-\mathrm{div}\left(\rho_\varepsilon^{-a}\nabla v\right)=0\qquad\mathrm{in \ }B_1.$$
\end{Lemma}
\proof
The case $\varepsilon>0$ goes through the explicit computation. To prove the assertion in the limit case $\varepsilon=0$, we use Lemma \ref{Energy->Good} in order to approximate the solution $w$ with  solutions $w_\varepsilon$ with $\varepsilon>0$, keeping a (local) energy bound $H^1(\rho_\varepsilon^{a}\mathrm{d}z)$. Thus $v_\varepsilon=\rho_\varepsilon^a\partial_yw_\varepsilon$ are solutions to the second equation, with a uniform-in-$\varepsilon$  $L^2(\rho_\varepsilon^{-a}\rm{d}z)$ bound, which promptly reflects into a local $H^1(\rho_\varepsilon^{-a}\mathrm{d}z)$ one. To end the proof, we apply Lemma \ref{Good->Energy}. 
\endproof

\subsection{Moser iterative technique and $L^\infty$ bounds}\label{subsec:moser}
Another important issue is boundedness of energy solutions to \eqref{La} and to \eqref{Lafks}. Using a Moser iteration argument (see also \cite[Section 8.4]{gt}), one can prove the following nowadays standard result. 
\begin{Proposition}\label{Moser}
Let $a\in\R$ and $\varepsilon\geq0$. Let $u\in H^{1}(B_1,\rho_\varepsilon^a(y)\mathrm{d}z)$ be an energy solution to
\begin{equation}\label{Larho}
-\mathrm{div}(\rho_\varepsilon^aA\nabla u)=\rho_\varepsilon^af\qquad\mathrm{in \ }B_1,
\end{equation}
or to
\begin{equation}\label{Larhofks}
-\mathrm{div}(\rho_\varepsilon^aA\nabla u)=\mathrm{div}\left(\rho_\varepsilon^aF\right)\qquad\mathrm{in \ }B_1,
\end{equation}
with $f\in L^p(B_1,\rho_\varepsilon^a(y)\mathrm{d}z)$ and $$p>\frac{n+1+a^+}{2},$$
or with $F\in L^p(B_1,\rho_\varepsilon^a(y)\mathrm{d}z)$ and $$p>n+1+a^+.$$
Then, for any $0<r<1$ and $\beta_0>1$ either for solutions to \eqref{Larho} or to \eqref{Larhofks} there exists a positive constant independent of $\varepsilon$ such that respectively
$$\|u\|_{L^\infty(B_{r})}\leq c\left(\|u\|_{L^{\beta_0}(B_1,\rho_\varepsilon^a(y)\mathrm{d}z)}+\|f\|_{L^p(B_1,\rho_\varepsilon^a(y)\mathrm{d}z)}\right),$$
and
$$\|u\|_{L^\infty(B_{r})}\leq c\left(\|u\|_{L^{\beta_0}(B_1,\rho_\varepsilon^a(y)\mathrm{d}z)}+\|F\|_{L^p(B_1,\rho_\varepsilon^a(y)\mathrm{d}z)}\right).$$
\end{Proposition}
\proof
Without loss of generality we will take $A=\mathbb I$. We prove the result for \eqref{Larho} (the other one is analogous). We want to apply the Moser iterative method. Let us fix $0<\overline r<1$. We take a sequence of radii $\{r_k\}$ such that
$$\begin{cases}
r_0=1\\
r_{k+1}=\frac{r_k+\overline r}{2}\\
r_k-r_{k+1}=\frac{1-\overline r}{2^{k+1}}.
\end{cases}$$
Let $\chi=2^*(a)/2$. We take also a sequence of exponents $\{\beta_k\}$ such that
$$\begin{cases}
\beta_0>1\\
\beta_{k}=\beta_0\chi^k.
\end{cases}$$
Moreover, let us consider a sequence of radial non increasing cut off functions $\{\eta_k\}$ such that
$$\begin{cases}
\eta_k\in C^\infty_c(B_{r_k})\\
0\leq\eta_k\leq1\\
\eta_k\equiv1 \ \mathrm{in \ }B_{r_{k+1}}\\
|\nabla\eta_k|\leq\frac{1}{r_k-r_{k+1}}.
\end{cases}$$
for the general pass $k$ of the iteration we will indicate briefly these objects as $r_k=R$, $r_{k+1}=r$, $\beta_k=\beta$ and $\eta_k=\eta$. Moreover for simplicity we will recall $\rho=\rho_\varepsilon^a$\\\\
Applying our Sobolev embedding results in the left hand side of the Caccioppoli inequality \eqref{eq:caccioppoli}, and an H\"older inequality on the second term in the right hand side, we obtain, for some constant which are uniform in $\varepsilon$,
$$\left(\int_{B_R}\rho(\eta u^{\beta/2})^{2^*(a)}\right)^{2/2^*(a)}\leq \frac{c}{(R-r)^2}\int_{B_R}\rho u^\beta+\|f\|_{L^p(B_1,\rho\mathrm{d}z)}\left(\int_{B_R}\rho(\eta u^{\beta/2})^{t}\right)^{1/p'},$$
where $t=2\frac{\beta-1}{\beta}p'$. Hence we apply an interpolation inequality with exponents
$$\frac{1}{t}=\frac{\delta}{2}+\frac{1-\delta}{2^*(a)},\qquad\mathrm{with \ }\delta=1-\frac{n+1+a^+}{2p}+\frac{n+1+a^+}{2p'(\beta-1)}.$$
We remark that as $\beta\to+\infty$, $\delta\to 1-\frac{n+1+a^+}{2p}>0$. So
$$\left(\int_{B_R}\rho(\eta u^{\beta/2})^{t}\right)^{1/p'}\leq\left(\int_{B_R}\rho \eta^2u^\beta\right)^{\frac{\beta-1}{\beta}\delta}\left(\int_{B_R}\rho(\eta u^{\beta/2})^{2^*(a)}\right)^{\frac{2}{2^*(a)}\frac{(\beta-1)}{\beta}(1-\delta)}.$$
Hence, using the Young inequality, we have
\begin{eqnarray*}
\|f\|_{L^p(B_1,\rho\mathrm{d}z)}\left(\int_{B_R}\rho(\eta u^{\beta/2})^{t}\right)^{1/p'}&\leq&\frac{\beta-1}{\beta}\delta\|f\|_{L^p(B_1,\rho\mathrm{d}z)}^{\frac{\beta}{\delta(\beta-1)}}\int_{B_R}\rho u^\beta\\
&&+\frac{\beta-1}{\beta}(1-\delta)\left(\int_{B_R}\rho(\eta u^{\beta/2})^{2^*(a)}\right)^{2/2^*(a)}.
\end{eqnarray*}
Putting together these computations, we find a positive constant $\overline c>0$, which is uniform with respect to $\beta\to+\infty$ and which depends on the $L^p(B_1,\rho\mathrm{d}z)$-norm of $f$, such that
$$\left(\int_{B_R}\rho(\eta u^{\beta/2})^{2^*(a)}\right)^{2/2^*(a)}\leq\left(\frac{\overline c}{R-r}\right)^2\int_{B_R}\rho u^\beta;$$
that is,
$$\left(\int_{B_r}\rho u^{\beta\chi}\right)^{1/\beta\chi}\leq\left(\frac{\overline c}{R-r}\right)^{2/\beta}\left(\int_{B_R}\rho u^\beta\right)^{1/\beta}.$$
Hence, applying an iteration we obtain
\begin{equation*}
\|u\|_{L^{\beta_{k+1}}(B_{r_{k+1}},\rho\mathrm{d}z)}\leq\prod_{j=0}^k\left(\frac{2\overline c}{R-r}\right)^{2/\beta_j}\prod_{j=0}^k(2^j)^{2/\beta_j}\|u\|_{L^{\beta_{0}}(B_{r_0},\rho\mathrm{d}z)}.
\end{equation*}
At this point we have to distinguish two different cases: when $a>-1$, each measure $\mathrm{d}\mu_\varepsilon=\rho_\varepsilon^a(y)\mathrm{d}z$ (comprising the limit case for $\varepsilon=0$) is absolutely continuous with respect to the Lebesgue one. Since the sequences 
$$\sum_{j=0}^k\frac{2}{\beta_j}\log\left(\frac{2\overline c}{R-r}\right),\qquad\sum_{j=1}^k\frac{2j}{\beta_j}\log 2$$
are convergent, passing to the limit we obtain
$$\|u\|_{L^{\infty}(B_{\overline r},\rho_\varepsilon^a\mathrm{d}z)}\leq c\|u\|_{L^{2}(B_{1},\rho_\varepsilon^a\mathrm{d}z)},$$
and by the absolute continuity mentioned above we obtain the result. Eventually, when $a\leq-1$ the absolute continuity property fails, but one can trivially use the following embedding
\begin{equation*}
\|u\|_{L^{\beta_{k+1}}(B_{r_{k+1}},\rho^b\mathrm{d}z)}\leq\|u\|_{L^{\beta_{k+1}}(B_{r_{k+1}},\rho^a\mathrm{d}z)}\leq\prod_{j=0}^k\left(\frac{2\overline c}{R-r}\right)^{2/\beta_j}\prod_{j=0}^k(2^j)^{2/\beta_j}\|u\|_{L^{\beta_{0}}(B_{r_0},\rho^a\mathrm{d}z)},
\end{equation*}
with $b=-1/2>a$ for instance. Hence, one can conclude as before. We remark that the constant in the result does depend on the $L^p(B_1,\rho_\varepsilon^a\mathrm{d}z)$-norm of $f$.
\endproof

\section{Liouville theorems}\label{sec:liouville}
In this section we provide two fundamental results which will be the main tools in order to prove regularity local estimates which are uniform with respect to $\varepsilon\geq0$ and which are contained in the main body of the paper.

\subsection{Trace and Boundary Hardy type inequalities}

At first we turn to the validity of Hardy (trace) type inequalities and their spectral stability. These results will be the key tools in order to establish a class of Liouville theorems. To start with, we need the following inequality, which is proved in \cite{SirTerVit2}: 

\begin{Lemma}[Boundary Hardy Inequality]
Let $a\in(-\infty,1)$ and $\varepsilon_0>0$. There exists $c_0>0$ such that, for every $\varepsilon\in[0,\varepsilon_0]$ and  $u\in\tilde H^1(B_1,\rho_\varepsilon^a(y)\mathrm{d}z)$, 

\begin{equation}\label{eq:boundaryHardy}
\int_{B_1}\rho_\varepsilon^a |\nabla u|^2\geq c_0\int_{\partial B_1}\rho_\varepsilon^a \frac{1}{y}u^2\;.
\end{equation}
\end{Lemma}
Next, let us consider the special solution
\begin{equation*}
\overline u_\varepsilon(y)=\int_0^y \rho_\varepsilon^{-a}(s)ds
\end{equation*}
to the equation 
\begin{equation}\label{eq:harmonic}
-\mathrm{div}(\rho_\varepsilon^a\nabla u)=0\;.
\end{equation}
We notice that $\overline u$ satisfies the weighted boundary condition:
\begin{equation*}
\rho_\varepsilon^{a}\dfrac{\partial \overline u}{\partial\nu}=y\qquad \mathrm{on}\;\partial B_1.
\end{equation*}
Let $u\in C^\infty_c(\overline\Omega\setminus\Sigma)$. By multiplying equation \eqref{eq:harmonic} by $u^2/\overline u$  and integrating over the ball we easily obtain:
\begin{equation*}
\begin{split}
 \int_{\partial B_1}\rho_\varepsilon^a \frac{ y}{\rho_\varepsilon^a \overline u_\varepsilon}u^2= \int_{\partial B_1}\rho_\varepsilon^{a}\dfrac{\partial \overline u}{\partial\nu}\dfrac{u^2}{\overline u}
=\int_{B_1}\rho_\varepsilon^a\nabla\overline  u\cdot \nabla \left(\dfrac{u^2}{\overline u}\right)=\\=\int_{B_1}\rho_\varepsilon^a|\nabla  u|^2- \rho_\varepsilon^a\left|\nabla u- \dfrac{u}{\overline u}\nabla \overline u\right|^2\leq \int_{B_1}\rho_\varepsilon^a|\nabla  u|^2\;.
\end{split}
\end{equation*}
In other words, we have the following
\begin{Lemma}[Trace Inequality]
Let $a\in(-\infty,1)$ and $\varepsilon\geq0$. For every $u\in\tilde H^1(B_1,\rho_\varepsilon^a(y)\mathrm{d}z)$, 
\begin{equation}\label{eq:trace}
\int_{B_1}\rho_\varepsilon^a |\nabla u|^2\geq \int_{\partial B_1}\rho_\varepsilon^a \frac{ y}{\rho_\varepsilon^a \overline u_\varepsilon}u^2\;.
\end{equation}
If $\varepsilon=0$, then we have
\begin{equation}\label{eq:tracea}
\int_{B_1}|y|^a |\nabla u|^2\geq (1-a) \int_{\partial B_1}|y|^a u^2\;.
\end{equation}
\end{Lemma}

With this tools we can prove the following further inequality:

\begin{Lemma}[Stability]\label{lemma:stability}
Let $a\in(-\infty,1)$ and $\mu<1-a$. There exists $\varepsilon_\mu>0$ such that, for every  $\varepsilon\in(0,\varepsilon_\mu)$ and for every $u\in\tilde H^1(B_1,\rho_\varepsilon^a(y)\mathrm{d}z)$ there holds, 

\begin{equation*}
\int_{B_1}\rho_\varepsilon^a |\nabla u|^2-\frac a2\varepsilon^2\int_{\partial B_1}\rho_\varepsilon^a\dfrac{u^2}{\varepsilon^2+y^2}\geq \mu\int_{\partial B_1}\rho_\varepsilon^a u^2\;.
\end{equation*}
\end{Lemma}

\begin{proof}
At first, we observe that there is a scaled quantity involved:
\[\mu_\varepsilon(y):=\frac{ y}{\rho_\varepsilon^a(y) \overline u_\varepsilon(y)}	=\mu_1(y/\varepsilon)\;,\]
with
\[
\lim_{y\to 0}\mu_1(y)=1\;,\qquad \lim_{y\to \infty}\mu_1(y)=1-a\;.
\]
Hence, as $\mu<1-a$, for $0<\delta<<1$ we have, using first \eqref{eq:boundaryHardy} and then \eqref{eq:trace},
\begin{equation*}
\begin{split}
\int_{B_1}\rho_\varepsilon^a |\nabla u|^2-\frac a2\varepsilon^2\int_{\partial B_1}\rho_\varepsilon^a\dfrac{u^2}{\varepsilon^2+y^2}  \geq
(1-\delta)\int_{B_1}\rho_\varepsilon^a |\nabla u|^2-\frac a2\varepsilon^2\int_{\partial B_1}\rho_\varepsilon^a\dfrac{u^2}{\varepsilon^2+y^2}+\delta c_0\int_{\partial B_1}\rho_\varepsilon^a \frac{1}{y}u^2\\ \geq 
\int_{\partial B_1}\rho_\varepsilon^a\left( (1-\delta)\frac{ y}{\rho_\varepsilon^a \overline u_\varepsilon}- \dfrac{a\varepsilon^2}{2(\varepsilon^2+y^2)}+\dfrac{\delta c_0}{y}\right)u^2 \geq
\mu\int_{\partial B_1}\rho_\varepsilon^a u^2\;.
\end{split}
\end{equation*}
Indeed, we have
\[ \lim_{\varepsilon\to 0}\left((1-\delta)\frac{ y}{\rho_\varepsilon^a \overline u_\varepsilon}- \dfrac{a\varepsilon^2}{2(\varepsilon^2+y^2)}\right)=(1-\delta)(1-a)>\mu\]
uniformly on every closed half-line $[y_0, +\infty)$ with $y_0>0$, and the  function is uniformly bounded in $[0,+\infty)$.
\end{proof}

\begin{Theorem}\label{Liouvilleeven}
Let $a\in(-\infty,1)$, $\varepsilon\geq0$ and $w$ be a (locally) energy solution to
\begin{equation}
\begin{cases}\label{eq:Liouvilleeven}
-\mathrm{div}(\rho_\varepsilon^a(y)\nabla w)=0 & \mathrm{in \ }\mathbb{R}^{n+1}_+\\
w=0 &\mathrm{in \ }\mathbb{R}^{n}\times\{0\},
\end{cases}
\end{equation}
and let us suppose that for some $\gamma\in[0,1)$, $C>0$ there holds
\begin{equation}\label{eq:wronggrowth}
|w(z)|\leq C\rho_\varepsilon^{-a}(y)(1 + |z|^\gamma)
\end{equation}
for every $z=(x,y)$. Then $w$ is identically zero.
\end{Theorem}
\proof
By a simple normalization argument, it is enough to prove the result only for $\varepsilon\in\{0,1\}$. We start with\\
$\bf{Case \ 1:}$ \ $\varepsilon=0$. Let us consider $w\in H^{1,a}_{\mathrm{loc}}(\mathbb{R}^{n+1}_+)$  and let us define
$$E(w,r)=E(r)=\frac{1}{r^{n+a-1}}\int_{B_r^+}y^{a}|\nabla w|^2,\qquad H(w,r)=H(r)=\frac{1}{r^{n+a}}\int_{\partial^+ B_r^+}y^{a}w^2.$$
Next we can compute
$$H'(r)=\frac{2}{r}E(r),$$
indeed, denoting $w^r(z)=w(rz)$ we have
$$E(r)=\int_{B_1^+}y^a|\nabla w^r|^2\qquad\mathrm{and}\qquad H(r)=\int_{S^n_+}y^a(w^r)^2,$$
Since we know from \eqref{eq:tracea} that
\begin{equation*}\label{tpti2}
\int_{B_1^+}y^a|\nabla u|^2\geq (1-a)\int_{S^n_+}y^a u^2\;,
\end{equation*}
 there holds
\begin{equation*}
H'(r)\geq\frac{2(1-a)}{r}H(r) \quad\Longrightarrow\quad
\frac{H(r)}{r^{2(1-a)}}\geq H(1), \qquad \forall r\geq 1\;.
\end{equation*}
On the other hand, as the weight $y^{-a}$ is locally integrable, \eqref{eq:wronggrowth} implies
\begin{equation}\label{eq:wgh}
H(r)\leq Cr^{-2a}(1+r^{2\gamma})\;, \forall r>0\;,
\end{equation}
with $\gamma <1$, a contradiction.\\
$\bf{Case \ 2:}$ \ $\varepsilon=1$. In this case we define
$$E(w,r)=E(r)=\frac{1}{r^{n+a-1}}\int_{B_r^+}(1+y^2)^{a/2}|\nabla w|^2,\qquad H(w,r)=H(r)=\frac{1}{r^{n+a}}\int_{\partial^+ B_r^+}(1+y^2)^{a/2}w^2.$$
Again, defining $w^r(z)=w(rz)$ one has
$$E(r)=\int_{B_1^+}\left(\frac{1}{r^2}+y^2\right)^{a/2}|\nabla w^r|^2,\qquad H(r)=\int_{S^n_+}\left(\frac{1}{r^2}+y^2\right)^{a/2}(w^r)^2\;,$$
which give
\begin{equation*}\label{H111}
H'(r)=\frac{2}{r}E(r)-\frac{a}{r^{3}}\int_{\partial^+ B_1^+}\left(\frac{1}{r^2}+y^2\right)^{a/2-1}(w^r)^2.
\end{equation*}
Let us take $\mu$ such that $\gamma-a<\mu<1-a$; by Lemma \ref{lemma:stability} we can find $\varepsilon_\mu>0$ such that, if $r>1/\varepsilon_\mu$ we have

\[
\int_{B_1^+}\left(\frac{1}{r^2}+y^2\right)^{a/2}|\nabla u|^2-\frac{a}{2r^2}\int_{S^n_+}\left(\frac{1}{r^2}+y^2\right)^{a/2-1}u^2
\geq \mu \int_{S^n_+}\left(\frac{1}{r^2}+y^2\right)^{a/2} u^2\;,
\]
Which implies
\[
H'(r)\geq\frac{2\mu}{r}H(r)\;,\quad\forall r\geq 1/\varepsilon_\mu\;.
\]

By integrating the above expression we find a minimal growth rate for $H$ of order $2\mu>2(\gamma-a)$, in contradiction with \eqref{eq:wgh}.
\endproof

\begin{Corollary}\label{Liouvilleevenbis}
Let $a\in(-1,+\infty)$, $\varepsilon\geq0$, and let $w$ be a solution to
\begin{equation*}
\begin{cases}
-\mathrm{div}(\rho_\varepsilon^a(y)\nabla w)=0 & \mathrm{in \ }\mathbb{R}^{n+1}_+\\
\rho_\varepsilon^a\partial_yw=0 &\mathrm{in \ }\mathbb{R}^{n}\times\{0\},
\end{cases}
\end{equation*}
and let us suppose that for some $\gamma\in[0,1)$, $C>0$ there holds
\begin{equation}\label{gro0}
|w(z)|\leq C(1 + |z|^\gamma)
\end{equation}
for every $z$. Then $w$ is constant.
\end{Corollary}
\proof
Again, it is enough to treat the cases $\varepsilon\in\{0,1\}$. Let us assume $\varepsilon=1$, the case $\varepsilon=0$ being very similar. Then we have (by an even reflection across $\Sigma$) an even solution $w$  to
\begin{equation*}
-\mathrm{div}\left(\left(1+y^2\right)^{a/2}\nabla w\right)=0 \qquad\mathrm{in} \ \mathbb{R}^{n+1}.
\end{equation*}
Such a solution is $w\in H^1_{\mathrm{loc}}(\mathbb{R}^{n+1})$, with the growth condition \eqref{gro0}. Now we observe that, as $w$ is not constant with a sublinear growth at infinity,  $v=(1+y^2)^{a/2}\partial_y w$ can not be trivial, otherwise
$w$ would be globally harmonic and sublinear, in contradiction with the Liouville theorem in \cite{NorTavTerVer}. Hence, if $w$ is not constant, $v$  must be  an odd and nontrivial solution to
\begin{equation*}
\begin{cases}
-\mathrm{div}\left((1+y^2)^{-a/2}\nabla v\right)=0 &\mathrm{in \ }\mathbb{R}^{n+1}_+\\
v=0 &\mathrm{in} \ \{y=0\}.
\end{cases}
\end{equation*}
By the arguments in the proof of  Theorem \ref{Liouvilleeven}, we know that the weighted average of $v^2$ must satisfy a minimal growth rate as
\begin{equation*}\label{gro1}
 H(r)=\frac{1}{r^{n-a}}\int_{\partial^+ B_r^+}(1+y^2)^{-a/2}v^2\geq cr^{2\mu}, \qquad \mu\in( \gamma+a,1+a)\;,
\end{equation*}
for $r\geq r_0$ depending on $\mu$. Therefore, by integrating, we obtain
\[
\int_{B_r^+}\left(1+y^2\right)^{a/2} (\partial_y w)^2=\int_0^r\mathrm{d}t\int_{\partial^+ B_t^+}(1+y^2)^{-a/2}v^2\geq c r^{n-a+2\mu+1}\;.
\]
On the other hand, we have, by \eqref{gro0}
\begin{equation*}
\begin{split}
\int_{B_r^+}\left(1+y^2\right)^{a/2} (\partial_y w)^2\leq \int_{B_r^+}\left(1+y^2\right)^{a/2} |\nabla w|^2\\ \leq c  \int_{B_{2r}^+}\left(1+y^2\right)^{a/2} |w|^2\leq c(1+r^{n+1+a+2\gamma})
\end{split}
\end{equation*}
in contradiction with the previous inequality, since $\mu>\gamma+a$.

\endproof

\section{Local uniform bounds in H\"older spaces}\label{sec:UniformHolder}
This section is devoted to the proof of local bounds in H\"older spaces which hold uniformly with respect to the parameter of regularization $\varepsilon\geq0$, under homogeneous Neumann boundary condition at $\Sigma$. It is worthwhile noticing that, when the weight is $A_2$, that is in the case $a\in(-1,1)$, H\"older estimates are provided in the seminal paper \cite{FabKenSer} for \emph{all (even and odd) solutions}. In this framework, the optimal exponent (which is, by the way $\alpha\in(0,1)\cap(0,1-a)$) remains unclear and is related with the parameters of validity of the Poincar\'e inequality. Let us stress that our argument avoids the use of Poincar\'e inequality which is, by the way, false in the super degenerate case (see once more Example \ref{ex1}). We shall exploit a blow-up argument which, together with the Liouville theorems set out in the last section, will drive us to the proof. The exposition is aimed at introducing the most suitable technique for establishing the a priori $C^{1,\alpha}$-estimates of the next section.

Moreover, we wish to stress  that in the super degenerate range $a\geq1$, Example \ref{ex1} is telling us that the Neumann  boundary condition in Theorem \ref{holdereven} can not be removed, when seeking H\"older regularity (or even continuity). This is due to the fact that when the degeneracy is too strong, energy solutions may behave wildly. In fact, $\Sigma$ has vanishing capacity, as solutions like that of Proposition \ref{prop:super} are approximated in $H^{1,a}(B_1)$ by functions in $C^\infty_c(\overline{B_1}\setminus\Sigma)$ annihilating on a neighbourhood of the characteristic manifold .

\begin{Theorem}\label{holdereven}
Let $a\in(-1,+\infty)$ and as $\varepsilon\to0$ let $\{u_\varepsilon\}$ be a family of solutions in $ B_1^+$ of either
\begin{equation}\label{1odd}
-\mathrm{div}\left(\rho_\varepsilon^aA\nabla u_\varepsilon\right)=\rho_\varepsilon^af_\varepsilon 
\end{equation}
or
\begin{equation}\label{2odd}
-\mathrm{div}\left(\rho_\varepsilon^aA\nabla u_\varepsilon\right)=\mathrm{div}\left(\rho_\varepsilon^aF_\varepsilon\right)
\end{equation}
satisfying the Neumann boundary condition
\[
\rho_\varepsilon^a\partial_yu_\varepsilon=0\quad \mathrm{on \ }\partial^0B_1^+.
\]
Let $r\in(0,1)$, $\beta>1$, $p>\frac{n+1+a^+}{2}$, and $\alpha\in(0,1)\cap(0,2-\frac{n+1+a^+}{p}]$, respectively $p>n+1+a^+$ and $\alpha\in(0,1-\frac{n+1+a^+}{p}]$. There are constants depending on $a$, $n$, $\beta$, $p$, $\alpha$ and $r$ only such that
$$\|u_\varepsilon\|_{C^{0,\alpha}(B_r^+)}\leq c\left(\|u_\varepsilon\|_{L^\beta(B_1^+,\rho_\varepsilon^a(y)\mathrm{d}z)}+ \|f_\varepsilon\|_{L^p(B_1^+,\rho_\varepsilon^a(y)\mathrm{d}z)}\right).$$
and, respectively,
$$\|u_\varepsilon\|_{C^{0,\alpha}(B_r^+)}\leq c\left(\|u_\varepsilon\|_{L^\beta(B_1^+,\rho_\varepsilon^a(y)\mathrm{d}z)}+ \|F_\varepsilon\|_{L^p(B_1^+,\rho_\varepsilon^a(y)\mathrm{d}z)}\right).$$
\end{Theorem}

\proof
We start by proving the result in the case $A=\mathbb I$. For the general proof the reader can look at Remark \ref{rem:variable coefficients}.
The proof follows some ideas already developed in in \cite{NorTavTerVer,TerVerZil1,TerVerZil2}. We give a unique proof of the two points. Without loss of generality we can assume the existence of a uniform constant $c>0$ as $\varepsilon\to0$ such that
$$\|u_\varepsilon\|_{L^\beta(B_1^+,\rho_\varepsilon^a(y)\mathrm{d}z)}\leq c,\qquad \|f_\varepsilon\|_{L^p(B_1^+,\rho_\varepsilon^a(y)\mathrm{d}z)}\leq c,\qquad \|F_\varepsilon\|_{L^p(B_1^+,\rho_\varepsilon^a(y)\mathrm{d}z)}\leq c.$$
We argue by contradiction; that is, there exist $0<r<1$, $\alpha\in(0,1)\cap(0,2-\frac{n+1+a^+}{p}]$, respectively  $\alpha\in(0,1-\frac{n+1+a^+}{p}]$, and a sequence of solutions $\{u_k\}:=\{u_{\varepsilon_k}\}$ as $\varepsilon_k\to0$ to \eqref{1odd} or \eqref{2odd} respectively such that
$$\|\eta u_k\|_{C^{0,\alpha}(B_1^+)}\to+\infty,$$
where the function $\eta\in C^\infty_c(B_1)$ is a radial decreasing cut-off function such that $\eta\equiv1$ in $B_r$, $0\leq\eta\leq1$ in $B_1$ and $\mathrm{supp}(\eta)=B_{\frac{1+r}{2}}$. Moreover we can take $\eta\in\mathrm{Lip}(B_{\frac{1+r}{2}})$ such that $\eta(z)\leq \ell\mathrm{dist}(z,\partial B_{\frac{1+r}{2}})$. First we perform an even reflection across $\Sigma$ of the solutions $\{u_k\}$, the forcing terms $\{f_{\varepsilon_k}\}$ and the divergences of the fields $\{F_{\varepsilon_k}\}$. We remark that Proposition \ref{Moser} gives a uniform bound
\begin{equation*}\label{eq:linfinity}
\|u_k\|_{L^\infty(B_{(1+r)/2})}\leq c\;.
\end{equation*}
 Hence, we are supposing that
\begin{equation*}
\max_{\substack{z,\zeta \in B^+_{1}\\z\neq \zeta}}\frac{\abs{(\eta u_k)(z)-(\eta u_k)(\zeta)}}{\abs{z-\zeta}^\alpha}=L_k \to +\infty.
\end{equation*}
We can assume that $L_k$ is attained by a sequence $z_k,\zeta_k\in B^+=B_{\frac{1+r}{2}}\cap\{y\geq0\}$ and we call $r_k:=|z_k-\zeta_k|$.\\\\
One can easily show that 
\begin{itemize}
\item[$i)$] $r_k\to0$,
\item[$ii)$] $\ddfrac{\mathrm{dist}(z_k,\partial^+ B)}{r_k}\to+\infty$ and $\ddfrac{\mathrm{dist}(\zeta_k,\partial^+ B)}{r_k}\to+\infty$.
\end{itemize}
In facts, $r_k\to0$ since
$$+\infty\longleftarrow L_k=\frac{|(\eta u_k)(z_k)-(\eta u_k)(\zeta_k)|}{|z_k-\zeta_k|^\alpha}\leq\frac{\|u_k\|_{L^\infty(B)}}{r_k^\alpha}(\eta(z_k)+\eta(\zeta_k))\leq\frac{c}{r_k^\alpha}.$$
Moreover, using the Lipschitz continuity of the cut off function $\eta$
$$+\infty\longleftarrow\frac{L_k}{r_k^{1-\alpha}}\leq\frac{\|u_k\|_{L^\infty(B)}}{r_k}(\eta(z_k)+\eta(\zeta_k))\leq\frac{c\ell}{r_k}(\mathrm{dist}(z_k,\partial^+ B)+\mathrm{dist}(\zeta_k,\partial^+ B)).$$
This obviously implies that at least one of the two terms in the sum diverges. This implies $ii)$, since
$$+\infty\longleftarrow \frac{\mathrm{dist}(z_k,\partial^+ B)}{r_k}\leq 1+\frac{\mathrm{dist}(\zeta_k,\partial^+ B)}{r_k}.$$
Moreover let us define
\begin{equation*}
v_k(z)=\frac{(\eta u_k)(z_k+r_kz)-(\eta u_k)(z_k)}{L_kr_k^\alpha},\qquad w_k(z)=\frac{\eta(z_k)(u_k(z_k+r_kz)- u_k(z_k))}{L_kr_k^\alpha},
\end{equation*}
with
$$z\in B(k):=\frac{B-z_k}{r_k}.$$
We remark that $v_k$ and $w_k$ are symmetric with respect to $$\Sigma_k=\left\{(x,y) \ : \ y=-\frac{y_k}{r_k}\right\}.$$
We have that
\begin{equation}\label{alpha1}
\max_{\substack{z,\zeta \in B(k)\\z\neq \zeta}}\frac{\abs{v_k(z)-v_k(\zeta)}}{\abs{z-\zeta}^\alpha}=\left|v_k\left(0\right)-v_k\left(\frac{\zeta_k-z_k}{r_k}\right)\right|=1\;.
\end{equation}
We remark that since we have taken $z_k\in B$, then $0\in B(k)$ for any $k$. One easily sees that for any compact $K\subset\mathbb{R}^{n+1}$
\begin{equation}\label{samebehaviour}
\max_{z\in K\cap B(k)}|v_k(z)-w_k(z)|\to0,
\end{equation}
and since $v_k(0)=w_k(0)=0$ there exists a positive constant $c$, only depending on $K$ so that for any $z\in K$
\begin{equation*}
|v_k(z)|+|w_k(z)|\leq c.
\end{equation*}

Up to subsequences, the limit set $\lim_{k\to +\infty} B(k)$ is the whole of $\R^{n+1}$.  For any  compact set $K$  contained in $B(k)$ for $k$ large enough, the functions of the sequence $\{v_k\}$ have the same unitary $C^{0,\alpha}$-seminorm and they are uniformly bounded on $K$ since $v_k(0)=0$. Then, by the Ascoli-Arzel\'a theorem, we can extract a subsequence $v_k\to w$ uniformly with $w\in C^{0,\alpha}(K)$. By a countable compact exaustion of $\R^{n+1}$ we obtain uniform convergence of  $v_k\to w$ on compact sets with $w\in C^{0,\alpha}(\R^{n+1})$ and, by \eqref{samebehaviour}, the sequence $\{w_k\}$  converges to the same limit.  Note that the limit $w$ is a non constant globally $\alpha$-H\"older continuous function. In fact, up to pass to a subsequence, $\frac{\zeta_k-z_k}{r_k}\to\overline z\in S^{n}$ since any point of the sequence belongs to $S^{n}$. Hence, by \eqref{alpha1}, uniform convergence and equicontinuity, we have $|w(0)-w(\overline z)|=1$. In order to obtain a contradiction we will invoke the appropriate Liouville type theorem of Section \ref{sec:liouville}.

For any $z\in B(k)$, the functions $w_k$ solve

\begin{equation}\label{Lakw}
-\mbox{div}\!\left(\rho_{\varepsilon_k}^a\left(y_k+r_k\cdot\right)\nabla w_k\right)(z)=\frac{\eta(z_k)}{L_k}r_k^{2-\alpha}\rho_{\varepsilon_k}^a\left(y_k+r_ky\right)f_{\varepsilon_k}(z_k+r_kz), 
\end{equation}
or, respectively,
\begin{equation}\label{Lakwfks}
-\mbox{div}\!\left(\rho_{\varepsilon_k}^a\left(y_k+r_k\cdot\right)\nabla w_k\right)(z)=\frac{\eta(z_k)}{L_k}r_k^{1-\alpha}\mathrm{div}\left(\rho_{\varepsilon_k}^a\left(y_k+r_k\cdot_{n+1}\right)F_{\varepsilon_k}(z_k+r_k\cdot)\right)(z). 
\end{equation}

Let us define by $\Gamma_k=(\varepsilon_k,y_k,r_k)$ and by $\nu_k=|\Gamma_k|$ which is a sequence bounded from above. Let
$$\tilde\Gamma_k=\frac{\Gamma_k}{\nu_k}=(\tilde\varepsilon_k,\tilde y_k,\tilde r_k)=\left(\frac{\varepsilon_k}{\nu_k},\frac{y_k}{\nu_k},\frac{r_k}{\nu_k}\right).$$
Moreover, we denote $\tilde\rho_k(y)=\left(\tilde\varepsilon_k^2+(\tilde y_k+\tilde r_ky)^2\right)^{a/2}$. Hence, since $|\tilde\Gamma_k|=1$, up to consider a subsequence, we can extract subsequences possessing limits (the second can be infinite):
$$\tilde\Gamma_k\to\tilde\Gamma=(\tilde\varepsilon,\tilde y,\tilde r)\in S^2\qquad  \lim_{k\to+\infty}\frac{\tilde y_k}{\tilde r_k}=\tilde l\in[0,+\infty],$$
and we define $\tilde\rho(y)=\left(\tilde\varepsilon^2+(\tilde y+\tilde ry)^2\right)^{a/2}$. Therefore, let $\tilde\Sigma=\lim\Sigma_k$; that is,
\begin{equation*}
\tilde\Sigma=\begin{cases}
\{(x,y)\in\mathbb{R}^{n+1} \ : \ y=-\tilde l \} &\mathrm{if}\; \tilde l <+\infty,\\
\emptyset &\mathrm{if}\; \tilde l =+\infty.
\end{cases}
\end{equation*}
{\bf Claim.} The limit $w$ is an energy solution of
\begin{equation*}
-\mathrm{div}\left(\tilde\rho\nabla w\right)=0\qquad \mathrm{in} \ \mathbb{R}^{n+1},
\end{equation*}
in the sense that for every $\phi\in C^\infty_c(\R^{n+1})$
$$\int_{\R^{n+1}}\tilde\rho\nabla w\cdot\nabla\phi=0.$$
We notice that equations \eqref{Lakw} and \eqref{Lakwfks} can be normalized in $B(k)$ as
\begin{equation}\label{tildeLakw}
-\mbox{div}\!\left(\tilde\rho_k\nabla w_k\right)(z)=\nu_k^{-a}\frac{\eta(z_k)}{L_k}r_k^{2-\alpha} \rho_{\varepsilon_k}^a\left(y_k+r_ky\right)f_{\varepsilon_k}(z_k+r_kz), 
\end{equation}
\begin{equation}\label{tildeLakwfks}
-\mbox{div}\!\left(\tilde\rho_k\nabla w_k\right)(z)=\nu_k^{-a}\frac{\eta(z_k)}{L_k}r_k^{1-\alpha}\mathrm{div}\left(\rho_{\varepsilon_k}^a\left(y_k+r_k\cdot_{n+1}\right)F_{\varepsilon_k}(z_k+r_k\cdot)\right)(z). 
\end{equation}
We remark that the right hand sides in both cases are  $L^1$-vanishing on compact sets of $\mathbb{R}^{n+1}$. Indeed, let $\phi\in C^\infty_c(\mathbb{R}^{n+1})$: using the fact that for $k$ large enough $\mathrm{supp}(\phi)\subset B_R\subset B(k)$, using H\"older inequality, we have
\begin{eqnarray}
&&\left|\int_{B_R}\rho_{\varepsilon_k}^a\left(y_k+r_ky\right)f_{\varepsilon_k}(z_k+r_kz)\phi(z)\mathrm{d}z\right|\nonumber\\
&\leq&r_k^{-\frac{n+1}{p}}\|\phi\|_{L^\infty(B_R)}\left(\int_{B_{r_kR}(z_k)}\left(\varepsilon_k^2+\zeta_{n+1}^2\right)^{a/2}|f_{\varepsilon_k}(\zeta)|^p\mathrm{d}\zeta\right)^{1/p}\left(\int_{B_{R}}\rho_{\varepsilon_k}^a\left(y_k+r_ky\right)\mathrm{d}z\right)^{1/p'}\nonumber\\
&\leq&cr_k^{-\frac{n+1}{p}}\nu_k^{\frac{a}{p'}},\nonumber
\end{eqnarray}
and hence the right hand side converges to zero since $\alpha\leq2-\frac{n+1+a^+}{p}$, the fact that $0\leq r_k\leq \nu_k$ and having
$$\frac{\eta(z_k)}{L_k}r_k^{2-\alpha-\frac{n+1+a^+}{p}}\left(\frac{r_k^{a^+}}{\nu_k^a}\right)^{1/p}\to0.$$
Similarily,  in the second case, the right hand side can be estimated as
\begin{eqnarray}
&&\left|\int_{B_R}\mathrm{div}(\rho_{\varepsilon_k}^a\left(y_k+r_k\cdot_{n+1}\right)F_{\varepsilon_k}(z_k+r_k\cdot))(z)\phi(z)\mathrm{d}z\right|\nonumber\\
&\leq&r_k^{-\frac{n+1}{p}}\|\nabla\phi\|_{L^\infty(B_R)}\left(\int_{B_{r_kR}(z_k)}\left(\varepsilon_k^2+\zeta_{n+1}^2\right)^{a/2}\left|F_{\varepsilon_k}(\zeta)\right|^p\mathrm{d}\zeta\right)^{1/p}\nonumber\cdot\left(\int_{B_{R}}\rho_{\varepsilon_k}^a\left(y_k+r_ky\right)\mathrm{d}z\right)^{1/p'}\nonumber\\
&\leq&cr_k^{-\frac{n+1}{p}}\nu_k^{\frac{a}{p'}},\nonumber
\end{eqnarray}
and the right hand side itself converges to zero since $\alpha\leq1-\frac{n+1+a^+}{p}$, the fact that $0\leq r_k\leq\nu_k$ and having
$$\frac{\eta(z_k)}{L_k}r_k^{1-\alpha-\frac{n+1+a^+}{p}}\left(\frac{r_k^{a^+}}{\nu_k^a}\right)^{1/p}\to0.$$
Moreover, using the same reasoning in the proof of Lemma \ref{Good->Energy}, we can conclude that
$$\int_{\R^{n+1}}\tilde\rho_k\nabla w_k\cdot\nabla\phi\to\int_{\R^{n+1}}\tilde\rho\nabla w\cdot\nabla\phi,$$
with $w\in H^1_{\mathrm{loc}}(\R^{n+1},\tilde\rho(y)\mathrm{d}z)$ and symmetric with respect to $\tilde\Sigma$ (if $\tilde\Sigma\neq\emptyset$), proving the {\bf Claim}.\\\\
In order to complete the proof we invoke our Liouville type theorems of Section \ref{sec:liouville}. We distinguish different cases.
\begin{itemize}
\item[{\bf Case 1.}] If $\tilde r=0$ or $\tilde\Sigma=\emptyset$, hence we have proved that the limit $w\in H^1_{\mathrm{loc}}(\mathbb{R}^{n+1})$ is not constant and globally harmonic in $\mathbb{R}^{n+1}$. Moreover it is globally $C^{0,\alpha}(\mathbb{R}^{n+1})$ with $\alpha<1$, in clear contradiction with the Liouville theorem in Corollary 2.3 in \cite{NorTavTerVer}.
\item[{\bf Case 2.}] If $\tilde r\neq0$ and $\tilde\varepsilon\neq 0$, then, up to compose a dilation of $\frac{\tilde\varepsilon}{\tilde r}$ with a vertical translation of $-\frac{\tilde y}{\tilde r}$, we obtain convergence to an even (symmetric with respect to $\Sigma$)  solution to
\begin{equation*}
-\mathrm{div}\left((1+y^2)^{a/2}\nabla w\right)=0 \qquad\mathrm{in} \ \mathbb{R}^{n+1}.
\end{equation*}
Moreover $w\in H^1_{\mathrm{loc}}(\mathbb{R}^{n+1})$, globally $C^{0,\alpha}(\mathbb{R}^{n+1})$ and not constant. Since $w\in C^{0,\alpha}(\mathbb{R}^{n+1})$ with $\alpha<1$, then it has a bound on the growth at infinity given by \eqref{gro0}, contradicting Corollary \ref{Liouvilleevenbis}.
\item[{\bf Case 3.}] If $\tilde r\neq0$ and $\tilde\varepsilon=0$, then, up to a vertical translation of $-\frac{\tilde y}{\tilde r}$, we obtain convergence to an even (symmetric with respect to $\Sigma$)  solution to
\begin{equation*}
-\mathcal L_aw=0 \qquad\mathrm{in} \ \mathbb{R}^{n+1}.
\end{equation*}
Such a solution is $w\in H^{1,a}_{\mathrm{loc}}(\mathbb{R}^{n+1})$, globally $C^{0,\alpha}(\mathbb{R}^{n+1})$ and not constant, again in contradiction with Corollary \ref{Liouvilleevenbis}.
\end{itemize}
\endproof

\begin{remark}\label{rem:variable coefficients}
Theorem \ref{holdereven} can be easily extended, with no relevant changes in the proof, to the variable coefficients operators:
$$\mathrm{div}(\rho_\varepsilon^a A(x,y)\nabla u)\;,$$
where the matrix $A$ satisfies Assumption (HA); that is, is symmetric, $\Sigma$-invariant, continuous, and satisfy the uniform ellipticity condition $\lambda_1|\xi |^2 \leq A(x,y)\xi\cdot\xi \leq \lambda_2|\xi |^2$, for all $\xi\in\R^{n+1}$, for every $(x,y)\in B_1$ and some ellipticity constants $0<\lambda_1\leq\lambda_2$. Notice that even reflections across $\Sigma$, in presence of the matrix $A$, follow the symmetry rules given in Assumption (HA).
\end{remark}

\section{Local uniform bounds in $C^{1,\alpha}$ spaces}\label{sec:UniformC1}
In this section we show that we can ensure local uniform bounds also in $C^{1,\alpha}$-spaces.
\begin{Theorem}\label{C1alphaeven}
Let $a\in(-1,+\infty)$ and as $\varepsilon\to0$ let $\{u_\varepsilon\}$ be a family of solutions  to
\begin{equation*}\label{pb:even}
\begin{cases}
-\mathrm{div}\left(\rho_\varepsilon^aA\nabla u_\varepsilon\right)=\rho_\varepsilon^af_\varepsilon &\mathrm{in \ } B_1^+\\
\rho_\varepsilon^a\partial_yu_\varepsilon=0 &\mathrm{in \ }\partial^0B_1^+,
\end{cases}
\end{equation*}
Let $r\in(0,1)$, $\beta>1$ and $p>n+1+a^+$. Let also $A$ be $\alpha$-H\"older continuous with $\alpha\in(0,1-\frac{n+1+a^+}{p}]$. Then, there is a positive constant depending on $a$, $n$, $p$, $\beta$, $\alpha$ and $r$ only such that
$$\|u_\varepsilon\|_{C^{1,\alpha}(B_r^+)}\leq c\left(\|u_\varepsilon\|_{L^\beta(B_1^+,\rho_\varepsilon^a(y)\mathrm{d}z)}+ \|f_\varepsilon\|_{L^p(B_1^+,\rho_\varepsilon^a(y)\mathrm{d}z)}\right).$$
\end{Theorem}
\proof
We start by proving the result in the case $A=\mathbb I$. Then for the general case the reader can look at Remark \ref{rem:variable coefficientsC1alpha}. The proof of the following result follows some ideas contained in \cite{SoaTer10}.  Arguing by contradiction, we can assume that there exists a positive constant, uniform in $\varepsilon\to0$, such that
$$\|u_\varepsilon\|_{L^\beta(B_1^+,\rho_\varepsilon^a(y)\mathrm{d}z)}\leq c\qquad \|f_\varepsilon\|_{L^p(B_1^+,\rho_\varepsilon^a(y)\mathrm{d}z)}\leq c,$$
and that  there exist $0<r<1$, $\alpha\in(0,1-\frac{n+1+a^+}{p}]$ and a sequence of solutions $\{u_k\}=\{u_{\varepsilon_k}\}$ as $\varepsilon_k\to0$, such that $$\|\cdot\|_{C^{1,\alpha}}\to+\infty\;.$$ 
As we already know that the $u_k$'s are bounded in $L^\infty$ and in the energy space, we infer that the H\"older seminorm tends to infinity:
\begin{equation*}
\max_{j=1,...,n+1}\sup_{\substack{z,\zeta \in B^+_{1}\\z\neq \zeta}}\frac{\abs{\partial_j(\eta u_k)(z)-\partial_j(\eta u_k)(\zeta)}}{\abs{z-\zeta}^\alpha}=L_k \to +\infty,
\end{equation*}
where $\partial_j=\partial_{x_j}$ for any $j=1,...,n$ and $\partial_{n+1}=\partial_y$, and the function $\eta$ is a radial and decreasing cut off function such that $\eta\in C^\infty_c(B_1)$ with $0\leq\eta\leq1$, $\eta\equiv1$ in $B_r$ and $\mathrm{supp}(\eta)=B_{\frac{1+r}{2}}$. All functions will be extended by reflection through $\Sigma$ to the whole of the ball.  Moreover we take $\eta\in\mathrm{Lip}(B_{\frac{1+r}{2}})$ with $\partial_j\eta\in\mathrm{Lip}(B_{\frac{1+r}{2}})$ for any $j=1,...,n+1$, with the same constant $\ell$, that is $\eta(z)\leq \ell d(z,\partial B_{\frac{1+r}{2}})$ and $\partial_j\eta(z)\leq \ell d(z,\partial B_{\frac{1+r}{2}})$.
Up to a subsequence, there exists $i\in\{1,...,n+1\}$, and two sequences of points $z_k,\zeta_k$ in $B=B_{\frac{1+r}{2}}$ such that
\begin{equation*}
\frac{\abs{\partial_i(\eta u_k)(z_k)-\partial_i(\eta u_k)(\zeta_k)}}{\abs{z_k-\zeta_k}^\alpha}=L_k.
\end{equation*}
We remark that it is not possible that the H\"older seminorms of the sequence of derivatives $\partial_i(\eta u_k)$ stay bounded while their $L^\infty$-norms explode. This would be in contradiction with the uniform energy bound of the sequence $u_k$. We define $r_k=|z_k-\zeta_k|\in[0,\mathrm{diam}(B_1)]$. Hence, up to a subsequence, we have $r_k\to\overline r\in[0,2]$. Now we want to define two blow up sequences: let $\hat{z}_k\in B$ to be specified below and
\begin{equation*}
v_k(z)=\frac{\eta(\hat{z}_k+r_kz)}{L_kr_k^{1+\alpha}}\left(u_k(\hat{z}_k+r_kz)-u_k(\hat{z}_k)\right),\qquad w_k(z)=\frac{\eta(\hat{z}_k)}{L_kr_k^{1+\alpha}}\left(u_k(\hat{z}_k+r_kz)-u_k(\hat{z}_k)\right),
\end{equation*}
for $z\in B(k):=\frac{B-\hat{z}_k}{r_k}$. Let $B^\infty=\lim_{k\to+\infty}B(k)$.\\\\ 
There are two possibilities:
\begin{itemize}
\item[\bf{Case \ 1:}] \ $\frac{d(z_k,\Sigma)}{r_k}\to+\infty$.
In this case, since the sequence $\{z_k\}$ is taken in a bounded set, one has $r_k\to0$. We define $\hat{z}_k=z_k$. Hence, $B^\infty=\mathbb{R}^{n+1}$.
\item[\bf{Case \ 2:}] \ $\frac{d(z_k,\Sigma)}{r_k}\leq c$ uniformly in $k$. In this case of course also $\frac{d(\zeta_k,\Sigma)}{r_k}\leq c$ and we choose $\hat{z}_k=(x_k,0)$ to be the projection on $\Sigma$ of $z_k$, where $z_k=(x_k,y_k)$.
\end{itemize}

To continue, we make some preliminary considerations holding in both cases. Since $\hat{z}_k\in B$, then the point $0\in B(k)$ for any $k$. Moreover, fixing $K$ a compact subset of $B^\infty$, then $K\subset B(k)$ definitely. Hence, for any $z,\zeta\in K$,
\begin{eqnarray*}
|\partial_iv_k(z)-\partial_iv_k(\zeta)|&\leq&\frac{1}{L_kr_k^\alpha}|\partial_i(\eta u_k)(\hat{z}_k+r_kz)-\partial_i(\eta u_k)(\hat{z}_k+r_k\zeta)|\\
&&+\frac{|u_k(\hat{z}_k)|}{L_kr_k^\alpha}|\partial_i\eta(\hat{z}_k+r_kz)-\partial_i\eta(\hat{z}_k+r_k\zeta)|\\
&\leq&|z-\zeta|^\alpha+\frac{|u_k(\hat{z}_k)|}{L_k}r_k^{1-\alpha}\ell|z-\zeta|,
\end{eqnarray*}
using the Lipschitz continuity of the partial derivative of $\eta$. Since $\alpha<1$, $r_k\to\overline r\in[0,2]$, $L_k\to+\infty$ and $\|u_k\|_{L^\infty(B)}\leq c$ uniformly in $k$ (Theorem \ref{holdereven}), then we can make
$$\frac{|u_k(\hat{z}_k)|}{L_k}r_k^{1-\alpha}\ell\sup_{z,\zeta\in K}|z-\zeta|^{1-\alpha}\leq 1.$$
Hence, fixing $K\subset B^\infty$ a compact set, there exists $\overline k$ such that for any $k>\overline k$,
\begin{equation}\label{abovec1}
\sup_{\substack{z,\zeta \in K\\z\neq \zeta}}\frac{\abs{\partial_iv_k(z)-\partial_iv_k(\zeta)}}{\abs{z-\zeta}^\alpha}\leq 2.
\end{equation}
Obviously, condition \eqref{abovec1} holds true for any partial derivative of $v_k$. Moreover, as $k\to+\infty$,
\begin{eqnarray}\label{belowc1}
&&\left|\partial_iv_k\left(\frac{z_k-\hat{z}_k}{r_k}\right)-\partial_iv_k\left(\frac{\zeta_k-\hat{z}_k}{r_k}\right)\right|\\&=&\left|\frac{1}{L_kr_k^\alpha}(\partial_i(\eta u_k)(\hat{z}_k+r_kz)-\partial_i(\eta u_k)(\hat{z}_k+r_k\zeta))\right.\nonumber\left.+\frac{u_k(\hat{z}_k)}{L_kr_k^\alpha}(\partial_i\eta(\hat{z}_k+r_k\zeta)-\partial_i\eta(\hat{z}_k+r_kz))\right|\nonumber\\
&=&1+O\left(\frac{|u_k(\hat{z}_k)|}{L_k}r_k^{1-\alpha}\ell\right)\nonumber=\left|\frac{z_k-\hat{z}_k}{r_k}-\frac{\zeta_k-\hat{z}_k}{r_k}\right|^\alpha+o(1).
\end{eqnarray}
Hence, fixing a compact subset of $B^\infty$, by \eqref{abovec1} and \eqref{belowc1} we have the following bound from above and below for the H\"older seminorms
$$1\leq[\partial_iv_k]_{C^{0,\alpha}(K)}\leq 2.$$
In $\bf{Case \ 1}$, let us now define for any $z\in B(k)$
$$\overline v_k(z)=v_k(z)-\nabla v_k(0)\cdot z,\qquad\overline w_k(z)=w_k(z)-\nabla w_k(0)\cdot z.$$
In $\bf{Case \ 2}$, let us now define for any $z=(x,y)\in B(k)$
$$\overline v_k(z)=v_k(z)-\nabla_x v_k(0)\cdot x,\qquad\overline w_k(z)=w_k(z)-\nabla_x w_k(0)\cdot x.$$
We can see that in both cases $\overline v_k(0)=\overline w_k(0)=0$ since $v_k(0)=w_k(0)=0$. Moreover $|\nabla\overline v_k|(0)=|\nabla\overline w_k|(0)=0$:  indeed,
in $\bf{Case \ 1}$  this is due to the fact that for any $j\in\{1,...,n+1\}$ we have 
\begin{equation}\label{partialover}
\partial_j\overline v_k(z)=\partial_jv_k(z)-\partial_jv_k(0)\qquad \mathrm{and}\qquad \partial_j\overline w_k(z)=\partial_jw_k(z)-\partial_jw_k(0). 
\end{equation}
Turning to $\bf{Case \ 2}$, then \eqref{partialover} holds for any $j\in\{1,...,n\}$, while $\partial_y\overline v_k=\partial_y v_k$ and $\partial_y\overline w_k=\partial_y w_k$. Using the fact that $\partial_yu_k(\hat{z}_k)=0$ since $\hat{z}_k=(x_k,0)\in\partial^0B_1^+$, then $\partial_y v_k(0)=\partial_y w_k(0)=0$.\\\\
Obviously, we have also that the H\"older seminorm $[\partial_j\overline v_k]_{C^{0,\alpha}(K)}=[\partial_jv_k]_{C^{0,\alpha}(K)}$ for any compact $K\subset B$ and any $j=1,...,n+1$. By the Ascoli-Arzel\'a theorem and compact embeddings, $\overline v_k\to\overline v$ in $C^{1,\gamma}_{\mathrm{loc}}(B^\infty)$ for any $\gamma\in(0,\alpha)$. Nevertheless, the limit $\overline v$ belongs to $C^{1,\alpha}(B^\infty)$ with $[\partial_i\overline v]_{C^{0,\alpha}(K)}\leq 2$ in any compact subset of $B^\infty$ (passing to the limit in \eqref{abovec1}).

Eventually, we work with the sequences of points
$$\frac{z_k-\hat{z}_k}{r_k},\frac{\zeta_k-\hat{z}_k}{r_k}\in B(k).$$
In $\bf{Case \ 1}$, they are respectively the constant sequence $0$ and the sequence $\frac{\zeta_k-z_k}{r_k}$ of points lying on the sphere $S^n$. Hence, up to subsequences, they converge to the couple of points $z_1=0$ and $z_2\in S^n$.

In $\bf{Case \ 2}$, the first sequence is in fact $\frac{d(z_k,\Sigma)}{r_k}e_{n+1}$ which lies on a bounded segment $\mathcal{R}=\{(0,y):y\in[0,\tilde R]\}$. The second sequence can be seen as
$$\frac{\zeta_k-\hat{z}_k}{r_k}=\frac{\zeta_k-z_k}{r_k}+\frac{d(z_k,\Sigma)}{r_k}e_{n+1},$$
that is, the sum of a sequence on the sphere $S^n$ and one on the segment $\mathcal{R}$. Hence, up  to subsequences, they converges respectively to a pair of points $z_1$ and $z_2$.

In both cases there exists a compact subset $K$ of $B^\infty$ such that $z_1,z_2\in K$. By local $C^1$ convergence, passing to the limit in \eqref{belowc1}, we infer that $|\partial_i \overline v(z_1)-\partial_i \overline v(z_2)|=1$ which means that $\overline v$ has a non constant gradient.

Next we show that also in $\bf{Case \ 2}$ the sequence $r_k\to0$. Seeking a contradiction, let us suppose that $r_k\to\overline r>0$. Hence,
$$\sup_{z\in B(k)}|v_k(z)|\leq\frac{2\|\eta\|_{L^\infty(B_1)}\|u_k\|_{L^\infty(B)}}{r_k^{1+\alpha}L_k}\leq\frac{c}{\overline r^{1+\alpha}L_k}\to0,$$
which means that $v_k\to0$ uniformly on compact subsets of $B^\infty$. This fact implies also that pointwisely in $B^\infty$
$$\overline v(z)=\lim_{k\to+\infty}\nabla_x v_k(0)\cdot x.$$
Since $0\in B(k)$ for any $k$, it is easy to see that $B^\infty$ contains balls $B_R$, for a small enough radius $R>0$. If the sequence $\{\partial_jv_k(0)\}$ were unbounded at least for $j=1,...,n$, then
$$|\overline v(Re_j)|=R\lim_{k\to+\infty}|\nabla v_k(0)\cdot e_j|=+\infty,$$
which is in contradiction with the fact that $\overline v\in C^{1,\alpha}(B_R)$ and hence bounded. Hence, $\{\nabla_x v_k(0)\}$ is a bounded sequence, and up to consider a subsequence, it converges to a vector $\nu\in\mathbb{R}^{n}$ and $\overline v(z)=\nu\cdot x$, which is in contradiction with the fact that $\overline v$ has non constant gradient.

Hence, we end up with $B^\infty=\mathbb{R}^{n+1}$ also in $\bf{Case \ 2}$.\\

Now we want to show that the sequences $\{\overline v_k\}, \{\overline w_k\}$ have the same asymptotic behaviour on compact subsets of $\mathbb{R}^{n+1}$; that is, fixing a compact subset $K\subset \mathbb{R}^{n+1}$, definitively it is contained in $B(k)$ and, since $\nabla v_k(0)=\frac{\eta(\hat{z}_k)}{L_kr_k^\alpha}\nabla u_k(\hat{z}_k)=\nabla w_k(0)$, then
\begin{eqnarray*}
|\overline w_k(z)-\overline v_k(z)|&=&|w_k(z)-v_k(z)|\\
&=&\frac{1}{L_kr_k^{1+\alpha}}|\eta(\hat{z}_k+r_kz)-\eta(\hat{z}_k)|\cdot |u_k(\hat{z}_k+r_kz)-u_k(\hat{z}_k)|\\
&\leq&\frac{c}{L_kr_k^{1+\alpha}}\cdot r_k|z|\cdot r_k^\alpha|z|^\alpha\leq\frac{c(K)}{L_k}\to0,
\end{eqnarray*}
which implies, in particular, boundedness in $L^\infty_{\mathrm{loc}}$ of the $\overline w_k$'s. Now we look at the equations fullfilled by them in $B(k)$, that is,
\begin{equation*}
-\mbox{div}\!\left(
\rho_{\varepsilon_k}^a\left(y_k+r_ky\right)\nabla \overline w_k\right)=\frac{\eta(\hat{z}_k)}{L_k}r_k^{1-\alpha}\rho_{\varepsilon_k}^a\left(y_k+r_ky\right)f_{\varepsilon_k}(\hat{z}_k+r_kz)+ \partial_y (\rho_{\varepsilon_k}^a\left(y_k+r_k\cdot\right))\partial_yw_k(0) 
\end{equation*}
where $y_k/r_k\to +\infty$ in {\bf Case 1} and $y_k\equiv 0$ in {\bf Case 2}.  Keeping in mind the notation in the proof of Theorem \ref{holdereven},  we can write the above equation as
\begin{equation}\label{eq:wk}
-\mbox{div}\!\left(\tilde
\rho_{k}\nabla \overline w_k\right)=\underbrace{\nu_k^{-a}\frac{\eta(\hat{z}_k)}{L_k}r_k^{1-\alpha}\rho_{\varepsilon_k}^a\left(y_k+r_ky\right)f_{\varepsilon_k}(\hat{z}_k+r_kz)}_{(i)}+ \underbrace{\nu_k^{-a}\partial_y (\rho_{\varepsilon_k}^a\left(y_k+r_k\cdot\right))\partial_yw_k(0)}_{(ii)} 
\end{equation}
Arguing as in the proof of the {\bf Claim} in Theorem \ref{holdereven}, one easily sees that the terms $(i)$ converge to zero in $L^1_{\mathrm{loc}}$. As to the term $(ii)$, note that, since $\partial_yw_k(0)=\partial_yv_k(0)$, we have $\partial_yw_k(0)\equiv 0$ in {\bf Case 2}, whereas, in  {\bf Case 1} we have, using \eqref{abovec1} (which holds true for any partial derivative of $v_k$),
$$|\partial_yv_k(0)|=|\partial_yv_k(0)-\partial_yv_k((0,-y_k/r_k))|\leq 2\frac{y_k^\alpha}{r_k^\alpha}= 2\frac{\tilde y_k^\alpha}{\tilde r_k^\alpha}.$$

Hence, since in  {\bf Case 1} we have $\tilde r_k\to 0$, we find
\[
\left|(ii)\right|=a\tilde\rho_k(y) \frac{1+\tilde r_k y}{\tilde\varepsilon_k^2+(\tilde y_k+\tilde r_ky)^2}\tilde r_k\left|\partial_yv_k(0)\right|\leq c \tilde r_k^{1-\alpha}\tilde y_k^\alpha\to 0\;.
\]
Thus in both cases we infer that the full right hand side of \eqref{eq:wk} converges to zero in $L^1_{\mathrm{loc}}$. Since we already know that the $\overline w_k$'s are uniformly bounded in $L^\infty_{\mathrm{loc}}$ and converge uniformly to $\overline v$, arguing as in the proof of Theorem \ref{holdereven}, 
we deduce that $\overline v=\overline w $ is a global energy solution to 
\begin{equation*}\label{eq:v}
-\mbox{div}\!\left(\tilde
\rho \nabla \overline v (z)\right)=0\;\qquad \mbox{in}\;\R^{n+1}\;.
\end{equation*}
Moreover at least one of its partial derivatives $\partial_i\overline v$ is non constant, while all are globally $C^{0,\alpha}(\mathbb{R}^{n+1})$ with $\alpha<1$. 
 
\begin{itemize}
\item[{\bf Case 1:}]  In this case $\tilde\rho\equiv 1$ and at least one of the partial derivatives $\partial_i\overline v$ is a non constant and globally $C^{0,\alpha}(\mathbb{R}^{n+1})$ harmonic function with $\alpha<1$, in contradiction by the Liouville theorem in Corollary 2.3 in \cite{NorTavTerVer}.
\item[\bf{Case 2:}] In this case $\hat{z}_k=(x_k,0)$ and $\overline v$ is even in the variable $y$.
According with the possible limits of the normalized sequence  $\nu=(\tilde\varepsilon,0,\tilde r)$, we have three possibilties: at first, if $\tilde\rho\equiv 1$, then we conclude exactly as in {\bf Case 1}. If either $\tilde \rho(y)\equiv \left(\tilde\varepsilon^2+\tilde r^2y^2\right)^{a/2}$, for some $\tilde\varepsilon\neq 0$ and $\tilde r\neq 0$, or $\tilde\rho(y)\equiv |y|^{a}$ then, we invoke Theorem \ref{Liouvilleeven}. By rescaling the equation we reduce to the case $\tilde\rho(y)\equiv \left(\varepsilon^2+y^2\right)^{a/2}$, with $\varepsilon\in\{0,1\}$. Looking at the weighted derivative $\tilde\rho\partial_y\overline v$, it is immediate to check that it is a solution to \eqref{eq:Liouvilleeven} (replacing $a$ with $-a$), with $\varepsilon\in\{0,1\}$, and with the global bound
\begin{equation*}
|\tilde\rho(y)\partial_y\overline v(z)|\leq C\tilde\rho(y)(1 + |z|^\alpha),
\end{equation*}
with $\alpha<1$. Hence, by Theorem \ref{Liouvilleeven}, $\partial_y\overline v\equiv 0$ and we proceed again as in {\bf Case 1}.

\end{itemize}
\endproof

\begin{Theorem}\label{C1alphaevendiv}
Let $a\in(-1,+\infty)$ and as $\varepsilon\to0$ let $\{u_\varepsilon\}$ be a family of solutions to
\begin{equation*}
\begin{cases}
-\mathrm{div}\left(\rho_\varepsilon^aA\nabla u_\varepsilon\right)=\mathrm{div}\left(\rho_\varepsilon^aF_\varepsilon\right) &\mathrm{in \ } B_1^+\\
\rho_\varepsilon^a\partial_yu_\varepsilon=0 &\mathrm{in \ }\partial^0B_1^+,
\end{cases}
\end{equation*}
with $F_\varepsilon=(f_\varepsilon^1,...,f_\varepsilon^{n+1})$ with the y-component vanishing on $\Sigma$: $f_\varepsilon^{n+1}(x,0)=f_\varepsilon^y(x,0)=0$ in $\partial^0B_1^+$. 
Let also $A$ be $\alpha$-H\"older continuous with $\alpha\in(0,1)$. Then, for any $r\in(0,1)$ and $\beta>1$, there exists a positive constant $c$ depending on $a$, $n$, $\beta$, $\alpha$ and $r$ only, such that
$$\|u_\varepsilon\|_{C^{1,\alpha}(B_r^+)}\leq c\left(\|u_\varepsilon\|_{L^\beta(B_1^+,\rho_\varepsilon^a(y)\mathrm{d}z)}+ \|F_\varepsilon\|_{C^{0,\alpha}(B_1^+)}\right).$$
\end{Theorem}
\proof
We start by proving the result in the case $A=\mathbb I$. Then for the general case the reader can look at Remark \ref{rem:variable coefficientsC1alpha}. In order to prove the claim we argue by contradiction and we follow the same steps of the proof of Theorem \ref{C1alphaeven}. With no loss of generality, we may assume that there exists a positive constant uniform in $\varepsilon\to0$ such that
$$\|u_\varepsilon\|_{L^\beta(B_1^+,\rho_\varepsilon^a(y)\mathrm{d}z)}\leq c\qquad \|F_\varepsilon\|_{C^{0,\alpha}(B_1^+)}\leq c\;.$$
From now on, all functions $u_\varepsilon$, $(f_\varepsilon^1,...,f_\varepsilon^{n})$ will be extended in an even manner through the characteristic hyperplane $\Sigma$, while $f_\varepsilon^{n+1}$ is extended as an odd-in-y function.

We first remark that fields $F_\varepsilon$ satisfy the conditions in Theorem \ref{holdereven} with any $p\in(0,+\infty)$, and hence we have an initial bound in $C^{0,\beta}$; that is, for any $\beta\in(0,1)$ and any $0<r<1$
$$\|u_\varepsilon\|_{C^{0,\beta}(B_r^+)}\leq c.$$
Hence, we follow the very same reasonings in Theorem \ref{C1alphaeven} (contradiction argument and blow-up procedure). We want to show that also in this setting, the equations satisfied by the limit of the blow-ups are the same. Following this path, and keeping the same notation, we consider 
the equation solved by the scaled variables $\overline w_k$,  in analogy with \eqref{eq:wk}, which now reads
\begin{equation}\label{eq:wdivk}
-\mbox{div}\!\left(\tilde
\rho_{k}\nabla \overline w_k\right)=\underbrace{\frac{\eta(\hat{z}_k)}{L_kr_k^{\alpha}}\mbox{div}\left(\tilde\rho_kF_{\varepsilon_k}(\hat{z}_k+r_k\cdot)\right)}_{(i)}+ \underbrace{\nu_k^{-a}\partial_y (\rho_{\varepsilon_k}^a\left(y_k+r_k\cdot\right))\partial_yw_k(0)}_{(ii)} 
\end{equation}
As in the proof of Theorem \ref{C1alphaeven}, we distinguish the two cases. In both of them, the term $(ii)$  has not changed and either vanishes identically or converges to zero uniformly on compact sets. 
Next we examine the term $(i)$ and we prove that its limit vanishes in the distributional sense.
\begin{itemize}
\item[\bf{Case \ 1:}] \ $\frac{d(z_k,\Sigma)}{r_k}\to+\infty$, that is $r_k/y_k\to0$. Again, we choose $\hat{z}_k=z_k$.  and  we have $\tilde\rho_k\to 1$. Let us consider $\phi\in C^\infty_c(\mathbb{R}^{n+1})$. For $k$ large enough $\mathrm{supp}(\phi)\subset B_R\subset B(k)$. Hence,
\begin{eqnarray*}
&&\frac{\eta(z_k)}{L_k}r_k^{-\alpha}\int_{B_R}\mbox{div}\!\left(\tilde \rho_kF_{\varepsilon_k}(z_k+r_k\cdot)\right)(z)\phi(z)\mathrm{d}z\\
&=&\frac{\eta(z_k)}{L_k}r_k^{-\alpha}\int_{B_R}\tilde \rho_k\left(F_{\varepsilon_k}(z_k+r_kz)-F_{\varepsilon_k}(z_k)\right)\cdot\nabla\phi(z)\mathrm{d}z+\frac{\eta(z_k)}{L_k}r_k^{-\alpha}\int_{B_R}\left(\tilde \rho_k-1\right)F_{\varepsilon_k}(z_k)\cdot\nabla\phi(z)\mathrm{d}z\\
&=&\frac{\eta(z_k)}{L_k}r_k^{-\alpha}\int_{B_R}\tilde \rho_k\left(F_{\varepsilon_k}(z_k+r_kz)-F_{\varepsilon_k}(z_k)\right)\cdot\nabla\phi(z)\mathrm{d}z+\frac{\eta(z_k)}{L_k}r_k^{-\alpha}\int_{B_R}\left(\tilde \rho_k-1\right)f^y_{\varepsilon_k}(z_k)\partial_y\phi(z)\mathrm{d}z\\
&\leq&c \frac{\eta(z_k)}{L_k}\|\nabla\phi\|_{L^\infty(B_R)}+c \frac{\eta(z_k)}{L_k}\|\partial_y\phi\|_{L^\infty(B_R)}\left(\frac{r_k}{y_k}\right)^{1-\alpha}\to0,
\end{eqnarray*}

where in the last inequality we have used the uniform $C^{0,\alpha}$-regularity of $F_{\varepsilon_k}$, the fact that $f^{y}_{\varepsilon_k}(x_k,0)=0$ since they are odd in $y$, and the fact that definitively
\begin{equation*}
\begin{split}
r_k^{-\alpha}(\tilde\rho_k-1)|f^y_{\varepsilon_k}(z_k)|=r_k^{-\alpha}\left[\left(1+\frac{r_ky(r_ky+2y_k)}{\varepsilon_k^2+y_k^2+r_k^2}y\right)^{a/2}-1\right]|f^y_{\varepsilon_k}(z_k)|\\\leq c\frac{r_ky_ky}{\varepsilon_k^2+y_k^2+r_k^2}\left(\frac{y_k}{r_k}\right)^\alpha\leq c\left(\frac{r_k}{y_k}\right)^{1-\alpha}.
\end{split}
\end{equation*}
Moreover we have also used the fact that for any $t\in[-R,R]$
\begin{equation*}\label{divx0}
\int_{B_R\cap\{y=t\}}\sum_{i=1}^nf_{\varepsilon_k}^i(z_k)\partial_{x_i}\phi \ \mathrm{d}x=\int_{B_R\cap\{y=t\}}\mathrm{div}_x\left(F_{\varepsilon_k}(z_k)\phi\right)\mathrm{d}x=0.
\end{equation*}
\item[\bf{Case 2:}] \ $\frac{d(z_k,\Sigma)}{r_k}\leq c$ uniformly in $k$. Then we choose $\hat{z}_k=(x_k,0)$ where $z_k=(x_k,y_k)$. The sequence $\overline w_k$ solves the equation \eqref{eq:wdivk} in $B(k)$, where the term$(ii)$ vanishes. Hence, testing with $\phi\in C^\infty_c(\mathbb{R}^{n+1})$. For $k$ large enough $\mathrm{supp}(\phi)\subset B_R\subset B(k)$. Hence we have
\begin{eqnarray*}
&&\frac{\eta(z_k)}{L_k}r_k^{-\alpha}\int_{B_R}\mbox{div}\!\left(\tilde \rho_kF_{\varepsilon_k}(z_k+r_k\cdot)\right)(z)\phi(z)\mathrm{d}z\\
&=&\frac{\eta(z_k)}{L_k}r_k^{-\alpha}\int_{B_R}\tilde \rho_k\left(F_{\varepsilon_k}(z_k+r_kz)-F_{\varepsilon_k}(z_k)\right)\cdot\nabla\phi(z)\mathrm{d}z
\\
&\leq&c \frac{\eta(z_k)}{L_k}\|\nabla\phi\|_{L^\infty(B_R)}\to0\;.
\end{eqnarray*}
\end{itemize}
The remaining part of the proof then follows exactly as in Theorem \ref{C1alphaeven}.
\endproof

\begin{remark}[The variable coefficients case]\label{rem:variable coefficientsC1alpha}
Theorems \ref{C1alphaeven} and \ref{C1alphaevendiv} can be extended, with a few but relevant changes in the proof, to the variable coefficients operators:
$$\mathrm{div}(\rho_\varepsilon^a A(x,y)\nabla u)\;,$$
when the matrix $A$ satisfies Assumption (HA); that is, is symmetric, continuous, $\Sigma$ is $A$-invariant and, as usual, satisfying the uniform ellipticity condition $\lambda_1|\xi |^2 \leq A(x,y)\xi\cdot\xi \leq \lambda_2|\xi |^2$, for all $\xi\in\R^{n+1}$, for every $(x,y)\in B_1$ and some ellipticity constants $0<\lambda_1\leq\lambda_2$.  We need also $\alpha$-H\"older continuity of $A$. Let us highlight here the modifications needed in the proof of the two Theorems, taking into account that the main novelty is that the linear functions of $x$ are not exact solutions in the case of variable coefficients when discussing {\bf Case 2}, nor, even more so, are the linear functions of $(x,y)$ in the {\bf Case 1}. 
Below we show how to modify the argument in Case 2, being the first case similar {\it mutatis mutandis} and somewhat easier. When writing the analogue of \eqref{eq:wk}, we find an extra term of the form
\[
\mathrm{div}\left(\tilde\rho_k \left(A(\hat z_k+r_kz)-A(\hat z_k)\right)\nabla w_k(z)\right)\;,
\]
that we wish to show be vanishing, in the limit, in the sense of distributions. To this end, we observe that   
\begin{equation}\label{eq:Aholderiana}
\tilde\rho_k (y)\left|A(\hat z_k+r_kz)-A(\hat z_k)\right| \left|\nabla w_k(z)\right|\leq C \tilde\rho_k (y) r_k^\alpha \left|\nabla w_k(z)\right|\;,
\end{equation}
and we show that the right hand side tends to zero in $L^1_{\mathrm{loc}}$. Notice that, was the sequence $\Vert\nabla u_k\Vert_\infty$ uniformly bounded, we would have
\[
\left|\nabla w_k(z)\right|=\dfrac{\eta(\hat z_k)\left|\nabla u_k( z_k+r_kz)\right|}{L_kr_k^\alpha}\leq \dfrac{c}{L_kr_k^\alpha}\;,
\]
 from which we would promptly conclude. At this stage, however, we only know that the $u_k$'s are uniformly bounded in H\"older spaces, for every exponent $\beta\in(0,1)$. Hence we proceed in two steps. First, we prove the desired $C^{1,\alpha^\prime}$ bounds for some $\alpha^\prime<\alpha$, in order to obtain uniform boundedness of $\Vert\nabla u_k\Vert_\infty$. From this,  the $C^{1,\alpha}$ estimate easily follows, as highlighted above.

By testing the original equation by $\varphi_r^2(z)(u_k(z)-u_k(c_k))$, where $c_k$ are points in $B_{1/2}$, $\varphi_r(z)=\varphi(rz)$, $\varphi$ being a cut-off function such that $\varphi\equiv 1$ in $B_1(c_k)$ and $\varphi\equiv 0$ on $\R^{n+1}\setminus B_2(c_k)$,  we obtain, for any given ball of radius $0<r<1/4$, the identities
\begin{eqnarray*}
&&\int_{\R^{n+1}} \rho_k A(x,y)\nabla (\varphi_r (u_k(z)-u_k(c_k)))\cdot \nabla (\varphi_r (u_k(z)-u_k(c_k)))\\
&=&\int_{\R^{n+1}} \rho_k f\varphi_r^2(u_k(z)-u_k(c_k))+\int_{\R^{n+1}} \rho_kA(x,y)\nabla (\varphi_r)\cdot\nabla(\varphi_r)(u_k(z)-u_k(c_k))^2\;,
\end{eqnarray*}
 and, respectively, 
 \begin{eqnarray*}
&&\int_{\R^{n+1}} \rho_k A(x,y)\nabla (\varphi_r (u_k(z)-u_k(c_k)))\cdot \nabla (\varphi_r (u_k(z)-u_k(c_k)))\\
&=&\int_{\R^{n+1}} \rho_k F\cdot\nabla (\varphi_r^2u_k)+\int_{\R^{n+1}} \rho_kA(x,y)\nabla (\varphi_r)\cdot\nabla(\varphi_r)(u_k(z)-u_k(c_k))^2\;,
\end{eqnarray*}
In both cases, by the uniform ellipticity of the matrix $A$, after standard computations, using that $u_k$ are $\beta$-H\"older continuous and $|\nabla \varphi_r|\leq c/r$, we obtain the estimate
\begin{equation}\label{stimagradiente}
 \int_{B_{r}(c_k)}  \rho_k|\nabla u_k|^2\leq \dfrac{c}{r^{2(1-\beta)}} \int_{B_{2r}(c_k)} \rho_k \;,
 \end{equation}
 yielding
  \[
 \int_{B_{R}}  \rho_k(y_k+r_ky)|\nabla w_k|^2\leq \dfrac{c\eta^2(\hat z_k)}{L_kr_k^{2(1-\beta+\alpha^\prime)}} \int_{B_{2R}} \rho_k \;,
 \]
 and finally
   \[
 \int_{B_{R}}  \tilde\rho_k|\nabla w_k|\leq  \left(\int_{B_{R}}  \tilde\rho_k|\nabla w_k|^2\right)^{1/2}\left( \int_{B_{R}} \tilde\rho_k\right)^{1/2}\leq \dfrac{c\eta(\hat z_k)}{\sqrt{L_k}r_k^{1-\beta+\alpha^\prime}} \int_{B_{2R}} \tilde\rho_k \;.
 \]
 With this, choosing $\alpha^\prime$ and $\beta$ such that  $\alpha^\prime\leq\alpha +\beta-1$, inserting into \eqref{eq:Aholderiana} we conclude the vanishing in $L^1_{\mathrm{loc}}$, as desired.
\end{remark}
Eventually, we wish to do the following remark, which shows that uniform bounds in $C^{1,\alpha}$-spaces with respect to this kind of regularization are optimal.
\begin{remark}
It is not possible to obtain $C^{2,\alpha}$ uniform estimates across $\Sigma$. The following is a simple example that underline this fact: let $a\in(-1,+\infty)$, $\varepsilon>0$, and let $\overline u_\varepsilon(y)$ be an even solution in $B_1^+$ in the only vertical variable $y\in[0,1)$ of
\begin{equation}
\begin{cases}
-\left(\rho_\varepsilon^a(\cdot)\overline u_\varepsilon'(\cdot)\right)'(y)=\rho_\varepsilon^a(y) &\mathrm{in \ } (0,1)\\
\overline u_\varepsilon(0)=1, \qquad \varepsilon^a \overline u'_\varepsilon(0)=0.
\end{cases}
\end{equation}
Solving the ODE, one obtains that the solution has the following form
$$\overline u_\varepsilon(y)=-\int_0^y\rho_\varepsilon^{-a}(s)\left(\int_0^s\rho_\varepsilon^a(t)\mathrm{d}t\right)\mathrm{d}s \ + \ 1,$$
with derivatives
$$\overline u'_\varepsilon(y)=-\rho_\varepsilon^{-a}(y)\int_0^y\rho_\varepsilon^a(t)\mathrm{d}t,$$
and
\begin{eqnarray}
\overline u''_\varepsilon(y)&=&a\frac{y\int_0^y\rho_\varepsilon^a(t)\mathrm{d}t}{(\varepsilon^2+y^2)^{\frac{a}{2}+1}}-1\nonumber\\
&=&a\frac{y/\varepsilon\int_0^{y/\varepsilon}(1+t^2)^{\frac{a}{2}}\mathrm{d}t}{(1+(y/\varepsilon)^2)^{\frac{a}{2}+1}}-1,\nonumber
\end{eqnarray}
which does not converge uniformly on $[0,1)$ since $u''_\varepsilon(0)=-1$ and $u''_\varepsilon(1/2)\to\frac{a}{a+1}-1$, but does converge uniformly to $\frac{a}{a+1}-1$ on compact sets of the form $[g(\varepsilon),b]$, with $b<1$ and $\varepsilon=o(g(\varepsilon))$ as $\varepsilon\to0$ (for example $g(\varepsilon)=\sqrt{\varepsilon}$).
\end{remark}

\section{Local regularity for energy solutions when $\varepsilon=0$}
In this section we show how our results of uniform local bounds in $C^{0,\alpha}$ and $C^{1,\alpha}$ spaces apply in order to ensure regularity for energy solutions to \eqref{La} and \eqref{Lafks}. To this end, we will show that all energy solutions can be approximated by solutions of the regularized problems, thus extending Lemma \ref{Energy->Good} to the non homogeneous case. the main results of this section are the following
\begin{Theorem}\label{regularityevenodd}
Let $a\in(-1,+\infty)$ and let $u\in H^{1,a}(B_1)$ be an even in $y$ energy solution to
\begin{equation*}
-\mathcal L_au=|y|^af\qquad\mathrm{in \ }B_1,
\end{equation*}
with $f\in L^p(B_1,|y|^a\mathrm{d}z)$. Then
\begin{itemize}
\item[i)]  If $A$ is continuous, $\alpha\in(0,1)\cap(0,2-\frac{n+1+a^+}{p}]$, $p>\frac{n+1+a^+}{2}$, $\beta>1$ and $r\in(0,1)$ one has: there exists a constant $c>0$ such that
$$\|u\|_{C^{0,\alpha}(B_r)}\leq c\left(\|u\|_{L^\beta(B_1,|y|^a\mathrm{d}z)}+\|f\|_{L^p(B_1,|y|^a\mathrm{d}z)}\right).$$
\item[ii)] If $A$ is $\alpha$-H\"older continuous with $\alpha\in(0,1-\frac{n+1+a^+}{p}]$, $p>n+1+a^+$, $\beta>1$ and $r\in(0,1)$ one has: there exists a constant $c>0$ such that
$$\|u\|_{C^{1,\alpha}(B_r)}\leq c\left(\|u\|_{L^\beta(B_1,|y|^a\mathrm{d}z)}+\|f\|_{L^p(B_1,|y|^a\mathrm{d}z)}\right).$$
\end{itemize}
\end{Theorem}
\begin{Theorem}\label{regularityevenodd1}
Let $a\in(-1,+\infty)$ and let $u\in H^{1,a}(B_1)$ be an even in $y$ energy solution to
\begin{equation*}
-\mathcal L_au=\mathrm{div}\left(|y|^aF\right)\qquad\mathrm{in \ }B_1,
\end{equation*}
with $F=(f_1,...,f_{n+1})$. Then
\begin{itemize}
\item[i)] If $A$ is continuous, $\alpha\in(0,1-\frac{n+1+a^+}{p}]$, $F\in L^p(B_1,|y|^a\mathrm{d}z)$ with $p>n+1+a^+$, $\beta>1$ and $r\in(0,1)$ one has: there exists a constant $c>0$ such that
$$\|u\|_{C^{0,\alpha}(B_r)}\leq c\left(\|u\|_{L^\beta(B_1,|y|^a\mathrm{d}z)}+\|F\|_{L^p(B_1,|y|^a\mathrm{d}z)}\right).$$
\item[ii)] If $A$ is $\alpha$-H\"older continuous and $F\in C^{0,\alpha}(B_1)$ with $\alpha\in(0,1)$, $f_{n+1}(x,0)=f^y(x,0)=0$, $\beta>1$ and $r\in(0,1)$ one has that there exists a constant $c>0$ such that
$$\|u\|_{C^{1,\alpha}(B_r)}\leq c\left(\|u\|_{L^\beta(B_1,|y|^a\mathrm{d}z)}+\|F\|_{C^{0,\alpha}(B_1)}\right).$$
\end{itemize}
\end{Theorem}

\begin{proof}[Proof of Theorems \ref{regularityevenodd} and \ref{regularityevenodd1}]
As we have already remarked, the technical Lemmas in Section \ref{secs2} and the following Proposition \ref{particular} are stated when $A=\mathbb I$, but they hold true also for general uniformly elliptic matrixes with the properties stated in the introduction. In order to prove the Theorems, our strategy is to apply the uniform-in-$\varepsilon$ bounds of Sections \ref{sec:UniformHolder} and \ref{sec:UniformC1} to a family of solutions to the regularized problems. We first note that, for homogenous solutions, the proof of the Theorems can be performed by direct application of Theorems \ref{holdereven} and \ref{C1alphaeven} to the family of solutions to the approximating problems given by Lemma \ref{Energy->Good}. In order to treat the inhomogenous problems, we need at first to suitably adjust the right hand sides, as done in the following elementary Proposition.

\begin{Proposition}\label{particular}
Let $a\in\mathbb{R}$. Let us consider either the problem \eqref{La} with $f$ satisfying Assumption (Hf), or problem \eqref{Lafks} with $F$ satisfying Assumption (HF).
Then, there exist two families of functions $\{u_\varepsilon\}$ and $\{f_\varepsilon\}$ (respectively $\{F_\varepsilon\}$)
such that the assumptions of Lemma \ref{Good->Energy}  hold true.
\end{Proposition}
\proof
Without loss of generality we will consider $A=\mathbb I$. Let us consider in $B_1$ the family of functions $\{f_\varepsilon\}$ for $\varepsilon>0$ such that
$$f_\varepsilon(z):=\begin{cases}
f(z) & \mathrm{if} \ a\leq0,\\
f(z)\left(\frac{|y|^a}{\rho_\varepsilon^a(y)}\right)^{1/p} & \mathrm{if} \ a>0.
\end{cases}$$
Using the monotone convergence theorem,  is immediate to check that 
$$\|f_\varepsilon\|_{L^p(B_1,\rho_\varepsilon^a(y)\mathrm{d}z)}\leq c\qquad\mathrm{and}\qquad f_\varepsilon\to f\quad\mathrm{in \ } L^p(B_1,\mathrm{d}z).$$

For any fixed $\varepsilon>0$, let us consider the unique solution $u_\varepsilon\in H^1_0(B_1,\rho_\varepsilon^a(y)\mathrm{d}z)$ to
\begin{equation*}
\begin{cases}
-\mathrm{div}\left(\rho_\varepsilon^a\nabla u_\varepsilon\right)=\rho_\varepsilon^af_\varepsilon & \mathrm{in} \ B_1\\
u_\varepsilon=0 & \mathrm{on} \ \partial B_1.
\end{cases}
\end{equation*}
By testing equation with $u_\varepsilon$, using the Sobolev embeddings of Section \ref{sobo} we deduce 
$$\|u_\varepsilon\|_{H^1_0(B_1,\rho_\varepsilon^a(y)\mathrm{d}z)}\leq c\|f_\varepsilon\|_{L^p(B_1,\rho_\varepsilon^a(y)\mathrm{d}z)},$$
for some a positive constant, independent of $\varepsilon$. 
Similarily, let us consider in $B_1$ the family of fields $\{F_\varepsilon\}$ for $\varepsilon>0$ such that
$$F_\varepsilon(z):=\begin{cases}
F(z) & \mathrm{if} \ a\geq0,\\
F(z)\left(\frac{|y|^a}{\rho_\varepsilon^a(y)}\right)^{1/p} & \mathrm{if} \ a<0.
\end{cases}$$
It is easy to see that (obvious for $a\geq0$)
$$\|F_\varepsilon\|_{L^p(B_1,\rho_\varepsilon^a(y)\mathrm{d}z)}\leq c\qquad\mathrm{and}\qquad F_\varepsilon\to F\quad\mathrm{in \ } L^p_{\mathrm{loc}}(B_1\setminus\Sigma).$$
For any fixed $\varepsilon>0$, let us consider the unique solution $u_\varepsilon\in H^1_0(B_1,\rho_\varepsilon^a(y)\mathrm{d}z)$ to
\begin{equation*}
\begin{cases}
-\mathrm{div}\left(\rho_\varepsilon^a\nabla u_\varepsilon\right)=\mathrm{div}\left(\rho_\varepsilon^aF_\varepsilon\right) & \mathrm{in} \ B_1\\
u_\varepsilon=0 & \mathrm{on} \ \partial B_1.
\end{cases}
\end{equation*}
Testing the equation with $u_\varepsilon$ one easily obtains
$$\|u_\varepsilon\|_{H^1_0(B_1,\rho_\varepsilon^a(y)\mathrm{d}z)}\leq c\|F_\varepsilon\|_{L^p(B_1,\rho_\varepsilon^a(y)\mathrm{d}z)}.$$

\end{proof}

\begin{proof}[End of the proof of Theorems \ref{regularityevenodd} and \ref{regularityevenodd1}]
Of course, every solution $u$ to the inhomogeneous problem admits an expression as the sum
$u=\tilde u+\overline u$, where $\overline u$ is the $\mathcal L_a$-harmonic extension of $u$ on the unit ball, while $\tilde u$ solves the inhomogeneous equation with zero boundary trace:

\begin{equation*}
\begin{cases}
-\mathcal L_a\tilde u=|y|^af & \mathrm{(respectively,} =\mathrm{div}\left(|y|^aF_\varepsilon\right),  \mathrm{in}  \;B_1,\\
\tilde u\in H^{1,a}_0(B_1).
\end{cases}
\end{equation*}
Now we solve the regularized problems for the adjusted forcing terms as in Proposition \ref{particular}: 

\begin{equation*}
\begin{cases}
-\mathrm{div}\left(\rho_\varepsilon^aA\nabla \tilde u_\varepsilon\right)=\rho_\varepsilon^af_\varepsilon & \mathrm{(respectively,} =\mathrm{div}\left(\rho_\varepsilon^aF_\varepsilon\right),  \mathrm{in}  \;B_1,\\
\tilde u_\varepsilon\in H^{1}_0(B_1, \rho_\varepsilon^a\mathrm{d}z)\,,
\end{cases}
\end{equation*}
and we know from the same Proposition that Lemma \ref{Good->Energy}  is applicable. Applying once again the uniform-in-$\varepsilon$ bounds of Sections \ref{sec:UniformHolder} and \ref{sec:UniformC1} to a family of solutions to the regularized problems we easily obtain the desired result.
\end{proof}

\section{Further regularity in the limit case}\label{sec:furtherregularity}
In this section we are going to develop a Schauder theory for even energy solutions to \eqref{La}. We deal here  with the case $A=\mathbb I$,  because we need to differentiate the equation itself. This restriction can be replaced with suitable smoothness conditions on $A$.
\subsection{Some notable operators}
Let us define for $a\in\mathbb{R}$ and a suitable function $u$, the "weighted derivative"
$$\partial_y^au=|y|^a\partial_yu.$$
Moreover, for our purpose, it is useful to define the following operators:
\begin{equation*}\label{GF_a}
\mathcal Gu:=\frac{\partial_yu}{y}\qquad\mathrm{and}\qquad \mathcal F_au:=\partial_y^{-a}\partial_y^au.
\end{equation*}
Let us also remark the following formal expression for the second partial derivarive
\begin{equation}\label{secondyy}
\partial_{yy}^2u=\mathcal F_au-a\mathcal Gu.
\end{equation}

We are going to establish some useful regularity facts related with the weighted derivatives and the operators $\mathcal G$ and $\mathcal F_a$. At first, we need to extend Lemma \ref{rho-1} to the inhomogeneous problems. This can be easily achieved by the same approximation argument already introduced in the proof of Theorem \ref{regularityevenodd} of last section.
\begin{Lemma}\label{covariant}
There hold the following two points.
\begin{itemize}
\item[1)] Let $a\in(-1,+\infty)$. Let $u\in H^{1,a}(B_1)$ be an energy even solution to \eqref{La} in $B_1$ with $f\in L^p(B_1,|y|^a\mathrm{d}z)$ and $p\geq 2$. Then, fixing any $0<r<1$, the function $\partial^a_yu=|y|^a\partial_yu$ belongs to $H^{1,-a}(B_r)$ and is an odd energy solution to 
\begin{equation}\label{L-aDa}
-\mathcal L_{-a}\partial_y^au=\partial_yf=\mathrm{div}(fe_y)=|y|^{-a}\partial^a_yf\qquad\mathrm{in \ }B_r.
\end{equation}
\item[2)] Let $a\in(-\infty,1)$. Let $u\in H^{1,a}(B_1)$ be an energy odd solution to \eqref{La} in $B_1$ with
$$\begin{cases}
f\in L^p(B_1,|y|^a\mathrm{d}z), \ p\geq2 &\mathrm{if} \ a\in(-1,1),\\
|y|^{a/2}f\in L^p(B_1), \ p\geq2 &\mathrm{if} \ a\in(-\infty,-1].
\end{cases}$$
Then, fixing any $0<r<1$, the function $\partial^a_yu=|y|^a\partial_yu$ belongs to $H^{1,-a}(B_r)$ and is an even energy solution to \eqref{L-aDa} in $B_r$.
\end{itemize}
\end{Lemma}
\proof

We proceed as in the proof of Theorem \ref{regularityevenodd}, writing $u=\tilde u+\overline u$, where $\overline u$ solves the homogenous problem and $\tilde u$ has zero boundary trace.  Lemma \ref{rho-1} is applicable to $\overline u$, so we are left with $\tilde u$. First of all, by Proposition \ref{particular} we have the existence of two sequences $\{\tilde u_\varepsilon\}$ and $\{f_\varepsilon\}$ such that Lemma \ref{Energy->Good} applies to the solutions of the homogenous problem, and
having limit $\tilde u$. Next, we consider  $\tilde v_\varepsilon=\rho_\varepsilon^a\partial_y\tilde u_\varepsilon$ which converges a.e. in $B_1$ to $\tilde v=|y|^a\partial_y\tilde u$. Moreover
$$-\mathrm{div}\left(\rho_\varepsilon^{-a}\nabla\tilde v_\varepsilon\right)=\partial_yf_\varepsilon\qquad\mathrm{in \ }B_1,$$
with
$$\int_{B_1}\rho_\varepsilon^{-a}\left|\frac{f_\varepsilon}{\rho_\varepsilon^{-a}}\right|^2=\int_{B_1}\rho_\varepsilon^{a}|f_\varepsilon|^2\leq c,$$
and
$$\int_{B_1}\rho_\varepsilon^{-a}\tilde v_\varepsilon^2=\int_{B_1}\rho_\varepsilon^{a}|\partial_y\tilde u_\varepsilon|^2\leq\int_{B_1}\rho_\varepsilon^{a}|\nabla\tilde u_\varepsilon|^2\leq c.$$
In order to complete the proof we apply again  Lemma \ref{Energy->Good} to the $\tilde v_\varepsilon$'s.
\endproof

By applying parts 1) and 2) consecutively,   we easily obtain the following Lemma.
\begin{Lemma}
Let $a\in(-1,+\infty)$ and let $u\in H^{1,a}(B_1)$ be an energy even solution to \eqref{La} in $B_1$ with $f\in L^p(B_1,|y|^a\mathrm{d}z)$ with $p\geq2$ and such that $\partial_yf\in L^q(B_1,|y|^a\mathrm{d}z)$ with $q\geq2$. Then, for any $0<r<1$, $\mathcal F_au=\partial_y^{-a}\partial_y^au\in H^{1,a}(B_r)$ is an even in $y$ energy solution to
\begin{equation}\label{LaFa}
-\mathcal L_a\mathcal F_au=\mathrm{div}\left(|y|^a\partial_yfe_y\right)=|y|^a\mathcal F_af\qquad\mathrm{in \ }B_r.
\end{equation}
\end{Lemma}

Now turn to the similar property for the operator $\mathcal G$

\begin{Lemma}\label{2+a}
Let $a\in(-1,+\infty)$ and let $u\in H^{1,a}(B_1)$ be an energy even solution to \eqref{La} in $B_1$ with $f\in L^p(B_1,|y|^a\mathrm{d}z)$ with $p\geq2$ and such that $\mathcal Gf\in L^q(B_1,|y|^{2+a}\mathrm{d}z)$ with $q\geq(2^*(2+a))'$. Then, for any $0<r<1$, $\mathcal Gu=y^{-1}\partial_yu\in H^{1,2+a}(B_r)$ is an even in $y$ energy solution to
\begin{equation}\label{L2+aG}
-\mathcal L_{2+a}\mathcal Gu=|y|^{2+a}\left(y^{-1}\partial_yf\right)=|y|^{2+a}\mathcal Gf\qquad\mathrm{in \ }B_r.
\end{equation}
\end{Lemma}
\proof
By Lemma \ref{covariant} part $1)$, fixed $0<R<1$ we have $\partial_y^au\in H^{1,-a}(B_R)$ odd energy solution to \eqref{L-aDa} in $B_R$. Nevertheless, also $y|y|^a\in H^{1,-a}(B_R)$ is $\mathcal L_a$-harmonic and odd. Considering $b=-a\in(-\infty,1)$, one can apply Proposition 2.10 in \cite{SirTerVit2} obtaining the thesis; that is $\mathcal Gu=y^{-1}\partial_yu\in H^{1,2+a}(B_r)$ and solves \eqref{L2+aG} in $B_r$, for $r<R$.

\endproof

\subsection{Schauder estimates}

As the operator $\mathcal L_a$ commutes with derivations with respect to the $x$ variables,  our first result promptly follows from point $ii)$ of Theorem \ref{regularityevenodd1}, via an inductive argument. Thus we prove: 

\begin{Lemma}\label{Schauderx}
Let $a\in(-1,+\infty)$, $k\in\mathbb{N}\cup\{0\}$ and $f\in C^{k,\alpha}(B_1)$ with $\alpha\in(0,1)$ and even in $y$. Let also $u\in H^{1,a}(B_1)$ be an even energy solution to \eqref{La}. Then, for any partial derivative of order $k+1$ in the only $x_i$ variables $\partial^{k+1}_{j}u\in C^{1,\alpha}_{\mathrm{loc}}(B_1)$.
\end{Lemma}


In order to gain the Schauder estimates also with respect to the $y$ variable, we need to exploit the regularity features of the operators $\mathcal G$ and $\mathcal F_a$.
\begin{Lemma}\label{yderivativeseven}
Let $a\in(-1,+\infty)$ and $f\in C^{0,\alpha}(B_1)$ with $\alpha\in(0,1)$ and even in $y$. Let also $u\in H^{1,a}(B_1)$ be an even energy solution to \eqref{La}. Then, $\mathcal Gu$, $\mathcal F_au$ and $\partial^2_{yy}u$ belong to $C^{0,\alpha}_{\mathrm{loc}}(B_1)$.
\end{Lemma}
\proof
First we remark that by Lemma \ref{Schauderx}, $\Delta_xu\in C^{0,\alpha}_{\mathrm{loc}}(B_1)$ (the Laplacian in the $x_i$ variables). Recalling \eqref{secondyy} and \eqref{La}, we can express $\mathcal F_au$ pointwisely as
\begin{equation*}
\mathcal F_au(x,y)=\partial^2_{yy}u(x,y)+a\mathcal Gu(x,y)=-\Delta_xu(x,y)-f(x,y):=g(x,y)\in C^{0,\alpha}_{\mathrm{loc}}(B_1).
\end{equation*}
Hence, since $u$ is even in $y$, we can restrict it only in $B_1^+$ and express
$$y^{a}\partial_yu(x,y)=\int_0^{y}t^ag(x,t)\mathrm{d}t,$$
that is,
\begin{equation}\label{integraG}
\mathcal Gu(x,y)=\frac{1}{y^{1+a}}\int_0^{y}t^a(g(x,t)-g(x,0))\mathrm{d}t \ + \ \frac{g(x,0)}{1+a}.
\end{equation}
H\"older continuity of $\mathcal Gu$ with respect to the $x$-variable being trivial,
we focus on the variations with respect to $y$. Let us fix $\varepsilon\in(0,1)$ and $r\in(0,1)$ and consider the two sets $\{(y_1,y_2): \ r>y_2\geq y_1\geq0, \ \mathrm{and} \ y_2-y_1\geq\varepsilon y_2\}$ and $\{(y_1,y_2): \ r>y_2\geq y_1\geq0, \ \mathrm{and} \ y_2-y_1<\varepsilon y_2\}$. Taking $(y_1,y_2)$ in the first set, then
\begin{eqnarray*}
\frac{|\mathcal G(x,y_1)-\mathcal G(x,y_2)|}{(y_2-y_1)^\alpha}&\leq&\frac{1}{\varepsilon^2}\sum_{i=1}^2\frac{1}{y_i^{1+a+\alpha}}\int_0^{y_i}t^a(g(x,t)-g(x,0))\mathrm{d}t\\
&\leq& c[g]_{C^{0,\alpha}(B_r)},
\end{eqnarray*}
Taking $(y_1,y_2)$ in the second set, we have the existence of $y_1<\xi<y_2$ such that
\begin{eqnarray*}
|\mathcal G(x,y_1)-\mathcal G(x,y_2)|&\leq&\left|\frac{1}{y_2^{1+a}}\int_{y_1}^{y_2}t^a(g(x,t)-g(x,0))\mathrm{d}t\right|\\
&&+\left|\left(\frac{1}{y_2^{1+a}}-\frac{1}{y_1^{1+a}}\right)\int_{0}^{y_1}t^a(g(x,t)-g(x,0))\mathrm{d}t\right|\\
&=&\left|\frac{(y_2-y_1)\xi^a(g(x,\xi)-g(x,0))}{y_2^{1+a}}\right|\\
&&+\left|\left[1-\left(1-\frac{y_2-y_1}{y_2}\right)^{1+a}\right]\frac{1}{y_1^{1+a}}\int_{0}^{y_1}t^a(g(x,t)-g(x,0))\mathrm{d}t\right|\\
&\leq&c\frac{(y_2-y_1)\xi^{a+\alpha}}{y_2^{1+a}}+c\frac{(y_2-y_1)y_1^\alpha}{y_2}.
\end{eqnarray*}
Eventually, using the fact that $y_2-y_1<\varepsilon y_2$, we have
$$1-\varepsilon\leq\frac{y_1}{y_2}\leq\frac{\xi}{y_2}\leq 1,$$
and so
$$\frac{|\mathcal G(x,y_1)-\mathcal G(x,y_2)|}{(y_2-y_1)^\alpha}\leq c\left(\frac{\xi}{y_2}\right)^{a+\alpha}+c\left(\frac{y_1}{y_2}\right)^\alpha\leq c.$$
Eventually, also $\partial^2_{yy}u\in C^{0,\alpha}_{\mathrm{loc}}(B_1)$ by \eqref{secondyy}.
\endproof

\begin{Lemma}\label{yderivativesodd}
Let $a\in(-1,+\infty)$ and $f\in C^{1,\alpha}(B_1)$ with $\alpha\in(0,1)$ and even in $y$. Let also $u\in H^{1,a}(B_1)$ be an even energy solution to \eqref{La}. Then, $\mathcal Gu$, $\mathcal F_au$ and $\partial^2_{yy}u$ belong to $C^{1,\alpha}_{\mathrm{loc}}(B_1)$.
\end{Lemma}
\proof
One can easily see that $\mathcal F_au\in H^{1,a}(B_1)$ and $\mathcal Gu\in H^{2+a}(B_1)$ are even energy solutions to \eqref{LaFa} and \eqref{L2+aG}. Hence, $\mathcal F_au$ belongs to $C^{1,\alpha}_{\mathrm{loc}}(B_1)$ using point $ii)$ in Theorem \ref{regularityevenodd1}, having
$$\partial_yf\in C^{0,\alpha}(B_1)$$
and using the fact that $\partial_yf(x,0)=0$.\\\\
Moreover, the equation \eqref{L2+aG} satisfied by $\mathcal Gu$ can be seen as
$$-\mathcal L_{2+a}\mathcal Gu=|y|^{2+a}\mathcal Gf=\mathrm{div}\left(|y|^{2+a}\Phi e_y\right),$$
where
$$\Phi(x,y)=\frac{1}{|y|^{2+a}}\int_0^yt^{1+a}\partial_yf(x,t)\mathrm{d}t.$$
The previous expression is analogous to the one in \eqref{integraG}, and following the same passages in the proof of Lemma \ref{yderivativeseven}, we obtain that the field $\Phi$ belongs to $C^{0,\alpha}(B_1)$. Moreover, since
$$\lim_{y\to0}\Phi(x,y)=\partial_yf(x,0)=0,$$
we can apply point $ii)$ in Theorem \ref{regularityevenodd1} again, obtaining that $\mathcal Gu$ belongs to $C^{1,\alpha}_{\mathrm{loc}}(B_1)$. Eventually \eqref{secondyy} gives the same regularity for the second partial derivative $\partial^2_{yy}u$.
\endproof

\begin{Lemma}\label{Forcing}
Let $a\in\mathbb{R}$, $k\in\mathbb{N}\cup\{0\}$, $f\in C^{k+2,\alpha}(B_1)$ even in $y$ with $\alpha\in[0,1]$. Then $\mathcal Gf$ and $\mathcal F_af$ belong  to  $C^{k,\alpha}(B_1)$.
\end{Lemma}
\proof
The partial derivative $\partial_yf$ belongs to $C^{k+1,\alpha}(B_1)$ and it is odd in $y$; that is, $\partial_yf(x,0)=0$ in $B_1$. By the symmetry, we can consider only what happens in $B_1^+$. Hence, by a Taylor expansion with integral remainder
\begin{eqnarray*}
\partial_yf(x,y)=\partial_yf(x,0)+\int_0^y\partial_{yy}^2f(x,t)\mathrm{d}t=y\int_0^1\partial_{yy}^2f(x,sy)\mathrm{d}s;
\end{eqnarray*}
that is,
$$\mathcal Gf(x,y)=\int_0^1\partial_{yy}^2f(x,sy)\mathrm{d}s.$$
By the Leibniz integral rule one can commute the integration and any partial derivative of $\mathcal Gf$ up to order $k$, using the regularity of the function under the integral. Hence, $\mathcal Gf\in C^{k}(B_1)$. Moreover, any partial derivative $\partial_i^k$ of order $k$ is H\"older continuous, since
\begin{eqnarray*}
|\partial_i^k\mathcal G(x_1,y_1)-\partial_i^k\mathcal G(x_2,y_2)|&\leq&\int_0^1|\partial_i^k\partial_{yy}^2f(x_1,sy_1)-\partial_i^k\partial_{yy}^2f(x_2,sy_2)|\mathrm{d}s\\
&\leq&\int_0^1|(x_1,sy_1)-(x_2,sy_2)|^\alpha\mathrm{d}s\leq |(x_1,y_1)-(x_2,y_2)|^\alpha.
\end{eqnarray*}
Finally, \eqref{secondyy} ensures the same regularity for $\mathcal F_a$.
\endproof
\begin{Lemma}\label{Schaudery}
Let $a\in(-1,+\infty)$, $k\in\mathbb{N}\cup\{0\}$ and $f\in C^{k,\alpha}(B_1)$ for $\alpha\in(0,1)$ and even in $y$. Let also $u\in H^{1,a}(B_1)$ be an even energy solution to \eqref{La}. Then, the partial derivative of order $k+2$ in the only $y$ variable $\partial^{k+2}_{y}u\in C^{0,\alpha}_{\mathrm{loc}}(B_1)$.
\end{Lemma}
\proof
We argue by induction.\\\\
{\bf Step 1:} If $k=0$, the result directly follows from Lemma \ref{yderivativeseven}, whilest, if $k=1$, the result follows from Lemma \ref{yderivativesodd}.\\\\
{\bf Step 2:} Now we suppose the result true for a fixed $k\in\mathbb{N}\cup\{0\}$ and we prove it for $k+2$.\\
Since $f\in C^{k+2,\alpha}(B_1)$, then by Lemma \ref{Forcing},
$$\mathcal F_af \   \mathrm{and} \ \mathcal Gf \ \mathrm{belong \ to \ } C^{k,\alpha}(B_1).$$
Let us remark that by \eqref{LaFa} and \eqref{L2+aG}, then $\mathcal F_au$ and $\mathcal Gu$ are even energy solutions to
$$-\mathcal L_a\mathcal F_au=|y|^a\mathcal F_af,\qquad\mathrm{and}\qquad -\mathcal L_{2+a}\mathcal Gu=|y|^{2+a}\mathcal Gf;$$
that is, they satisfy the inductive hypothesis, and hence $\mathcal Gu$ and $\mathcal F_au$ own the partial derivatives of order $k+2$ in the only $y$ variable $\partial^{k+2}_y\mathcal Gu$ and $\partial^{k+2}_y\mathcal F_au$ belonging to $C^{0,\alpha}_{\mathrm{loc}}(B_1)$. Eventually, by \eqref{secondyy} we have
$$\partial_y^{k+4}u=\partial^{k+2}_y\left(\partial^2_{yy}u\right)=\partial^{k+2}_y\mathcal F_au-a\partial^{k+2}_y\mathcal Gu\in C^{0,\alpha}_{\mathrm{loc}}(B_1).$$
\endproof
Thanks to the previous Lemmas, we are now in a position to establish Schauder estimates for an even solution to \eqref{La}.
\begin{Theorem}\label{Schauder}
Let $a\in(-1,+\infty)$, $k\in\mathbb{N}\cup\{0\}$ and $f\in C^{k,\alpha}(B_1)$ for $\alpha\in(0,1)$ and even in $y$. Let also $u\in H^{1,a}(B_1)$ be an even energy solution to \eqref{La}. Then, $u\in C^{k+2,\alpha}_{\mathrm{loc}}(B_1)$. If moreover $f\in C^{\infty}(B_1)$, then, $u\in C^{\infty}(B_1)$.
\end{Theorem}
\proof
It is a direct consequence of the previous Lemmas: indeed, let us consider a general partial derivative of order $k+2$ in any variable $\partial^{k+2}u$. So, let $0\leq j\leq k+2$ the number of derivatives in the variables $x_i$. If $j=k+2$ or $j=k+1$, then $\partial^{k+2}u=\partial_l(\partial^{k+1}_{x_i}u)$ for a certain $l\in\{1,...,n+1\}$. Then by Lemma \ref{Schauderx}, the partial derivative $\partial^{k+1}_{x_i}u$ of order $k+1$ is an even energy solution to the problem (for any $r\in(0,1)$)
$$-\mathcal L_a(\partial^{k+1}_{x_i}u)=|y|^a\partial^{k+1}_{x_i}f\qquad\mathrm{in} \ B_r.$$
Hence, since $\partial^{k+1}_{x_i}u\in C^{1,\alpha}_{\mathrm{loc}}(B_1)$, the result is proved.\\\\
Let now consider the case $j\leq k$. By Lemma \ref{Schauderx}, the partial derivative $\partial^j_{x_i}u$ of order $j$ is an even energy solution to the problem (for any $r\in(0,1)$)
$$-\mathcal L_a(\partial^j_{x_i}u)=|y|^a\partial^j_{x_i}f\qquad\mathrm{in} \ B_r.$$
Since $j\leq k$, then $\partial^j_{x_i}f\in C^{k-j,\alpha}(B_1)$ and $\partial^j_{x_i}u\in C^{1,\alpha}_{\mathrm{loc}}(B_1)$. Eventually it remains to perform the $y$ derivatives applying Lemma \ref{Schaudery}, obtaining that
$$\partial^{k+2}u=\partial^{k+2-j}_y(\partial^j_{x_i}u)\in C^{0,\alpha}_{\mathrm{loc}}(B_1).$$
\endproof

As an immediate consequence of Theorem \ref{Schauder} , by a simple inductive argument, we can deal with nonlinear equations in the form:
\begin{equation}\label{Lau}
-\mathcal L_au=|y|^af(u)\qquad\mathrm{in} \ B_1.
\end{equation}

\begin{Corollary}
Let $a\in(-1,+\infty)$, $k\in\mathbb{N}\cup\{0\}$ and $f\in C^{k,\alpha}_{\mathrm{loc}}(\mathbb{R})$ and even for $\alpha\in(0,1)$. Let also $u\in H^{1,a}(B_1)\cap L^{\infty}(B_1)$ be an even energy solution to \eqref{Lau}. Then, $u\in C^{k+2,\alpha}_{\mathrm{loc}}(B_1)$. Moreover, if $f\in C^{\infty}(\mathbb{R})$, then $u\in C^{\infty}(B_1)$.
\end{Corollary}

Note that the assumption $u\in L^\infty(B_1)$ may be removed for energy solutions under proper growth conditions on $f$, via suitable adaptation of Brezis-Kato arguments, for example $|f(s)|\leq C(1+|s|^{2^*(a)-1})$ do (see \cite[Appendix B]{Struwe}).

Eventually we remark that our Theorem \ref{Schauder} is providing local $C^\infty$-regularity for extensions of $s$-harmonic functions (which correspond to homogeneous Neumann boundary condition) in any variable (also the extension variable $y$) up to the characteristic manifold $\Sigma$.

\section{Inhomogenous Neumann boundary conditions}\label{sec:neumann}
Now we turn to the Neumann  boundary value problems associated with the equation
\begin{equation}\label{eq:neumannbce}
\begin{cases}
-\mathrm{div}\left(\rho_\varepsilon^a\nabla u_\varepsilon\right)=0\;, & \textrm{in}\;B_1^+\\
-\rho_\varepsilon^a\partial_yu_\varepsilon= f_\varepsilon &  \textrm{on}\; \partial^0 B_1^+\;,\\
\end{cases}
\end{equation}
looking, as usual, for uniform-in-$\varepsilon$ bounds. When $\varepsilon=0$ the equation reads

\begin{equation}\label{eq:neumannbc}
\begin{cases}
-\mathrm{div}\left(y^a\nabla u\right)=0\;, & \textrm{in}\;B_1^+\\
-\lim_{y\to 0^+}y^a\partial_yu= f &  \textrm{on}\; \partial^0 B_1^+\\
\end{cases}
\end{equation}
 and has to be intended in a variational sense $u\in H^{1,a}(B^+_1)$ and:
\[
\int_{B_1^+} y^a\nabla u\cdot\nabla \phi=\int_{\partial^0 B_1^+} f\phi\;,\qquad \forall\;\phi\in C^\infty_c(B_1^+\cup\partial^0B_1^+)\;.
\]

At first, we stress that,  whenever $a\geq 1$, by Proposition \ref{prop:super}, the characteristic hyperspace $\Sigma$ has vanishing capacity with respect to the Dirichlet energy. Thus, in order to properly deal with energy solutions we are bound to restrict our analysis to the case $a\in(-1,1)$. 
 Moreover, when $a\in(-1,0)$ the approximating problems have less regularising power than the original ones.  This affects the range of the optimal regularity estimates, when we wish tho make them stable with respect to $\varepsilon$. Let us define
\begin{equation*}\label{p*a}
p^*(t)=\frac{2n}{n-1+t}.
\end{equation*}
In the range $a\in(-1,1)$, the following trace Sobolev embeddings hold.

\begin{Theorem}\label{teo:tracesobemb}
Let $a\in(-1,1)$, $n\geq2$, $\varepsilon_0>0$. There exists a constant, independent of $\varepsilon\in[0,\varepsilon_0]$, such that  for every $u\in C^\infty_c(B_1^+\cup\partial^0B_1^+)$,there holds
\begin{equation*}\label{sobo>-1}
\left(\int_{\partial^0 B_1^+ }|u|^{p^*(a^+)}\right)^{2/p^*(a^+)}\leq c(n,a)\int_{B_1^+}\rho_\varepsilon^a|\nabla u|^2\;.
\end{equation*}
When $n=1$ and $a^+>0$ the same inequality holds. When $n=1$ and $a^+=0$ then the embedding holds in $L^p(\partial^0B_1^+)$ for every $p\geq 1$. Moreover, when $\varepsilon=0$ and $a\in(-1,0)$, we have
\[
H^{1,a}(B_1^+)	\hookrightarrow L^{p^*(a)}(\partial^0 B_1^+)\;.
\]

\end{Theorem}

Exploiting these embeddings, is rather standard to adapt Moser's iterative technique to obtain $L^\infty$-bounds. Next, one can easily modify the proof of Theorem \ref{holdereven} to obtain

\begin{Theorem}\label{teo:holderneumann}
Let $a\in(-1,1)$,  $p>\frac{n}{1-a^+}$, and $\alpha\in(0,1-a^+-\frac{n}{p}]$, $r\in(0,1)$, $\beta>1$. As $\varepsilon\to0$ let $\{u_\varepsilon\}$ be a family of solutions to \eqref{eq:neumannbce}. Then, there exists a positive constant depending on $a$, $n$, $\beta$, $p$, $\alpha$ and $r$ such that
$$\|u_\varepsilon\|_{C^{0,\alpha}(B_r^+)}\leq c\left(\|u_\varepsilon\|_{L^\beta(B_1^+,\rho_\varepsilon^a(y)\mathrm{d}z)}+ \|f_\varepsilon\|_{L^p(\partial^0 B_1^+)}\right).$$
In addition, when $\varepsilon=0$, $p>\frac{n}{1-a}$, and $\alpha\in(0,1-a-\frac{n}{p}]\cap(0,1)$, the same inequality holds for solutions to \eqref{eq:neumannbc}.
\end{Theorem}

\begin{proof}
We repeat all the steps of the proof of Theorem \ref{holdereven}, up to the point when equation satisfied by the blow up sequence is considered in \eqref{tildeLakw}, which now reads as
\begin{equation}\label{tildeLakw1}
\begin{cases}
-\mbox{div}\!\left(\tilde\rho_k\nabla w_k\right)(z)= 0\;\quad & y>-y_k/r_k\\
-\tilde \rho_k\partial_yw_k=
\nu_k^{-a}\frac{\eta(z_k)}{L_k}r_k^{1-\alpha} f_{\varepsilon_k}(x_k+r_kx), & y=-y_k/r_k
\end{cases}
\end{equation}
where $z_k=(x_k,y_k)$. As usual, we notice that, when $a\geq 0$, we have $\nu_k^{-a}\leq r_k^{-a}$, while, when $a<0$ the sequence $\nu_k^{-a}$ is bounded. With this remark one easily sees that the right hand side vanishes in $L^1_{\mathrm{loc}}$ and concludes as in the proof of Theorem \ref{holdereven}. Once proved the first part of the statement, when $\varepsilon=0$ we can approximate $f$ with a sequence smoothened ones, apply the first part and the derive a priori estimates arguing again by contradiction as in the proof of Theorem \ref{holdereven}. Notice that now, as $\varepsilon_k=0$, the boundary persists in the limit only when $y_k/r_k$ stays bounded and therefore $\nu_k\simeq r_k$. The second part of the statement then follows by observing that the right hand side of the boundary condition in \eqref{tildeLakw1} vanishes in $L^1_{\mathrm{loc}}$ also in this case.
\end{proof}

When $a\in(-1,0)$, higher regularity can be expected, even though not necessarily in an $\varepsilon$-stable manner. Indeed, the following  result holds true:

\begin{Theorem}\label{teo:C1alphaneumannunstable}
Let $a\in(-1,0)$, $p>\frac{n}{-a}$, $\alpha\in(0,-a-\frac{n}{p}]$, $r\in(0,1)$ and $\beta>1$. There is a positive constant depending on $a$, $n$, $\beta$, $p$, $\alpha$ and $r$ only such that every solution to \eqref{eq:neumannbc} satisfies:
$$\|u\|_{C^{1,\alpha}(B_r^+)}\leq c\left(\|u\|_{L^\beta(B_1^+,y^a\mathrm{d}z)}+ \|f\|_{L^p(\partial^0 B_1^+)}\right).$$
\end{Theorem}

Similarly, for H\"older continuous right hand sides, solutions are expected to be $C^{1,\alpha}$, up to $\alpha=-a$. Indeed, $y^{1-a}/(1-a)$ solves \eqref{eq:neumannbc} with $f\equiv 1$. Surprisingly enough, this property turns out to be $\varepsilon$-stable in some range of $\alpha$. Indeed, we have the following result.

\begin{Theorem}\label{teo:C1alphaneumannstable}
Let $a\in(-1,0)$, $\alpha\in(0,-a]$, $r\in(0,1)$ and $\beta>1$. As $\varepsilon\to0$ let $\{u_\varepsilon\}$ be a family of solutions to \eqref{eq:neumannbce}. Then, there is a positive constant depending on $a$, $n$, $\beta$, $p$, $\alpha$ and $r$ such that
$$\|u_\varepsilon\|_{C^{1,\alpha}(B_r^+)}\leq c\left(\|u_\varepsilon\|_{L^\beta(B_1^+,\rho_\varepsilon^a(y)\mathrm{d}z)}+ \|f_\varepsilon\|_{C^{0,\alpha}(\partial^0 B_1^+)}\right).$$
\end{Theorem}

\begin{proof}[Proof of Theorem \ref{teo:C1alphaneumannstable}]
At first, let us consider the special solutions of \eqref{eq:neumannbc} with $f\equiv 1$:
\[
\mu_\varepsilon^a(y)=\int_0^y \rho_\varepsilon^{-a}(s) \mathrm{d}s\;,
\]
and remark that they are uniformly bounded in the $C^{1,\alpha}$-norm, provided $\alpha\leq -a$ (note that $a\in(-1,0)$. Now we go back to the proof of Theorem \ref{C1alphaeven} and we slightly change the two blow up sequences: writing $\hat{z}_k=(\hat x_k,\hat y_k)$ and
\begin{equation*}
\begin{split}
v_k(z)=\frac{\eta(\hat{z}_k+r_kz)}{L_kr_k^{1+\alpha}}\left(u_k(\hat{z}_k+r_kz)-u_k(\hat{z}_k)-f_{\varepsilon_k}(\hat x_k)\mu_k(\hat y_k+r_ky)\right),\\
w_k(z)=\frac{\eta(\hat{z}_k)}{L_kr_k^{1+\alpha}}\left(u_k(\hat{z}_k+r_kz)-u_k(\hat{z}_k)-f_{\varepsilon_k}(\hat x_k)\mu_k(\hat y_k+r_ky)\right)\;.
\end{split}
\end{equation*}
In this way we have $\partial_y v_k(\hat x_k, -\hat y_k/r_k)=\partial_y w_k(\hat x_k, -\hat y_k/r_k)=0$, as for the  two original blow-up sequences under homogenous Neumann boundary conditions. These modified sequences  still fulfill all the properties needed to make it work the proof of Theorem \ref{C1alphaeven}, up to the point when, dealing with the boundary value problem satisfied by the $\overline w_k$s (see \eqref{eq:wk}), we find
\begin{equation}\label{eq:wkbc}
\begin{cases}
-\mbox{div}\!\left(\tilde
\rho_{k}\nabla \overline w_k\right)=\nu_k^{-a}\partial_y (\rho_{\varepsilon_k}^a\left(y_k+r_k\cdot\right))\partial_yw_k(0) \;,\qquad & y>-\hat y_k/r_k\\
-\tilde \rho_{k}\partial_y \overline w_k =\nu_k^{-a}\frac{\eta(\hat{z}_k)}{L_k}r_k^{-\alpha}\left(f_{\varepsilon_k}(\hat{x}_k+r_kx)-f_{\varepsilon_k}(\hat x_k)\right) & y=-\hat y_k/r_k\
\end{cases}
\end{equation}
One easily sees that the right hand side of the boundary condition converge uniformly to zero, since $a<0$, the $\nu_k$'s are bounded, and $f_{\varepsilon_k}$ are bounded in $C^{0,\alpha}$. As to the right hand side of the differential equation, note that $\partial_yw_k(\hat y_k/r_k)=\partial_yv_k(\hat y_k/r_k)$, and therefore we have $\partial_yw_k(\hat y_k/r_k)\equiv 0$. Thus,  the same reasoning as in the proof of Theorem \ref{C1alphaeven} apply to obtain the desired estimate.
\end{proof}

\begin{proof}[Proof of Theorem \ref{teo:C1alphaneumannunstable}] Here $\varepsilon=0$. At first, we can approximate the right hand $f$ with a sequence $f_k$ of H\"older continuous functions such that $f_k\to f$ in $L^p$ and apply Theorem \ref{teo:C1alphaneumannstable} for each fixed $k$ with $\varepsilon_k\equiv0$. Next we wish to show that the estimate holds for a constant $c$ independent of $k$. To this aim, we argue by contradiction and we follow once more the steps in the proof of Theorem \ref{C1alphaeven}, noticing that thanks to the fact that $a<0$ there holds $\partial_yw_k(\hat y_k/r_k)=\partial_yv_k(\hat y_k/r_k)=0$. In addition, as $\nu_k\simeq r_k$ whenever the boundary condition persists in the limit, the right hand side of  \eqref{eq:wkbc} vanishes in $L^1_{\mathrm{loc}}$. From this fact,  the same reasoning as in the proof of Theorem \ref{C1alphaeven} yield the desired estimate.
\end{proof}


\end{document}